\documentclass{amsart}
\usepackage{amsmath}
  \usepackage{graphics} 
  \usepackage{epsfig} 
\usepackage{graphicx}
\usepackage{epstopdf}
 \usepackage[colorlinks=true]{hyperref}
\hypersetup{urlcolor=blue, citecolor=red}

\newtheorem{theorem}{Theorem}[section]
\newtheorem{corollary}{Corollary}

\newtheorem{lemma}[theorem]{Lemma}
\newtheorem{proposition}{Proposition}

\theoremstyle{definition}
\newtheorem{definition}[theorem]{Definition}
\newtheorem{remark}{Remark}

\newcommand{\e}{\varepsilon}
\newcommand{\ds}{\displaystyle}

\def\XXint#1#2#3{{\setbox0=\hbox{$#1{#2#3}{\int}$ }
\vcenter{\hbox{$#2#3$ }}\kern-.6\wd0}}

\def\div{\,\mathrm{div}\,}

\title[OCP for Linear Elliptic Equations]
      {Optimal $L^2$-Control Problem in Coefficients\\ for a Linear Elliptic Equation}

\author[T.~Horsin and P.~I.~Kogut]{}


 \email{p.kogut@i.ua}
 \email{thierry.horsin@cnam.fr}

\thanks{ }

\begin{document}
\maketitle

\centerline{\scshape Thierry Horsin}
\medskip
{\footnotesize
 \centerline{ Conservatoire National des Arts et M\'{e}tiers,}
   \centerline{M2N, Case 2D 5000,}
   \centerline{292 rue Saint-Martin, 75003 Paris, France}
}

\medskip

\centerline{\scshape Peter I. Kogut }
\medskip
{\footnotesize
 \centerline{Department of Differential Equations,}
   \centerline{Dnipropetrovsk National University,}
   \centerline{ Gagarin av., 72, 49010 Dnipropetrovsk, Ukraine}
} 

\bigskip

 \centerline{(Communicated by the associate editor name)}

\begin{abstract}
In this paper we study an optimal control
problem (OCP) associated to a linear elliptic equation { on a bounded domain $\Omega$}. The
matrix-valued coefficients $A$ of such systems  is our control { in $\Omega$ and will be taken in}
$L^2(\Omega;\mathbb{R}^{N\times N})$ { which in particular may comprise som cases of unboundedness}.{ Concerning the boundary value problems associated to }the equations of this
type, one may face non-uniqueness of
weak solutions--- namely,  approximable solutions as well as another type of weak solutions  that can not be obtained through the $L^\infty$-approximation of matrix $A$.
Following the direct method in the calculus of variations, we
show that the given OCP is well-posed in the sense that it admits at least one solution.
At the same time, optimal solutions to such problem may have a singular character in the above sense.
In view of this, we indicate two types of optimal solutions to the above problem: the so-called variational and non-variational solutions, and show that some of that optimal solutions can be attainable by solutions of special optimal boundary control problems.
\end{abstract}

In this paper we deal with the following optimal control problem (OCP) in coefficients for a linear elliptic  equation
\begin{equation}\label{0.1}
\left\{
\begin{array}{c}
\ds\text{Minimize } I(A,y)=\left\|y-y_d\right\|^2_{L^2(\Omega)}
 + \int_\Omega\left(\nabla y, A^{sym}\nabla y\right)_{\mathbb{R}^N}\,dx\\
 \text{subject to the constraints }\\
-\div\big(A^{sym}\nabla y+A^{skew}\nabla y\big) = f\quad\text{ in }\Omega,\\
y=0\quad\text{ on }\partial\Omega\\
A\in \mathfrak{A}_{ad},
\end{array}
\right.
\end{equation}
where  $(A^{sym},\,A^{skew})\in L^\infty(\Omega;R^{N\times N})\times L^2(\Omega;R^{N\times N})$ are respectively the symmetric and antisymmetric part of the control $A$, $y_d\in L^2(\Omega)$ and $f\in H^{-1}(\Omega)$ are  given distributions, and $\mathfrak{A}_{Ad}$ denotes the class of admissible controls which will be precised later.

The characteristic feature of this problem is the fact that the skew-symmetric part of matrix $A(x)=[a_{ij}(x)]_{i,j=1,\dots,N}$  belongs to $L^2$-space (rather than $L^\infty)$. The existence, uniqueness, and variational properties of a  weak solution to \eqref{0.1}  { are }usually  drastically different from the corresponding properties of solutions to the elliptic equations with $L^\infty$-matrices in coefficients. In most of the cases, the situation can  deeply change for the matrices $A$ with unremovable singularity.
Typically, in such cases,  the above boundary value problem  may admit infinitely many weak solutions which can be divided into two classes: approximable and non-approximable solutions \cite{F_P,Zh_97,Zhik1_04}. A function $y=y(A)$ is called an approximable
solution to the boundary value problem in \eqref{0.1} if it can be attained  by  weak solutions to the similar boundary value problems with $L^\infty$-approximated matrix $A$. However, this type  does not exhaust all weak solutions to the above problem. There is another type of weak solutions, which cannot be approximated by  weak solutions of such regularized problems. Usually, these are called non-variational \cite{Zh_97,Zhik1_04}, singular \cite{Buttazzo_Kogut,Kogut1,Kogut2,Zuazua1}, pathological \cite{Maz,Serin} and others.

It may seem puzzling to consider, for an optimal control problem, a
state equation with singular matrix involved in the coefficients. Despite
this offhand abstract aspect of the problem, one should be aware that
singular equations appear naturally when considering optimal control
problems with a nonlinear state equation (see, for instance, \cite{Casas1} for
quasi-linear elliptic equations). Moreover, formal analysis in optimization are
well-known to state that optimal control problems and their adjoints  are
completely dual from each other through saddle points consideration
which also justifies the fact that one may be interested in dealing with
optimization of linear singular equations.

The aim of this work is to study the existence of optimal controls to the problem \eqref{0.1}, propose a scheme of their approximations, and discuss the optimality conditions of this problem.  Using the direct
method in the Calculus of Variations, we show in Section \ref{Sec 2} that the
original OCP admits in general a non-unique solution even if the
corresponding boundary value problem is ill-possed. This problem is
thus another example of the difference between well-posedness
of optimal control problems for systems with distributed parameters
and ill-posedness of boundary value problems for partial differential equations.

In Section \ref{Sec 4} we show that there are two types of optimal solutions: the so-called variational and non-variational solutions. By the first type we mean those optimal solutions which can be attained through the sequence of optimal solutions to regularized OCP for boundary value problem \eqref{0.1} with skew-symmetric parts of admissible controls
$A^{skew}_k\in L^\infty(\Omega;\mathbb{S}^N)$ such that $A^{skew}_k\rightarrow A^{skew}$ strongly in $L^2(\Omega;\mathbb{S}^N)$. We give the sufficient conditions which guarantee that the solutions to OCP \eqref{0.1} have a variational character.
The second type of optimal solutions is related to those which
cannot be attained by the above procedure. We discuss in Section~\ref{Sec A3a} the example of an optimal control problem in coefficients with non-variational optimal solution.  This stimulates us to develop
another approach of approximation for the considered optimal control problems.

In Section~\ref{Sec 33} we discuss optimality conditions for OCP \eqref{0.1}. In
spite of the fact that the corresponding Lagrange functional is, in general, not
G\^{a}teaux differentiable, we show that the optimality conditions
can be derived using the notion of quasi-adjoint state to the original problem
\cite{Serov}. As for a result,
this leads to an optimality system which contains the so-called extended values
of bilinear forms generated by $L^2$-skew-symmetric matrices.

In section \ref{Sec 5} we give a precise description of the class of admissible
controls $\mathfrak{A}_{ad}\subset L^2\big(\Omega;\mathbb{R}^{N\times N}\big)$
which guarantee that non-variational solutions can be attained through the
sequence of optimal solutions to OCPs in special perforated domains with
fictitious boundary controls on the boundary of holes. Namely, we consider the
following family of regularized OCPs
\begin{equation}\label{0.2}
\left\{
\begin{array}{c}
\ds\text{Minimize } I_\e(A,v,y):=\left\|y-y_d\right\|^2_{L^2(\Omega_\e)}+ \int_{\Omega_\e}\left(\nabla y, A^{sym}\nabla y\right)_{\mathbb{R}^N}\,dx\\
\ds  + \frac{1}{\e^\sigma}\|v\|^2_{H^{-\frac{1}{2}}(\Gamma_\e)}\\
 \text{subject to the constraints }\\
-\div\big(A^{sym}\nabla y+A^{skew}\nabla y\big) = f\quad\text{in }\ \Omega_\e,\\
y=0\text{ on }\partial\Omega,\quad
\partial y/ \partial \nu_{A}=v\ \text{on }\Gamma_\e,\\
\ y\in H^1_0(\Omega_\e;\partial\Omega),
\end{array}
\right.
\end{equation}
where $\Omega_\e$ is the subset of $\Omega$ such that $\partial\Omega\subset\partial\Omega_\e$, $\sigma>0$, and $\|A(x)\|_{\mathbb{S}^N}:=\max_{i,j=1,\dots,N}\left|a_{ij}(x)\right|\le\e^{-1}$ a.e. in $\Omega_\e$.
Here, $v$ stands for the fictitious control.

We show that OCP \eqref{0.2} has a nonempty set of solutions $(A_\e^0,v_\e^0,y_\e^0)$ for every $\e>0$.
Moreover, as follows from \eqref{0.2}$_1$, the cost functional
$I_\e$ seems to be rather sensitive with respect to the fictitious
controls. Due to this fact,
we prove that  the sequence $\left\{(A_\e^0,y_\e^0)\right\}_{\e>0}$ gives in the limit an optimal solution $(A^0,y^0)$ to the
original problem.

The main technical difficulty, which is related with the study of the asymptotic behaviour of OCPs \eqref{0.2} as $\e\to 0$,
deals with the identification of the limit
$\lim_{\e\to 0}\left\{\left<v^0_\e,y^0_{\e}\right>_{H^{-\frac{1}{2}}(\Gamma_\e);H^{\frac{1}{2}}(\Gamma_\e)}\right\}_{\e>0}$
of two weakly convergent sequences. Due to the special properties of the skew-symmetric parts of admissible controls $A\in \mathfrak{A}_{ad}\subset L^2\big(\Omega;\mathbb{S}^N\big)$, we show that this limit can be recovered in an explicit form.
We also show in this section that
the energy equalities to the regularized boundary value problems can be specified by two extra terms which characterize the presence of the-called hidden singular energy coming from $L^2$-properties of skew-symmetric components $A^{skew}$ of admissible controls.

In conclusion,  in Section~\ref{Sec 6}, we derive the optimality conditions for regularized OCPs \eqref{0.2} and show that the limit passage in optimality system for the regularized problems \eqref{0.2} as $\e\to 0$ leads to the optimality system for the original OCP \eqref{0.1}.

Let us point out that situations where the non uniqueness of some problems
occurs can lead to serious numerical difficulties. A good numerical scheme is
assumed to construct a desired solution. At a basic stage, the proof of the
Cauchy-Peano theorem for O.D.E is relevant of this situation: though the
construction of the solution may seem explicit, the fact the convergence is
obtained only for a subsequence is a brake to finding the desired solution see
\cite{coddlevin}. In the context of this paper, due to limited capacities of
computers, any kind of representation of matrices with $L^2$-coefficients will
lead to a truncated version of it. Naturally, thus, any attempt to treat
numerically some problem of the type \eqref{0.1}, will probably force the
algorithm to obtain an optimal variational solution. Thus, in order to produce
numerically non-variational optimal solutions of the problem \eqref{0.1}, the
method of perforated domain can be used. But in this case, one has to face the
fact that fictitious controls are distributions, which, of course, have a quite
bad numeric representation. One may thus think that those fictitious controls
could be taken in spaces of higher regularity ({\it e.g.} $L^2(\partial
\Omega_\e)$), but basic examples of non-variational solutions (see
\cite{Zh_97,Zhik1_04} ) shows that it is probably in general possible to have
non variational solutions with such properties.

Of course, one may wonder if situations of non uniqueness and moreover of lack
of procedure to obtain some uniqueness are relevant from the point of view of
applications. Nematic liquid crystals, as modeled by harmonic maps between
manifolds, can be, throughout this model, represented by minimizing
harmonic maps or stationary harmonic maps, for which, both of them satisfy
formally the same equation, but mathematically not. We refer to \cite{nematics}
for descriptions of this topic.

For our particular
equations, there may be some physical situations where the ill-posed problem
in \eqref{0.1} has also a mere sense in itself, notwithstanding the optimal
control problem. In deed, it is a common old principle to assume that the
stress Cauchy tensor $\sigma$ in mechanics is symmetric and leads to the
classical relations $$-\div(\sigma e(u))=f$$ where $e(u)$ is
given by
$$e(u)=\dfrac 12(\dfrac{\partial u_i}{\partial x_j}+\dfrac{\partial
u_j}{\partial x_i})$$ see \cite{Cauchy}.

On the contrary to this equation which can be stated in the form
$$-\div(A\nabla u) =f$$ for some symmetric matrix $A$, the Cosserats brothers have
introduced a non symmetric form for this equation, \cite{cosserats}. Of course
at a gross scale, the symmetric part of the stress behavior dominates the
behavior, but some micro-rotation may be observed in material according to
strain actions, for example in bones or some specific materials \cite{novacki}.
In that sense the assumption on $A^{skew}$ may be reflecting some
particular fragile point of a material, fragile meaning with respect to some
local ability of the surrounding matricant to degenerate in torsion, while
remaining stable in elongation.

In the spirit of the OCP, \eqref{0.1} can be thought as a way of realizing some
specific material with objective $y_d$ (for example a desired deformation)
according to some prescribed set of singular behaviours (the points where
$A^{skew}$ is singular), that is the material has a micro-rotative behavior at
only some prescribed set. Of course, according to the previous analysis,
designing such a material may be difficult to realize as a result of the
following analysis.


\section{Notation and Preliminaries}\label{Sec_1}
\label{Sec 1}

Let $\Omega$ be a bounded open connected subset of $\mathbb{R}^N$ ($N\ge 2$)
with Lipschitz boundary $\partial\Omega$. The spaces
$\mathcal{D}^\prime(\Omega)$ of distributions in $\Omega$ is the
dual of the space $C^\infty_0(\Omega)$.
As usual by $H_0^1(\Omega)$ we denote the closure of $C_0^\infty
(\Omega)$-functions in
the Sobolev space $H^{1}(\Omega)$, while $H^{-1}(\Omega)$ { denotes
the dual of $H^{1}_0(\Omega)$, any of its element can be represented, in the
sense of distribution,} as $f=f_0+\sum_j \partial_j
f_j$, with $f_0, f_1,\dots,f_N\in L^2(\Omega)$.
The {
usual } norm in $H_0^1(\Omega)$ { will be replaced by the equivalent
one  defined by}
$$
\|y\|_{H_0^1(\Omega)}=\left(\int_\Omega \|\nabla
y\|^2_{\mathbb{R}^N}\,dx\right)^{1/2}.
$$

Let $\Gamma$ be a part of the boundary
$\partial\Omega$ with positive $(N-1)$-dimensional measures. We consider
$$
C^\infty_0(\mathbb{R}^N;\Gamma)=\left\{\varphi\in
C^\infty_0(\mathbb{R}^N)\ :\ \varphi=0\text{ on }\Gamma\right\},
$$
and denote  $H^{1}_0(\Omega;\Gamma)$ its  closure   with respect to the norm
$$\|y\|=\left(\int_\Omega \|\nabla y\|^2_{\mathbb{R}^N}\,dx\right)^{1/2}.$$

For any vector field $v\in
L^2(\Omega;\mathbb{R}^N)$, the divergence of $v$
is an element $\mathrm{div}\,v$ of the space $H^{-1}(\Omega)$ defined by the formula
\begin{equation}
\label{1.00}
\left<\mathrm{div}\,v,\varphi\right>_{H^{-1}(\Omega);H^1_0(\Omega)}= -\int_\Omega
(v,\nabla\varphi)_{\mathbb{R}^N}\,dx,\quad \forall\,\varphi\in
C^\infty_0(\Omega),
\end{equation}
where $\left<\cdot,\cdot\right>_{H^{-1}(\Omega);H^1_0(\Omega)}$ denotes the duality pairing between
$H^{-1}(\Omega)$ and $H^1_0(\Omega)$, and
$(\cdot,\cdot)_{\mathbb{R}^N}$ stands for the scalar product in $\mathbb{R}^N$.

\bigskip \noindent
\textit{Symmetric and skew-symmetric matrices}.
Let $\mathbb{M}^N$ be the set of all $N\times N$ {real} matrices.
We denote by $\mathbb{S}^N_{skew}$ the  set of all
skew-symmetric matrices $C=[c_{ij}]_{i,j=1}^N$, i.e., $C$ is a square matrix whose transpose is also its {opposite}. Thus, if $C\in  \mathbb{S}^N_{skew}$ then $c_{ij}=-c_{ji}$ and, hence,
$c_{ii}=0$. Therefore, the set $\mathbb{S}^N_{skew}$ can be identified with the Euclidean space $\mathbb{R}^{\frac{N(N-1)}{2}}$.
Let $\mathbb{S}^N_{sym}$ be the set of all $N\times N$ symmetric matrices, which are obviously determined by $N(N+1)/2$ scalars.
Since $\mathbb{M}^N=\mathbb{S}^N_{sym}+\mathbb{S}^N_{skew}$ and $\mathbb{S}^N_{sym}\cap \mathbb{S}^N_{skew}=\emptyset$, it follows that $\mathbb{M}^N=\mathbb{S}^N_{sym}\oplus\mathbb{S}^N_{skew}$. Moreover, for each matrix
$B\in \mathbb{M}^N$, we have a unique representation
\begin{equation}
\label{1.0}
B=B^{sym}+B^{skew},
\end{equation}
where $B^{sym}:=\frac{1}{2}\left(B+B^t\right)\in \mathbb{S}^N_{sym}$ and $B^{skew}:=\frac{1}{2}\left(B-B^t\right)\in \mathbb{S}^N_{skew}$.  In the sequel, we will always identify each matrix
$B\in \mathbb{M}^N$ with its decomposition in the form \eqref{1.0}.

Let
$L^2(\Omega)^{\frac{N(N-1)}{2}}=L^2\big(\Omega;\mathbb{S}^N_{skew}\big)$
be the {normed} space of measurable square-integrable functions whose values are skew-sym\-met\-ric matrices  with the norm
$$
\|A\|_{L^2(\Omega;\mathbb{S}^N_{skew})}= \left(\int_\Omega \left(\max_{{i,j=1,\dots,N}\atop j>i} |a_{ij}(x)|\right)^2\,dx\right)^{1/2}.
$$
By analogy, we can define the spaces
$$
L^2(\Omega)^{\frac{N(N+1)}{2}}=L^2\big(\Omega;\mathbb{S}^N_{sym}\big)\ \text{ and }\ L^2(\Omega)^{N\times N}=L^2\big(\Omega;\mathbb{M}^N\big).
$$

Let $A(x)$ and $B(x)$ be given  matrices such that $A,B\in L^2(\Omega;\mathbb{S}^N_{skew})$. We say that these matrices are related by the binary relation $\preceq$ on the set $L^2(\Omega;\mathbb{S}^N_{skew})$ (in symbols, $A(x)\preceq B(x)$ a.e. in $\Omega$), if
\begin{equation}
\label{1.0b}
\mathcal{L}^N \left\{\bigcup_{i=1}^N \bigcup_{j=i+1}^N\left\{ x\in\Omega\ :\ | a_{ij}(x) |>  |b_{ij}(x)|\right\}\right\}=0.
\end{equation}
Here, $\mathcal{L}^N(E)$ denotes the $N$-dimensional Lebesgue measure of $E\subset \mathbb{R}^N$  defined on the completed borelian $\sigma$-algebra.

We define the divergence $\div A$ of a matrix $A\in L^2\big(\Omega;\mathbb{M}^N\big)$ as a vector-valued distribution
$d\in H^{-1}(\Omega;\mathbb{R}^N)$ by the following rule
\begin{equation}
\label{1.0a}
\left<d_i,\varphi\right>_{H^{-1}(\Omega);H^1_0(\Omega)}= -\int_\Omega
(\mathbf{a}_i,\nabla\varphi)_{\mathbb{R}^N}\,dx,\quad \forall\,\varphi\in
C^\infty_0(\Omega),\quad\forall\,i\in\left\{1,\dots,N\right\},
\end{equation}
where $\mathbf{a}_i$ stands for the $i$-th row of the matrix $A$.

For fixed two constants $\alpha$ and $\beta$ such that
$0<\alpha\le\beta<+\infty$,
we define $\mathfrak{M}_\alpha^\beta(\Omega)$ as a set of all matrices
$A=[a_{i\,j}\,]$ in $L^\infty(\Omega;\mathbb{S}^N_{sym})$ such that
\begin{equation}
\label{1.1}
\alpha I\le  A(x)\le \beta I,\quad \text{a.e. in } \Omega.
\end{equation}
Here, $I$ is the identity matrix  in $\mathbb{M}^{N}$, and \eqref{1.1} should be
considered in the sense of quadratic forms defined by $\left(A\xi,\xi\right)_{\mathbb{R}^N}$ for
$\xi\in \mathbb{R}^{N}$.

\textit{Unbounded bilinear forms on $H^1_0(\Omega)$.}
Let $A\in L^2\big(\Omega;\mathbb{M}^N\big)$ be an arbitrary matrix. In view of the representation $A=A^{sym}+A^{skew}$, we can associate with $A$ the form $\varphi(\cdot,\cdot)_A:H^1_0(\Omega)\times H^1_0(\Omega)\to \mathbb{R}$ following the rule
\[
\varphi(y,v)_A= \int_\Omega \big(\nabla v,A^{skew}(x)\nabla y\big)_{\mathbb{R}^N}\,dx,\quad \forall\,y,v\in H^1_0(\Omega).
\]
It is easy to see that, in general, this form is unbounded on $H^1_0(\Omega)$, however, it is expected some kind of alternating and antisymmetric properties of it. In order to deal with these concepts, we  introduce of the following set.

\begin{definition}
\label{Def 2.7} Let $A=A^{sym}+A^{skew}\in L^2\big(\Omega;\mathbb{M}^N\big)$ be a given matrix. We say that an element $y\in H^1_0(\Omega)$ belongs to the set $D(A)$ if
\begin{equation}
\label{2.8}
\left|\int_\Omega \big(\nabla \varphi,A^{skew}\nabla y\big)_{\mathbb{R}^N}\,dx\right|\le c(y,A^{skew}) \left(\int_\Omega |\nabla \varphi|^2_{\mathbb{R}^N}\,dx\right)^{1/2},\  \forall\,\varphi\in C^\infty_0(\Omega)
\end{equation}
with some constant $c$ depending only of $y$ and $A^{skew}$.
\end{definition}

{Consequently}, having set
\[
[y,\varphi]_A= \int_\Omega \big(\nabla \varphi,A^{skew}(x)\nabla y\big)_{\mathbb{R}^N}\,dx,\quad \forall\,y\in D(A),\ \forall\,\varphi\in C^\infty_0(\Omega),
\]
we see that the bilinear form $[y,\varphi]_A$ can be defined for all $\varphi\in H^1_0(\Omega)$ using \eqref{2.8} and the standard rule
\begin{equation}
\label{2.8a}
[y,\varphi]_A=\lim_{\e\to 0}\,[y,\varphi_\e]_A,
\end{equation}
where $\left\{\varphi_\e\right\}_{\e>0}\subset C^\infty_0(\Omega)$ and $\varphi_\e\rightarrow \varphi$ strongly in $H^1_0(\Omega)$. In this case the value $[v,v]_A$ is finite for every $v\in D(A)$, although the "integrand"\
$\big(\nabla v,A^{skew}\nabla v\big)_{\mathbb{R}^N}$ need not be integrable, in general.

\bigskip

\textit{Functions with bounded variations}. Let $f:\Omega\to \mathbb{R}$ be a function of $L^1(\Omega)$. Define
\begin{multline*}
TV(f):=\int_\Omega|Df|= \sup\Big\{\int_\Omega f\,(\nabla,\varphi)_{\mathbb{R}^N}\,dx\,  :\\
\varphi=(\varphi_1,\dots,\varphi_N)\in C^1_0(\Omega;\mathbb{R}^N),\
|\varphi(x)|\le 1\ \text{for}\ x\in \Omega\Big\},
\end{multline*}
where $(\nabla,\varphi)_{\mathbb{R}^N}=\sum_{i=1}^N\frac{\partial\varphi_i}{\partial x_i}$.

According to the Radon-Nikodym theorem, if $TV(f)<+\infty$ then the distribution $Df$ is a measure and there exist a vector-valued function $\nabla f\in[L^1(\Omega)]^N$ and a measure $D_s f$, singular with respect to the $N$-dimensional Lebesgue measure $\mathcal{L}^N\lfloor\Omega$ restricted to $\Omega$, such that
$Df=\nabla f\mathcal{L}^N\lfloor\Omega+D_s f.$

\begin{definition}
\label{Def 1.10} A function $f\in L^1(\Omega)$  is said to have a
bounded variation in $\Omega$ if $TV(f)<+\infty$. By
$BV(\Omega)$ we denote the space of all functions in $L^1(\Omega)$
with bounded variation, i.e.
$$
BV(\Omega)=\left\{f\in L^1(\Omega)\ :\ TV(f)<+\infty\right\}.
$$
\end{definition}

Under the norm
$\|f\|_{BV(\Omega)}=\|f\|_{L^1(\Omega)}+TV(f),$
$BV(\Omega)$ is a Banach space.  For our further analysis, we need the following properties
of $BV$-functions (see \cite{Evans}):
\begin{proposition}
\label{Prop 1.11}
\begin{enumerate}
\item Let $\{f_k\}_{k=1}^\infty$ be a sequence in $BV(\Omega)$ strongly converging to some $f$ in $L^1(\Omega)$ and satisfying condition $\sup_{k\in \mathbb{N}}TV(f_k)<+\infty$. Then
    $$f\in BV(\Omega)\quad\text{and}\quad TV(f)\le\liminf_{k\to\infty}TV(f_k);$$
\item for every $f\in BV(\Omega)\cap L^r(\Omega)$, $r\in [1,+\infty)$, there exists a sequence $\{f_k\}_{k=1}^\infty\subset C^\infty(\Omega)$ such that
    \[
    \lim_{k\to\infty} \int_\Omega |f-f_k|^r\,dx =0\quad\text{and}\quad \lim_{k\to\infty} TV(f_k)=TV(f);
    \]
\item for every bounded sequence $\{f_k\}_{k=1}^\infty\subset BV(\Omega)$ there exists a subsequence, still denoted by $f_k$, and a function $f\in BV(\Omega)$ such that $f_k\rightarrow f$ in $L^1(\Omega)$.
\end{enumerate}
\end{proposition}

\bigskip
\noindent
\textit{Variational convergence of optimal control problems.}
Throughout the paper $\e$ denotes a small parameter which varies
within a strictly decreasing sequence of positive numbers converging
to $0$. When we write $\e>0$, we consider only the elements of
this sequence, in the case
 $\e\ge 0$ we also consider its
limit $\e=0$.
Let $I_\e:\mathbb{U}_\e\times
\mathbb{Y}_\e\rightarrow \overline{\mathbb{R}}$ be a cost functional, $\mathbb{Y}_\e$ be a space of states,
and $\mathbb{U}_\e$ be a space of controls. Let $\min\left\{I_\e(u,y)\,:\
(u,y)\in \Xi_\e\right\}$ be a parameterized OCP,
where
$$
\Xi_\e\subset\left\{(u_\e,y_\e)\in \mathbb{U}_\e\times
\mathbb{Y}_\e \ :\ u_\e\in U_\e,\ I_\e(u_\e,y_\e)<+\infty\right\}
$$
is a set of all admissible pairs linked by some state equation.
Hereinafter we always associate to such OCP the
corresponding constrained minimization problem:
\begin{equation}
\label{1.3} (\mathrm{CMP_\e})\ :\qquad\qquad \left<
\inf\limits_{(u,y)\in\,\Xi_\e} I_\e(u,y)\right>.
\end{equation}
Since the sequence of
constrained minimization problems \eqref{1.3} lives in variable spaces
$\mathbb{U}_\e\times\mathbb{Y}_\e$, we assume that there exists a Banach
space $\mathbb{U}\times \mathbb{Y}$ with respect to
which a convergence in the scale of spaces $\left\{\mathbb{U}_\e\times \mathbb{Y}_\e\right\}_{\e>0}$ is defined
(for the details, we refer to \cite{KL,Zh_98}). In the sequel, we use the following notation for this convergence
$(u_\e,y_\e)\stackrel{\mu}{\longrightarrow}\, (u,y)$  in $\mathbb{U}_\e\times \mathbb{Y}_\e$.

In order to study the asymptotic behavior of a family of $(\mathrm{CMP_\e})$,
the passage to the limit in \eqref{1.3} as the small parameter $\e$ tends to zero has to be realized.
The expression \textquotedblleft passing to the limit"\ means that we have to find a kind
of \textquotedblleft limit cost functional"\ $I$ and \textquotedblleft limit set of constraints" $\Xi$
with a clearly defined structure such that the limit object
$\left<\inf_{(u,y)\in\Xi} I(u,y)\right>$ may be interpreted
as some~OCP.

Following the scheme of the direct variational convergence \cite{KL},
we adopt the following definition for the convergence of minimization problems in variable spaces.

\begin{definition}
\label{Def 1.4} A problem $\left<\inf_{(u,y)\in\Xi} I(u,y)\right>$ is
the variational limit of the sequence (\ref{1.3}) as $\e\to 0$
\[
\left(\text{in symbols, }\quad
\left<\inf\limits_{(u,y)\in\,\Xi_\e} I_\e(u,y)\right>\,\stackrel{\text{Var}}{\xrightarrow[\e\to 0]{}}\,
\left<\inf_{(u,y)\in\Xi} I(u,y)\right>\ \right)
\]
if
and only if the following conditions are satisfied:
\begin{enumerate}
\item[$\quad$(d)] The space $\mathbb{U}\times\mathbb{Y}$ possesses the
weak $\mu$-approximation property with respect to the scale of
spaces  $\left\{\mathbb{U}_\e\times \mathbb{Y}_\e\right\}_{\e>0}$, that is, for every $\delta>0$ and
every pair $(u,y)\in \mathbb{U}\times\mathbb{Y}$, there exist a pair $(u^\ast,y^\ast)\in
\mathbb{U}\times\mathbb{Y}$ and a sequence $\left\{(u_\e,y_\e)\in
\mathbb{U}_\e\times \mathbb{Y}_\e\right\}_{\e>0}$ such that
\begin{equation}
\label{1.5} \|u-u^\ast\|_{\mathbb{U}}+\|y-y^\ast\|_{\mathbb{Y}}\le \delta
\quad\text{and}\quad (u_\e,y_\e)\stackrel{\mu}{\longrightarrow}\, (u^\ast,y^\ast)\quad \text{in }\ %
\mathbb{U}_\e\times \mathbb{Y}_\e.
\end{equation}

\item[$\quad$(dd)] If sequences $\left\{\e_k\right\}_{k\in \mathbb{N}}$ and
$\left\{(u_k,y_k)\right\}_{k\in \mathbb{N}}$ are such that $\varepsilon_k
\rightarrow 0$ as $k\rightarrow \infty$, $(u_k,y_k)\in
\Xi_{\varepsilon_k}$ $\forall\,k\in \mathbb{N}$, and
$(u_k,y_k)\stackrel{\mu}{\longrightarrow}\,(u,y)$ in $\mathbb{U}_{\varepsilon_k}\times\mathbb{Y}_{\varepsilon_k}$, then
\begin{equation}
\label{1.6} (u,y)\in \Xi;\quad I(u,y)\le
\liminf_{k\to\infty}I_{\e_k}(u_k,y_k).
\end{equation}

\item[$\quad$(ddd)] For every $(u,y) \in \Xi\subset \mathbb{U}\times\mathbb{Y}$ and any
$\delta>0$, there are a constant $\varepsilon^0 >0$ and a
sequence
$\left\{(u_\e,y_\e)\right\}_{\varepsilon>0}$ (called a $(\Gamma,\delta)$-realizing sequence) such that
\begin{gather}
(u_\e,y_\e)\in\Xi_\varepsilon,\ \forall\,\varepsilon \leq
  \varepsilon^0,\quad
(u_\e,y_\e)\,\stackrel{\mu}{\longrightarrow}\, (\widehat{u},\widehat{y})\
\text{ in }\ \mathbb{U}_\e\times \mathbb{Y}_\e,\label{1.7a}\\
 \|u-\widehat{u}\|_{\mathbb{U}}+\|y-\widehat{y}\|_{\mathbb{Y}}\le
\delta,\label{1.7b}\\
\label{1.7c} I(u,y)\ge \limsup_{\varepsilon\to 0}
I_{\varepsilon}(u_\e,y_\e)-\widehat{C}\delta,
\end{gather}
with some constant $\widehat{C}>0$ independent of $\delta$.
\end{enumerate}
\end{definition}

Then the following result takes place \cite{KL}.
\begin{theorem}
\label{Th 1.8} Assume that the constrained minimization problem
\begin{equation}
\label{1.9} \Big \langle \inf_{(u,y) \in \Xi_0} I_0(u,y)\Big
\rangle
\end{equation}
is the variational limit of sequence (\ref{1.3}) in the sense
of Definition \ref{Def 1.4} and this problem has a nonempty set of
solutions
$$
\Xi_0^{opt}:=\left\{(u^0,y^0)\in\Xi_0\ :\ I_0(u^0,y^0)=\inf_{(u,y) \in \Xi_0} I_0(u,y)\right\}.
$$
For every $\varepsilon >0$, let
$(u^0_\varepsilon,y^0_\e)\in \Xi_\varepsilon$ be a minimizer of $I_\e$ on the
corresponding set $\Xi_\e$. If the sequence
$\{(u^0_\varepsilon,y^0_\e)\}_{\varepsilon >0}$ is relatively compact with
respect to the $\mu$-convergence in variable spaces $\mathbb{U}_\e\times \mathbb{Y}_\e$, then
there exists a pair $(u^0,y^0)\in \Xi_0^{opt}$ such that
\begin{gather}
\label{1.10} (u^0_\varepsilon,y^0_\e)\,\stackrel{\mu}{\longrightarrow}\,
(u^0,y^0)\quad \text{in }\ %
\mathbb{U}_\e\times \mathbb{Y}_\e,\\
\inf_{(u,y)\in\,\Xi_0}I_0(u,y)= I_0\left(u^0,y^0\right) =\lim_{\varepsilon\to
0} I_{\varepsilon}(u^0_\varepsilon,y^0_\varepsilon) =\lim_{\varepsilon\to
0}\inf_{(u_\varepsilon,y_\e)\in\,{\Xi}_\varepsilon}
{I}_\varepsilon(u_\varepsilon,y_\e).\label{1.11}
\end{gather}
\end{theorem}


\section{Setting of the Optimal Control Problem}
\label{Sec 2}

Let $f\in H^{-1}(\Omega)$ be a given distribution.
The optimal control
problem we consider in this paper is to minimize the discrepancy (tracking error)
between a given distribution $y_d\in L^2(\Omega)$ and a solution $y$ of the Dirichlet boundary value problem for the linear elliptic equation
\begin{gather}
\label{2.1}
-\div\big(A(x)\nabla y\big) = f\quad\text{in }\ \Omega,\\
\label{2.2}
y=0\text{ on }\partial\Omega
\end{gather}
by choosing an appropriate control $A\in L^2(\Omega;\mathbb{M}^N)$.

More precisely, we are
concerned with the following OCP
\begin{gather}
\label{2.3}
\text{Minimize } I(A,y)=\left\|y-y_d\right\|^2_{L^2(\Omega)}+\int_\Omega\left(\nabla y, A^{sym}\nabla y\right)_{\mathbb{R}^N}\,dx\\
\label{2.3a}
\text{subject to the constraints \eqref{2.1}--\eqref{2.2} with $A\in \mathfrak{A}_{ad}\subset L^2(\Omega;\mathbb{M}^N)$.}
\end{gather}

In order to make a precise meaning of the OCP setting and indicate of its characteristic properties, we begin with the
definition of the class of admissible controls $\mathfrak{A}_{ad}$.

Let $A^\ast\in L^2(\Omega;\mathbb{S}^N_{skew})$ be a given nonzero matrix, let $c$ be a given positive constant, and let $Q$ be a nonempty convex compact subset of $L^2(\Omega;\mathbb{S}^N_{skew})$ such that the null matrix $A\equiv [0]$ belongs to $Q$. To define possible classes of
admissible controls, we introduce the following sets
\begin{align}
\label{2.3b} U_{a,1}&=\left\{\left. A=[a_{i\,j}]\in L^1(\Omega;\mathbb{S}^N_{sym})\
\right| TV(a_{ij})\le c,\ 1\le i\le j\le N\right\},\\[1ex]
\label{2.3bb} U_{b,1}&=\left\{\left. A=[a_{i\,j}]\in L^\infty(\Omega;\mathbb{S}^N_{sym})\
\right| A\in \mathfrak{M}_\alpha^\beta(\Omega)\right\},\\[1ex]
\label{2.3c} U_{a,2}&=\left\{\left. A=[a_{i\,j}]\in L^2(\Omega;\mathbb{S}^N_{skew})\ \right|\ A(x)\preceq A^\ast(x)\ \text{a.e. in }\ \Omega\right\},\\
\label{2.3cc} U_{b,2}&=\left\{\left. A=[a_{i\,j}]\in L^2(\Omega;\mathbb{S}^N_{skew})\ \right|\  A\in Q\right\}.
\end{align}

\begin{remark}
It is worth to note that
$$
\mathfrak{A}_{ad,1}:=U_{a,1}\cap U_{b,1}\ne\emptyset\quad\text{ and }\quad \mathfrak{A}_{ad,2}:=U_{a,2}\cap U_{b,2}\ne\emptyset.
$$
Indeed, the validity of these relations immediately follows from \eqref{2.3b}--\eqref{2.3bb}, definition of the binary relation $\preceq$, and properties of the matrix $A^\ast$.  In order to describe a possible way for the choice of the set $Q$, we may use the following result \cite{Orlov}:
An arbitrary closed bounded subset $C\subset L^2(\Omega)$ is compact if and only if, for any orthonormal basis $\left\{g_k\right\}_{k\in \mathbb{N}}$ in $L^2(\Omega)$, there exists a compact ellipsoid
\[
C_\e=\left\{g=\sum_{k=1}^\infty  \alpha_k g_k \in L^2(\Omega)\ \left|\  \sum_{k=1}^\infty \frac{|\alpha_k|^2}{\e_k^2}\le 1\right.\right\}
\]
with $\e_k\to 0$ as $k\to\infty$, such that $C\subseteq C_\e$.
\end{remark}

As a result, we adopt the following concept.
\begin{definition}
\label{Def 2.3e}
We say that a matrix $A=A^{sym}+A^{skew}$ is an
admissible control to the Dirichlet boundary value problem \eqref{2.1}--\eqref{2.2} (in symbols, $A\in \mathfrak{A}_{ad}\subset L^2(\Omega;\mathbb{M}^N)$)
if $A^{sym}\in \mathfrak{A}_{ad,1}$ and $A^{skew}\in \mathfrak{A}_{ad,2}$.
\end{definition}

For our further analysis, we  use of the following results.
\begin{proposition}
\label{Prop 2.3d} If
$\left\{A_k^{sym}\right\}_{k\in
\mathbb{N}}\subset \mathfrak{A}_{ad,1}$ and $A_k^{sym}\rightarrow A_0^{sym}$ in $L^1(\Omega;\mathbb{S}^N_{sym})$ as $k\to\infty$, then $A_0^{sym}\in \mathfrak{A}_{ad,1}$ and
\begin{equation}
\label{2.3dd}
A_k^{sym}\rightarrow A_0^{sym}\quad\text{in }\ L^p(\Omega;\mathbb{S}^N_{sym}),\ \forall\,p\in [1,+\infty).
\end{equation}
\end{proposition}
\begin{proof}
Since the sequence $\left\{A_k^{sym}\right\}_{k\in
\mathbb{N}}$ converges strongly to $A_0^{sym}$ in $L^1(\Omega;\mathbb{S}^N_{sym})$ and $A_k^{sym}\in \mathfrak{M}_\alpha^\beta(\Omega)$ for every $k\in \mathbb{N}$, it follows that $\alpha I\le A_0^{sym}\le \beta I$ a.e. in $\Omega$.
Hence, $A_0^{sym}\in U_{b,1}$. At the same time, following assertion (i) of Proposition~\ref{Prop 1.11}, we have $TV(a_{ij})\le c$ for each entry of matrix $A_0^{sym}$. As a result, we conclude $A_0^{sym}\in U_{a,1}$, and, therefore, $A_0^{sym}\in \mathfrak{A}_{ad,1}$.
Concerning the property \eqref{2.3dd}, it immediately follows from the following estimate
\begin{align*}
&\|A_k^{sym}-A_0^{sym}\|^p_{L^p(\Omega;\mathbb{S}^N_{sym})}= \int_\Omega \left(\max_{{i,j=1,\dots,N}\atop j\ge i} |a^k_{ij}(x)-a^0_{ij}(x)|\right)^p\,dx\\
=\ &\int_\Omega \left(\max_{{i,j=1,\dots,N}\atop j\ge i} |(a^k_{ij}(x)-\alpha)-(a^0_{ij}(x)-\alpha)|\right)^{p-1}\max_{{i,j=1,\dots,N}\atop j\ge i} |a^k_{ij}(x)-a^0_{ij}(x)|\,dx\\
\le\ & 2^{p-1}(\beta-\alpha)^{p-1}\|A_k^{sym}-A_0^{sym}\|_{L^1(\Omega;\mathbb{S}^N_{sym})},\quad\forall\,p\in [1,+\infty).
\end{align*}
\end{proof}

\begin{proposition}
\label{Prop 2.3ddd} $\mathfrak{A}_{ad,1}$ is a sequentially compact subset of $L^p(\Omega;\mathbb{S}^N_{sym})$ for every $p\in [1,+\infty)$.
\end{proposition}
\begin{proof}
Let $\left\{A_k^{sym}\right\}_{k\in
\mathbb{N}}$ be a sequence of $\mathfrak{A}_{ad,1}$. In view of definition of the set $U_{a,1}$, we see that $\left\{A_k^{sym}\right\}_{k\in
\mathbb{N}}$ is a bounded sequence in $BV(\Omega;\mathbb{S}^N_{sym})$. Hence, to conclude the proof, it is enough to apply Proposition~\ref{Prop 2.3d} and assertion~(iii) of Proposition~\ref{Prop 1.11}.
\end{proof}

Taking these observations into account, we prove the following results.
\begin{proposition}
\label{Prop 2.3f} The set $\mathfrak{A}_{ad}$ is nonempty, convex, and sequentially compact with respect to the strong topology of
$L^2(\Omega;\mathbb{M}^N)$.
\end{proposition}
\begin{proof}
Let
$\left\{A_k=A_k^{sym}+A_k^{skew}\right\}_{k\in
\mathbb{N}}\subset \mathfrak{A}_{ad}$ be an arbitrary sequence of admissible
controls.  Since
\begin{align*}
\mathfrak{A}_{ad}&=\mathfrak{A}_{ad,1}\oplus \mathfrak{A}_{ad,2}, \quad \mathfrak{A}_{ad,1}\subset BV(\Omega;\mathbb{S}^N_{sym}),\\ \mathfrak{A}_{ad,2}&\subset U_{b,2},\ \text{ and }\
\text{$U_{b,2}$ is a compact in }\ L^2(\Omega;\mathbb{S}^N_{skew}),
\end{align*}
we may suppose that there
exist matrices $A^{sym}_0\in BV(\Omega;\mathbb{S}^N_{sym})\cap L^\infty(\Omega;\mathbb{S}^N_{sym})$ (see Propositions~\ref{Prop 2.3d}--\ref{Prop 2.3ddd}) and $A^{skew}_0\in U_{b,2}\subset L^2(\Omega;\mathbb{S}^N_{skew})$ such that within a subsequence
\begin{align}
\label{2.3m}
  A_k^{sym}&\rightarrow A_0^{sym}\quad\text{in }\ L^p(\Omega;\mathbb{S}^N_{sym}),\ \forall\,p\in [1,+\infty),\\
  A_k^{sym}&\stackrel{\ast}{\rightharpoonup}\, A_0^{sym}\quad\text{in }\ L^\infty(\Omega;\mathbb{S}^N_{sym}),\\
  \label{2.3n} A_k^{skew}&\rightarrow A_0^{skew}\quad\text{in }\ L^2(\Omega;\mathbb{S}^N_{skew}),\\
  \label{2.3l}\text{and }\ A_k^{skew}&\rightarrow A_0^{skew}\ \text{ almost everywhere in }\  \Omega.
 \end{align}

Combining these facts with \eqref{2.3c} and the definition of the binary relation $\preceq$ (see \eqref{1.0b}),
we arrive at the conclusion: $A_0^{skew}\in U_{a,2}$, and hence
$$
A_k:=A_k^{sym}+A_k^{skew}\rightarrow A_0^{sym}+A_0^{skew}=:A_0\ \text{ in }\ L^2(\Omega;\mathbb{M}^N).
$$
Thus, $A_0\in \mathfrak{A}_{ad}$. Since the convexity of $\mathfrak{A}_{ad}$ is obviously valid, this concludes the proof.
\end{proof}

The distinguishing feature of optimal control problem \eqref{2.3}--\eqref{2.3a} is the fact that the matrix-valued control $A\in \mathfrak{A}_{ad}$ is merely measurable and belongs to the space $L^2\big(\Omega;\mathbb{M}^N\big)$ (rather than the space of bounded matrices $L^\infty\big(\Omega;\mathbb{M}^N\big)$). As we will see later, this entails a number of pathologies with respect to the standard properties of optimal control problems for the classical elliptic equations, even with 'a good' right-hand $f$. In particular, the unboundedness of the skew-symmetric part of matrix $A\in \mathfrak{A}_{ad}$ can have a reflection in non-uniqueness of weak solutions to the corresponding boundary value problem.

\begin{definition}
\label{Def 2.4} We say that a function $y=y(A,f)$ is  a weak
solution to boundary value problem \eqref{2.1}--\eqref{2.2} for
a fixed control $A=A^{sym}+A^{skew}\in \mathfrak{A}_{ad}$ and a given distribution $f\in
H^{-1}(\Omega)$, if
$y\in H^1_0(\Omega)$
and the integral identity
\begin{equation}
\label{2.5}
\int_\Omega \big(\nabla \varphi,A^{sym}\nabla y+A^{skew}\nabla y\big)_{\mathbb{R}^N}\,dx
=\left<f,\varphi\right>_{H^{-1}(\Omega);H^1_0(\Omega)}
\end{equation}
holds true for any $\varphi\in C^\infty_0(\Omega)$.
\end{definition}

Note that by H\"{o}lder's inequality this definition makes sense for any matrix $A\in L^2\big(\Omega;\mathbb{M}^N\big)$. At the same time, in view of Definition~\ref{Def 2.7}, the following result gives another motivation to introduce the set $D(A)$.
\begin{proposition}
\label{Prop 2.9}
Let $y\in H^1_0(\Omega)$ be a weak solution to the boundary value problem \eqref{2.1}--\eqref{2.2} for a given control $A=A^{sym}+A^{skew}\in \mathfrak{A}_{ad}$ in the sense of Definition~\ref{Def 2.4}. Then $y\in D(A)$.
\end{proposition}
\begin{proof}
In order to verify the validity of this assertion it is enough to rewrite the integral identity \eqref{2.5} in the form
\begin{equation}
\label{2.10}
[y,\varphi]_A=-\int_\Omega \big(A^{sym}\nabla y,\nabla \varphi\big)_{\mathbb{R}^N}\,dx
 + \left<f,\varphi\right>_{H^{-1}(\Omega);H^1_0(\Omega)}
\end{equation}
and apply H\"{o}lder's inequality to the right-hand side of \eqref{2.10}. As a result, we have
\begin{align*}
\Big|[y,\varphi]_A\Big|&\le\left(\|A^{sym}\|_{L^\infty(\Omega;\mathbb{S}^N_{sym})}\|\nabla y\|_{L^2(\Omega;\mathbb{R}^N)}+
\|f\|_{H^{-1}(\Omega)}\right)\|\varphi\|_{H^1_0(\Omega)}\\
&\le \left(\beta\|y\|_{H^1_0(\Omega)}+\|f\|_{H^{-1}(\Omega)}\right)\|\varphi\|_{H^1_0(\Omega)}.
\end{align*}
\end{proof}

\begin{remark}
\label{Rem 2.11}
Due to Proposition \ref{Prop 2.9}, Definition \ref{Def 2.4} can be reformulated as follows: $y$ is a weak solution to the problem \eqref{2.1}--\eqref{2.2} for a given control $A=A^{sym}+A^{skew}\in \mathfrak{A}_{ad}$, if and only if $y\in D(A)$ and
\begin{equation}
\label{2.11}
\int_\Omega \big(A^{sym}\nabla y,\nabla \varphi\big)_{\mathbb{R}^N}\,dx + [y,\varphi]_A
=\left<f,\varphi\right>_{H^{-1}(\Omega);H^1_0(\Omega)}\,\quad\forall\,\varphi\in H^1_0(\Omega).
\end{equation}
Moreover, as immediately follows from \eqref{2.8a} and \eqref{2.11}, every weak solution $y\in D(A)$ to the problem \eqref{2.1}--\eqref{2.2} satisfies the energy equality
\begin{equation}
\label{2.12}
\int_\Omega \big(A^{sym}\nabla y,\nabla y\big)_{\mathbb{R}^N}\,dx + [y,y]_A
=\left<f,y\right>_{H^{-1}(\Omega);H^1_0(\Omega)}.
\end{equation}
\end{remark}

It is well known that boundary value problem \eqref{2.1}--\eqref{2.2} is ill-posed, in general (see, for instance, \cite{F_P,Maz,Serin,Zh_97,Zhik1_04}). It means that there exists
a matrix $A\in L^2\big(\Omega;\mathbb{M}^N\big)$ such that the corresponding state $y\in H^1_0(\Omega)$ may be not unique. It is clear that
in this case, it would not be correct to write down $y=y(A,f)$.
To avoid this situation, we  adopt the following notion.
\begin{definition}
\label{Def 2.13}
We say that  $(A,y)$ is an admissible pair to the OCP \eqref{2.3}--\eqref{2.3a} if $A\in \mathfrak{A}_{ad}\subset L^2\big(\Omega;\mathbb{M}^N\big)$,
$y\in D(A)\subset H^1_0(\Omega)$, and the pair $(A,y)$ is related by the integral identity \eqref{2.11}.
\end{definition}

We
denote by $\Xi$ the set of all admissible pairs for
the OCP \eqref{2.3}--\eqref{2.3a}. We say that a pair $(A^0,y^0)\in L^2\big(\Omega;\mathbb{M}^N\big)\times D(A^0)$ is
optimal for problem \eqref{2.3}--\eqref{2.3a} if
\[
(A^0,y^0)\,\in\,\Xi\ \text{ and }\ I(A^0,y^0)=\inf_{(A,y)\in\,\Xi}
I(A,y).
\]

As follows from the definition of the bilinear form $[y,\varphi]_A$, the value $[y,y]_A$ may not of constant sign for all $y\in D(A)$. Hence, the energy equality \eqref{2.12} does not allow us to derive a reasonable a priory estimate in $H^1_0$-norm for the weak solutions. In spite of this, we show that the  OCP \eqref{2.3}--\eqref{2.3a} is well-posed. This problem is,
thus, yet another example for the difference between well-posedness
for optimal control problems for systems with distributed parameters
and  partial differential equations (see \cite{KL} for a discussion and further examples).

Let $\tau$ be the topology on the set of admissible pairs
$\Xi\subset L^2\big(\Omega;\mathbb{M}^N\big)\times
H^{1}_0(\Omega)$
which we define as the product of the strong topology of
$L^2\big(\Omega;\mathbb{M}^N\big)$ and the weak topology of
$H^{1}_0(\Omega)$.

\begin{theorem}
\label{Th 2.14} Assume that OCP \eqref{2.3}--\eqref{2.3a} is regular, i.e. $\Xi\ne\emptyset$. Then, for each $f\in
H^{-1}(\Omega)$  and $y_d\in L^2(\Omega)$, this problem admits at least one solution.
\end{theorem}
\begin{proof}
Since the original problem is regular and the cost functional for the given problem is bounded below on $\Xi$,
it follows that there exists a minimizing sequence
$\left\{(A_k,y_k)\right\}_{k\in\,\mathbb{N}}\subset\Xi$ such that
\[
I(A_k,y_k)\,\xrightarrow[k\to\infty]{} I_{\mathrm{min}}\equiv
\inf_{(A,y)\in\,\Xi} I(A,y)\ge 0.
\]
Hence,  $\sup_{k\in \mathbb{N}}I(A_k,y_k)\le C$, where the constant
$C$ is independent of $k$. Since
\begin{gather*}
\sup_{k\in \mathbb{N}}\left\|y_k\right\|^2_{H^1_0(\Omega)}
 \le \alpha^{-1} \sup_{k\in \mathbb{N}}\int_\Omega\left(\nabla y_k, A^{sym}_k\nabla y_k\right)_{\mathbb{R}^N}\,dx
\le \alpha^{-1} \sup_{k\in \mathbb{N}}I(A_k,y_k) \le \alpha^{-1} C,
\end{gather*}
in view of Proposition~\ref{Prop 2.3f}, it follows that passing to a
subsequence if necessary, we may assume that there exists a pair $(A_0,y_0)\in \mathfrak{A}_{ad}\times H^1_0(\Omega)$ such that
\begin{align}
\label{2a.16}
A_k:=A_k^{sym}+A_k^{skew}&\rightarrow A_0^{sym}+A_0^{skew}=:A_0\ \text{ in }\ L^2(\Omega;\mathbb{M}^N),\\
\label{2a.16a}
A_k^{sym}&\rightarrow A_0^{sym}\quad\text{in }\ L^p(\Omega;\mathbb{S}^N_{sym}),\ \forall\,p\in [1,+\infty),\\
\label{2a.16b} A_k^{skew}&\rightarrow A_0^{skew}\quad\text{in }\ L^2(\Omega;\mathbb{S}^N_{skew}),\\
\label{2a.16c} y_k&\rightharpoonup y_0 \ \text{ in
}\ H^1_0(\Omega),\quad I(A_0,y_0)<+\infty.
\end{align}
Since $( A_k,y_k)\in\Xi$ for every $k\in \mathbb{N}$, it follows that the integral identity
\begin{equation}
\label{2a.5}
\int_\Omega \big(\nabla \varphi,A^{sym}_k\nabla y_k\big)_{\mathbb{R}^N}\,dx
+\int_\Omega \big(\nabla \varphi,A^{skew}_k\nabla y_k\big)_{\mathbb{R}^N}\,dx
=\left<f,\varphi\right>_{H^{-1}(\Omega);H^1_0(\Omega)}
\end{equation}
holds true for all $\varphi\in C^\infty_0(\Omega)$.

In order to pass to the limit in \eqref{2a.5}, we note that
\begin{align*}
\int_\Omega \big(\nabla \varphi,A^{skew}_k\nabla y_k\big)_{\mathbb{R}^N}\,dx=&
-\int_\Omega \big((A^{skew}_k-A^{skew}_0)\nabla \varphi,\nabla y_k\big)_{\mathbb{R}^N}\,dx\\
&-\int_\Omega \big( A^{skew}_0\nabla\varphi,\nabla y_k\big)_{\mathbb{R}^N}\,dx=I_{1,k}+I_{2,k}
\end{align*}
by the skew-symmetry property of $A^{skew}_k$ and $A^{skew}_0$. Hence, in view of \eqref{2a.16b}--\eqref{2a.16c}, we have
\begin{align*}
\lim_{k\to\infty}|I_{1,k}|&\le \|\varphi\|_{C^1(\Omega)} \sup_{k\in \mathbb{N}}\|\nabla y_k\|_{L^2(\Omega;\mathbb{R}^N)}
\lim_{k\to\infty}\left\| A^{skew}_k-A^{skew}_0\right\|_{L^2(\Omega;\mathbb{S}^N_{skew})}=0,\\
\lim_{k\to\infty}I_{2,k}&\stackrel{\text{by \eqref{2a.16c}}}{=}\,-\int_\Omega \big(A^{skew}_0\nabla \varphi,\nabla y_0\big)_{\mathbb{R}^N}\,dx=\int_\Omega \big(\nabla \varphi,A^{skew}_0\nabla y_0\big)_{\mathbb{R}^N}\,dx\\
\text{ since }&\ A^{skew}_0\nabla\varphi\in L^2(\Omega;\mathbb{R}^N)\quad \forall\,\varphi\in C^\infty_0(\Omega).
\end{align*}

Having applied the same arguments to the first term in \eqref{2a.5}, as a result of the limit passage in \eqref{2a.5}, we finally obtain:
the pair $(A_0,y_0)$ is related by identity \eqref{2.5}. Hence, $y_0\in D(A_0)$ by Proposition \ref{Prop 2.9}.
Thus, $(A_0,y_0)$ is an admissible pair to problem \eqref{2.3}--\eqref{2.3a}.

It remains to show that $(A_0,y_0)$ is an optimal pair. Indeed, in view of the compactness of the embedding $H^1_0(\Omega)\hookrightarrow L^2(\Omega)$, one gets
\begin{align*}
I_{\mathrm{min}}&=\lim_{k\to\infty}I(A_k,y_k)=
\lim_{k\to\infty} \left[\left\|y_k-y_d\right\|^2_{L^2(\Omega)}+\int_\Omega\left(\nabla y_k, A^{sym}_k\nabla y_k\right)_{\mathbb{R}^N}\,dx\right]\\
&=\left\|y_0-y_d\right\|^2_{L^2(\Omega)}+\lim_{k\to\infty}\int_\Omega\left\|(A^{sym}_k)^{1/2}\nabla y_k\right\|^2_{\mathbb{R}^N}\,dx.
\end{align*}
At the same time, due to \eqref{2a.16a}, we obviously have $(A^{sym}_k)^{1/2}\rightarrow (A^{sym}_0)^{1/2}$ in $L^2(\Omega;\mathbb{S}^N_{sym})$. Hence, taking into account the condition \eqref{2a.16c}, we get $(A^{sym}_k)^{1/2}\nabla y_k \rightharpoonup (A^{sym}_0)^{1/2}\nabla y_0$ in $L^2(\Omega;\mathbb{R}^N)$. So, using the lower semicontinuity of the norm $\|\cdot\|_{L^2(\Omega;\mathbb{R}^N)}$ with respect to the the weak topology of $L^2(\Omega;\mathbb{R}^N)$, we finally obtain
\begin{align}
\lim_{k\to\infty}\int_\Omega\left\|(A^{sym}_k)^{1/2}\nabla y_k\right\|^2_{\mathbb{R}^N}\,dx\ge&
\int_\Omega\left\|(A^{sym}_0)^{1/2}\nabla y_0\right\|^2_{\mathbb{R}^N}\,dx\notag\\
=&\int_\Omega\left(\nabla y_0, A^{sym}_0\nabla y_0\right)_{\mathbb{R}^N}\,dx.\label{2a.5a}
\end{align}
Thus,
\[
I_{\mathrm{min}}\ge \left\|y_0-y_d\right\|^2_{L^2(\Omega)}+\int_\Omega\left(\nabla y_0, A^{sym}_0\nabla y_0\right)_{\mathbb{R}^N}\,dx=
I(A_0,y_0),
\]
and hence, the pair $(A_0,y_0)$ is optimal for problem \eqref{2.3}--\eqref{2.3a}.  The proof is complete.
\end{proof}


\section{On variational solutions to OCP (\ref{2.3})--(\ref{2.3a}) and their approximation}
\label{Sec 4}

The question we are going to discuss in this section is about some pathological properties that can be inherited by optimal pair to the problem \eqref{2.3}--\eqref{2.3a} and other unexpected surprises concerning the approximation of the original OCP and its solutions.

To begin with, we show that the main assumption on the regularity property of OCP \eqref{2.3}--\eqref{2.3a} in Theorem \ref{Th 2.14} can be eliminated due to the approximation approach. It is clear that the condition $A^\ast\in L^2(\Omega;\mathbb{S}^N_{skew})$ ensures the existence of the sequence of skew-symmetric matrices
$\left\{A^\ast_k\right\}_{k\in \mathbb{N}}\subset L^\infty(\Omega;\mathbb{S}^N_{skew})$ such that $A^\ast_k\rightarrow A^\ast$ strongly in $L^2(\Omega;\mathbb{S}^N_{skew})$. This leads us to the idea to consider the following sequence of constrained minimization problems associated with matrices $A^\ast_k$
\begin{equation}
\label{4.0} \left\{\ \left<\inf_{(u,y)\in\Xi_k}
I_k(u,y)\right>,\quad k\to \infty \right\}.
\end{equation}
Here,
\begin{gather}
\label{4.0aa}
I_k(u,y):=I(u,y)\quad  \forall\,(u,y)\in L^2(\Omega;\mathbb{M}^N)\times H^1_0(\Omega),\quad \forall\,k\in \mathbb{N},\\[1ex]
\label{4.0a}
\Xi_k=\left\{(u,y)\ \left|\
\begin{array}{c}
-\div\big(A^{sym}\nabla y+A^{skew}\nabla y\big) = f\quad\text{in }\ \Omega,\\[1ex]
y=0\text{ on }\partial\Omega,\\[1ex]
A=A^{sym}+A^{skew}\in \mathfrak{A}^k_{ad}=\mathfrak{A}_{ad,1}\oplus \mathfrak{A}^k_{ad,2},\ y\in H^1_0(\Omega),\\[1ex]
\mathfrak{A}^k_{ad,2}=U_{a,2}\cap U_{b,2}^k,\\[1ex]
U_{b,2}^k=\left\{\left. B=[b_{i\,j}]\in L^2(\Omega;\mathbb{S}^N_{skew})\ \right|\ B(x)\preceq A^\ast_k(x)\ \text{a.e. in }\ \Omega\right\}.
\end{array}
\right.\right\}
\end{gather}

Before we will provide an accurate analysis of the optimal control problems \eqref{4.0}, we make use of the following auxiliary result.
\begin{lemma}
\label{Lemma 4.1}
The sequence of sets $\left\{U_{b,2}^k\right\}_{k\in \mathbb{N}}$ converges to $U_{b,2}$ as $k\to\infty$ in the sense of Kuratowski
with respect to the strong topology of $L^2(\Omega;\mathbb{S}^N_{skew})$.
\end{lemma}
\begin{proof}
We recall here
that a sequence $\left\{U_{b,2}^k\right\}_{k\in \mathbb{N}}$ of the subsets of $L^2(\Omega;\mathbb{S}^N_{skew})$
is said to be convergent to a closed set $S$ in the
sense of Kuratowski with respect to the strong topology of $L^2(\Omega;\mathbb{S}^N_{skew})$, if the following two properties hold:
\begin{enumerate}
\item[$(K_1)$] for every $B\in S$, there exists a sequence of matrices
$\left\{B_k\in U_{b,2}^k\right\}_{k\in \mathbb{N}}$ such that $B_k\rightarrow B$ in $L^2(\Omega;\mathbb{S}^N_{skew})$
as $k\to\infty$;
\item[$(K_2)$] if $\left\{k_n\right\}_{n\in \mathbb{N}}$ is a
sequence of indices converging to $+\infty$, $\left\{B_n\right\}_{n\in
\mathbb{N}}$ is a sequence of skew-symmetric matrices such that $B_n\in U_{b,2}^{k_n}$ for each
$n\in \mathbb{N}$, and $\left\{B_n\right\}_{n\in
\mathbb{N}}$ strongly converges in $L^2(\Omega;\mathbb{S}^N_{skew})$ to some matrix $B$, then $B\in S$.
\end{enumerate}
For the details we refer to \cite{KL}.

In order to show that $S=U_{b,2}$, we begin with the verification of $(K_2)$-item. Let $\left\{k_n\right\}_{n\in \mathbb{N}}$ be a
given sequence of indices such that $k_n\to\infty$, and let $\left\{B_n\in U_{b,2}^{k_n}\right\}_{n\in \mathbb{N}}$
be a sequence satisfying the property $B_n\rightarrow B$ in $L^2(\Omega;\mathbb{S}^N_{skew})$ and, hence, $B_n(x)\rightarrow B(x)$ almost everywhere in $\Omega$ as $n\to\infty$. By definition of $U_{b,2}^k$, we have
\begin{equation}
\label{4.1.1} B_n(x)\preceq A^\ast_{k_n}(x)\quad\text{a.e. in }\ \Omega,
\end{equation}
where $A^\ast_k\rightarrow A^\ast$ strongly in $L^2(\Omega;\mathbb{S}^N_{skew})$.
Taking into account the fact that the binary relation $\preceq$ is reflexive  and transitive,
we can pass to the limit in relation \eqref{4.1.1} as $n\to\infty$ (in the sense of almost everywhere) and get $B(x)\preceq A^\ast(x)$ almost everywhere in $\Omega$, hence, $B\in U_{b,2}$.

It remains to verify the $(K_1)$-item. To this end, we fix an arbitrary skew-symmetric matrix $B\in U_{b,2}$ and
make use of the concept of the Lebesgue set $\mathfrak{W}(B)$. We say that $x\in\Omega$ is of the Lebesgue set $\mathfrak{W}(B)$ for the matrix $B\in U_{b,2}\subset L^2(\Omega;\mathbb{S}^N_{skew})$, if $x$ is a Lebesgue point of $B$. In other words, at this point matrix $B(x)$ must be approximately continuous and, hence, it does not oscillate too much, in an average sense. It is well known that almost each point in $\Omega$ is a Lebesgue point for an absolutely locally integrable function \cite{Evans}. Hence, $\mathcal{L}^N(\Omega\setminus \mathfrak{W}(B))=0$. Moreover, since $A^\ast_k\in L^\infty(\Omega;\mathbb{S}^N_{skew})$, it follows that any point of approximate continuity of $A^\ast_k$ is its Lebesgue point \cite{Evans}.
As a result, we construct a strong convergent sequence $\left\{B_k\in U_{b,2}^k\right\}_{k\in \mathbb{N}}$ to $B\in U_{b,2}$ as follows: $B_k(x)=[b^k_{ij}(x)]_{i,j=1}^N$, where
\begin{align}
\label{4.1.2}
b^k_{ij}(x)=&\left\{
\begin{array}{ll}
b_{ij}(x), & \text{if }\ \left|b_{ij}(x)\right|\le \left|a^{\ast,k}_{ij}(x)\right|\ \text{ and }\ x\in \mathfrak{W}(B),\\
a^{\ast,k}_{ij}(x), & \text{if }\ \left|b_{ij}(x)\right|> \left|a^{\ast,k}_{ij}(x)\right|\ \text{ and }\ x\in \mathfrak{W}(B),\\
0, & \text{otherwise},
\end{array}
\right.,
\end{align}
for all $i,j\in\{1,\dots,N\}$ and $k\in \mathbb{N}$.

Since the strong convergence $A^\ast_k\rightarrow A^\ast$ in $L^2(\Omega;\mathbb{S}^N_{skew})$ implies  (up to a subsequence) the pointwise convergence a.e. in $\Omega$, and $B\preceq A^\ast$, it follows that the sequence $\left\{B_k\in U_{b,2}^k\right\}_{k\in \mathbb{N}}$, given by \eqref{4.1.2}, satisfies all properties of $(K_1)$-item. This concludes the proof.
\end{proof}

We are now in a position to study the optimal control problems \eqref{4.0}.
\begin{theorem}
\label{Th 4.1}
Let $y_d\in L^2(\Omega)$ and $f\in
H^{-1}(\Omega)$ be given distributions.
Then for every $k\in \mathbb{N}$ there exists a minimizer  $(A^0_k,y^0_k)\in\Xi_k$ to the corresponding minimization problems \eqref{4.0} such that the sequence of pairs
$\left\{(A^0_k,y^0_k)\in\Xi_k\right\}_{k\in \mathbb{N}}$
is relatively compact with respect to the $\tau$-topology on $L^2(\Omega;\mathbb{M}^N)\times H^1_0(\Omega)$ and each of its $\tau$-cluster pairs $(\widehat{A},\widehat{y})$ possesses the properties:
\begin{equation}
\label{4.2}
(\widehat{A},\widehat{y})\in\Xi,\quad [\widehat{y},\widehat{y}\,]_{\widehat{A}}\ge 0.
\end{equation}
\end{theorem}
\begin{proof}
To begin with, we show that the sequence of minimal values  for the problems \eqref{4.0} is uniformly bounded, i.e.
\begin{equation}
\label{4.2a}
\sup_{k\in \mathbb{N}}\inf_{(u,y)\in\Xi_k}I_k(u,y)\le C\quad\text{for some }\ C>0.
\end{equation}
Indeed, for each $k\in \mathbb{N}$, we obviously have $\mathfrak{A}^k_{ad,2}\ne\emptyset$ and $\mathfrak{A}^k_{ad,2}\subset L^\infty(\Omega;\mathbb{S}^N_{skew})$. Hence, for any admissible control $A_k=A_k^{sym}+A_k^{skew}\in \mathfrak{A}^k_{ad}$, we can claim that $A_k^{skew}\in L^\infty(\Omega;\mathbb{S}^N_{skew})$ and, therefore, the corresponding bilinear form
$$
[y,\varphi]_{A_k}=\int_\Omega \big(\nabla \varphi,A^{skew}_k\nabla y\big)_{\mathbb{R}^N}\,dx
$$
is bounded on $H^1_0(\Omega)$ and satisfies the identity
\[
\int_\Omega \big(\nabla \varphi,A^{skew}_k \nabla y\big)_{\mathbb{R}^N}\,dx=
-\int_\Omega \big(\nabla y,A^{skew}_k \nabla \varphi\big)_{\mathbb{R}^N}\,dx.
\]
Therefore,
\begin{equation}
\label{4.3}
\int_\Omega \big(\nabla v,A^{skew}_k(x)\nabla v\big)_{\mathbb{R}^N}\,dx=0\quad\forall\,v\in H^1_0(\Omega)
\end{equation}
and, hence, the boundary value problem \eqref{4.0a} has a unique solution $y_k\in H^1_0(\Omega)$ for each $A_k\in \mathfrak{A}^k_{ad}\subset L^\infty(\Omega;\mathbb{M}^N)$ by the Lax-Milgram lemma. As obvious consequence of this observation and the property of $\tau$-lower semicontinuity  of the cost functional $I_k$, we conclude (see for comparison Theorem~\ref{Th 2.14}):  the corresponding minimization problem \eqref{4.0} admits at least one solution \cite{Lions71}
$$
I_k(A_k^0,y_k^0)=\inf_{(A,y)\in\Xi_k}I_k(A,y),\quad (A_k^0,y_k^0)\in\Xi_k.
$$
Moreover, having fixed a control $A_k\in \mathfrak{A}^k_{ad}$, condition \eqref{4.3} implies the fulfilment of  the following identities for every $k\in \mathbb{N}$
\begin{gather}
\int_\Omega \big(\nabla \varphi,A^{sym}_k\nabla  y_k+A^{skew}_k\nabla y_k\big)_{\mathbb{R}^N}\,dx
=\left<f,\varphi\right>_{H^{-1}(\Omega);H^1_0(\Omega)},\ \forall\,\varphi\in C^\infty_0(\Omega),
\label{4.4}
\\
\label{4.5}
\int_\Omega\left(\nabla y_k,A^{sym}_k\nabla y_k\right)_{\mathbb{R}^N}\,dx
=\left<f,y_k\right>_{H^{-1}(\Omega);H^1_0(\Omega)},
\end{gather}
where $y_k=y_k(A_k,f)\in H^1_0(\Omega)$ are the corresponding solutions to the boundary value problems \eqref{4.0a}.
Hence, the sequence $\left\{y_k\right\}_{k\in \mathbb{N}}$ is bounded in $H^1_0(\Omega)$ and due to the a priori estimate
\begin{equation}
\label{4.5.0}
\|y_k\|_{H^1_0(\Omega)}\le \alpha^{-1}\|f\|_{H^{-1}(\Omega)},
\end{equation}
we arrive at the relation
\begin{align}
\notag
I_k(A_k^0,y_k^0)=&\inf_{(A,y)\in\Xi_k}I_k(A,y)\le I_k(A_k,y_k)\\
&\le 2\|y_d\|^2_{L^2(\Omega)}
+2\|y_k\|^2_{L^2(\Omega)}+\beta\|y_k\|^2_{H^1_0(\Omega)}\\
&\le 2\|y_d\|^2_{L^2(\Omega)}+(2C_1+\beta)\alpha^{-2}\|f\|^2_{H^{-1}(\Omega)}\le C
\quad\forall\,k\in \mathbb{N}.
\label{4.5.0a}
\end{align}
Thus, \eqref{4.2a} holds true and it implies that $\sup_{k\in \mathbb{N}} \|y_k^0\|^2_{H^1_0(\Omega)}<+\infty$.
So, we can suppose that the sequence of optimal states $\left\{y_k^0\right\}_{k\in \mathbb{N}}$ is weakly convergent: $y^0_k\rightharpoonup \widehat{y}$ in $H^1_0(\Omega)$. At the same time, due to the definition of the sets $\mathfrak{A}_{ad}^k$, it is easy to see that the corresponding sequence of optimal controls $\left\{A_k^0\right\}_{k\in \mathbb{N}}$ belongs to $U_{a,1}\oplus U_{b,2}$. Hence, applying the arguments of the proof of Proposition~\ref{Prop 2.3f}, we get: there exists a matrix $\widehat{A}\in U_{a,1}\oplus U_{b,2}$ such that
\begin{align}
\label{4.5.1}
A^0_k:=A_k^{0,sym}+A_k^{0,skew}&\rightarrow \widehat{A}^{sym}+\widehat{A}^{skew}=:\widehat{A}\ \text{ in }\ L^2(\Omega;\mathbb{M}^N),\\
\label{4.5.1a}
A_k^{0,sym}&\rightarrow \widehat{A}^{sym}\quad\text{in }\ L^p(\Omega;\mathbb{S}^N_{sym}),\ \forall\,p\in [1,+\infty),\\
\label{4.5.1c}
A_k^{0,sym}&\stackrel{\ast}{\rightharpoonup}\, \widehat{A}^{sym}\quad\text{in }\ L^\infty(\Omega;\mathbb{S}^N_{sym}),\\
\label{4.5.1b} A_k^{0,skew}&\rightarrow \widehat{A}^{skew}\quad\text{in }\ L^2(\Omega;\mathbb{S}^N_{skew}).
\end{align}
Therefore, in view of Lemma~\ref{Lemma 4.1}, we can conclude: $\widehat{A}\in \mathfrak{A}_{ad}$. As a result, summing up the above properties
of the sequences $\left\{y_k^0\right\}_{k\in \mathbb{N}}$ and $\left\{A_k^0\right\}_{k\in \mathbb{N}}$, we obtain:
$$
(A_k^0,y_k^0)\,\stackrel{\tau}{\rightarrow}\, (\widehat{A},\widehat{y}\,)\ \text{ and }\ (\widehat{A},\widehat{y}\,)\in\Xi.
$$

It remains to prove the properties \eqref{4.2}. To do so, we note that
due to the strong convergence $A^0_k\rightarrow \widehat{A}$ in $L^2(\Omega;\mathbb{M}^N)$, we get
\begin{align*}
&\left|\int_\Omega \left(\nabla\varphi,\widehat{A}\nabla \widehat{y}-A^0_k \nabla y^0_k\right)_{\mathbb{R}^N}\,dx\right|\\\le\ &
\int_\Omega \|A^0_k-\widehat{A}\|_{\mathbb{M}^N} \|\nabla y^0_k\|_{\mathbb{R}^N} \|\nabla\varphi\|_{\mathbb{R}^N}\,dx+
\left|\int_\Omega \left(\widehat{A}\nabla\varphi, \nabla \widehat{y}-\nabla y^0_k\right)_{\mathbb{R}^N}\,dx\right|\\ \le\ &
\|\varphi\|_{C^1(\overline{\Omega})}\sup_{k\in \mathbb{N}}\|y^0_k\|_{H^1_0(\Omega)}\|A^0_k-\widehat{A}\|_{L^2(\Omega;\mathbb{M}^N)}\\
&+ \left|\int_\Omega \left(\widehat{A}\nabla\varphi, \nabla \widehat{y}\right)_{\mathbb{R}^N}\,dx  -\int_\Omega\left(\widehat{A}\nabla\varphi, \nabla y^0_k\right)_{\mathbb{R}^N}\,dx\right|\longrightarrow 0\quad\text{as}\ k\to\infty
\end{align*}
for every $\varphi\in C^\infty_0(\Omega)$.
Hence, $A^0_k\nabla y^0_k \stackrel{\ast}{\rightharpoonup}\,\widehat{A}\nabla \widehat{y}$ in $L^1(\Omega;\mathbb{R}^N)$. It means that we can pass to the limit in integral identity \eqref{4.4} with $A=A_k^0$. As a result, we have: the pair $(\widehat{A},\widehat{y}\,)$ is related by the integral identity \eqref{2.5}, therefore, $\widehat{y}$ is a weak solution to the original boundary value problem \eqref{2.1}--\eqref{2.2} under $A=\widehat{A}$. Thus, $\widehat{y}\in D(\widehat{A})$ and, hence, $(\widehat{A},\widehat{y}\,)\in\Xi$.

In order to proof the property \eqref{4.2}$_2$, we
pass to the limit in the energy equality \eqref{4.5} using the lower semicontinuity of the norm $\|\cdot\|_{H^1(\Omega)}$ with respect to the weak convergence
$\nabla y^0_k \rightharpoonup \nabla \widehat{y}$ in $L^2(\Omega;\mathbb{R}^N)$ and the property \eqref{4.5.1c}.
To do so, we note that due to the inclusion $\widehat{A}^{sym}\in \mathfrak{A}_{ab,1}$, we have $A\in \mathfrak{M}_\alpha^\beta(\Omega)$.
Hence, the norms $\|y\|_{H^1(\Omega)}$ and $|||y|||:=\left(\int_\Omega\left(\nabla y,\widehat{A}^{sym}\nabla y^0_k\right)_{\mathbb{R}^N}\,dx\right)^{1/2}$ are equivalent in $H^1_0(\Omega)$.
As a result, we obtain
\begin{equation}
\label{4.6}
\begin{split}
\left<f,\widehat{y}\,\right>_{H^{-1}(\Omega);H^1_0(\Omega)}=&
\lim_{k\to\infty} \int_\Omega\left(\nabla y^0_k,(A^{0,sym}_k-\widehat{A}^{sym})\nabla y^0_k\right)_{\mathbb{R}^N}\,dx\\
&+\lim_{k\to\infty} \int_\Omega\left(\nabla y^0_k,\widehat{A}^{sym}\nabla y^0_k\right)_{\mathbb{R}^N}\,dx\\
&\stackrel{\text{by \eqref{4.5.1a}}}{=}\,\lim_{k\to\infty} \int_\Omega\left(\nabla y^0_k,\widehat{A}^{sym}\nabla y^0_k\right)_{\mathbb{R}^N}\,dx\\
&\stackrel{\text{by \eqref{1.1}}}{\ge}\,\int_\Omega\left(\nabla \widehat{y},\widehat{A}^{sym}\nabla \widehat{y}\right)_{\mathbb{R}^N}\,dx.
\end{split}
\end{equation}
Thus, the desired inequality \eqref{4.2}$_2$ obviously follows from \eqref{2.12} and \eqref{4.6}. The proof is complete.
\end{proof}
\begin{remark}
\label{Rem 4.7} As Theorem \ref{Th 4.1} proves, for any approximation $\left\{A^\ast_k\right\}_{k\in \mathbb{N}}$ of the  matrix $A^\ast\in L^2\big(\Omega;\mathbb{S}^N_{skew}\big)$ with properties $\left\{A^\ast_k\right\}_{k\in \mathbb{N}}\subset L^\infty(\Omega;\mathbb{S}^N_{skew})$ and
$A^\ast_k\rightarrow A^\ast$ strongly in $L^2(\Omega;\mathbb{S}^N_{skew})$, optimal solutions to the regularized OCPs \eqref{4.0}--\eqref{4.0a}  always lead us in the limit to some admissible (but not optimal in general) solution $(\widehat{A},\widehat{y}\,)$ of the original OCP \eqref{2.3}--\eqref{2.3a}. Moreover, this limit pair can depend on the choice of the approximative sequence $\left\{A^\ast_k\right\}_{k\in \mathbb{N}}$. It is reasonably  to call such pairs attainable admissible solutions to OCP. However, the entire structure of the set of all attainable solutions remain unclear; for instance, it is not known whether this set is convex and closed in $\Xi$. It is also unknown whether the optimal solution to OCP \eqref{2.3}--\eqref{2.3a} is attainable.
At the end of this section we give the conditions on the  matrix $A^\ast$ which ensures the attainability of optimal solutions to the original OCP.
\end{remark}

Taking these observations into account, we make use of the following notion.
\begin{definition}
\label{Def 4.8} We say that a pair $(\widehat{A},\widehat{y}\,)\in L^2(\Omega;\mathbb{M}^N)\times H^1_0(\Omega)$ is a variational solution to OCP \eqref{2.3}--\eqref{2.3a} if
there exists an approximation $\left\{A^\ast_k\right\}_{k\in \mathbb{N}}\subset L^\infty(\Omega;\mathbb{S}^N_{skew})$ of the  matrix $A^\ast\in L^2\big(\Omega;\mathbb{S}^N_{skew}\big)$ with property
$A^\ast_k\rightarrow A^\ast$ strongly in $L^2(\Omega;\mathbb{S}^N_{skew})$ such that
\begin{gather}
\label{4.9}
I(\widehat{A},\widehat{y}\,)=\inf_{(A,y)\in\Xi}I(A,y), \quad (\widehat{A},\widehat{y}\,)\in \Xi, \text{ and }\\
\label{4.10}
\left<\inf\limits_{(A,y)\in\,\Xi_k} I_k(A,y)\right>\,\stackrel{\text{Var}}{\xrightarrow[k\to\infty]{}}\,
\left<\inf_{(A,y)\in\Xi} I(A,y)\right>\ \text{ in the sense of Definition \ref{Def 1.4}},
\end{gather}
where the minimization problems $\left<\inf\limits_{(A,y)\in\,\Xi_k} I_k(A,y)\right>$ are defined by \eqref{4.0aa}--\eqref{4.0a}.
\end{definition}

As a direct consequence of Definition \ref{Def 4.8}, Theorem \ref{Th 4.1}, and properties of the variational limits of constrained minimization problems (see Theorem \ref{Th 1.8}), we have the following result.
\begin{proposition}
\label{Prop 4.13}
Let $(\widehat{A},\widehat{y}\,)\in L^2(\Omega;\mathbb{M}^N)\times H^1_0(\Omega)$ be a variational solution to OCP \eqref{2.3}--\eqref{2.3a}. Then
$[\widehat{y},\widehat{y}]_{\widehat{A}}= 0$ and the pair $(\widehat{A},\widehat{y}\,)$ can be attained by optimal solutions $(A_k^0,y_k^0)$ to the regularized OCPs \eqref{4.0}--\eqref{4.0a} as follows
\begin{equation}
\label{4.14}
\left.
\begin{array}{c}
A_k^0\rightarrow \widehat{A}\ \text{strongly in }\ L^2(\Omega;\mathbb{M}^N),\\[1ex]
y_k^0\rightharpoonup \widehat{y}\ \text{weakly in }\ H^1_0(\Omega)\ \text{ as } k\to\infty,\\[1ex]
\ds \lim_{k\to\infty} \int_\Omega\left(\nabla y^0_k,A^{0,sym}_k\nabla y^0_k\right)_{\mathbb{R}^N}\,dx=
\int_\Omega\left(\nabla \widehat{y},\widehat{A}^{sym}\nabla \widehat{y}\right)_{\mathbb{R}^N}\,dx.
\end{array}
\right\}
\end{equation}
\end{proposition}
\begin{proof}
Indeed,  in view of a priori estimates \eqref{4.5.0}--\eqref{4.5.0a} and properties \eqref{4.5.1}--\eqref{4.5.1b}, within a subsequence, we have
\begin{align}
\label{4.15} y^0_k&\rightharpoonup \widehat{y}\ \text{ in }\ H^1_0(\Omega),\\
A^0_k:=A_k^{0,sym}+A_k^{0,skew}&\rightarrow \widehat{A}^{sym}+\widehat{A}^{skew}=:\widehat{A}\ \text{ in }\ L^2(\Omega;\mathbb{M}^N),\\
A_k^{0,sym}&\rightarrow \widehat{A}^{sym}\quad\text{in }\ L^p(\Omega;\mathbb{S}^N_{sym}),\ \forall\,p\in [1,+\infty).\label{4.15'}
\end{align}
On the other hand, following  main properties of the variational convergence (see Theorem \ref{Th 1.8}), we can claim that
there exists an optimal pair $(A^0,y^0)\in \Xi$ for the problem \eqref{2.3}--\eqref{2.3a} such that
\begin{align}
\inf_{(A,y)\in\,\Xi}I(A,y)=& I\left(A^0,y^0\right):=\left\|y^0-y_d\right\|^2_{L^2(\Omega)}+\int_\Omega\left(\nabla y^0, A^{0,sym}\nabla y^0\right)_{\mathbb{R}^N}\,dx
 \notag\\
 =&\lim_{k\to \infty}\inf_{(A_k,y_k)\in\,{\Xi}_k}
{I}_k(A_k,y_k)=\lim_{k\to \infty}{I}_k(A^0_k,y^0_k)\notag\\
=&\lim_{k\to\infty} \left[\left\|y^0_k-y_d\right\|^2_{L^2(\Omega)}+
\int_\Omega\left(\nabla y^0_k, A^{0,sym}_k\nabla y^0_k\right)_{\mathbb{R}^N}\,dx
\right].\label{4.16'}
\end{align}
However, because of condition \eqref{4.15}--\eqref{4.15'}, it turns out that (see the estimate \eqref{2a.5a})
\begin{gather*}
\inf_{(A,y)\in\,\Xi}I(A,y)\,\stackrel{\text{by \eqref{4.16'}}}{=}\,\lim_{k\to\infty}\left[\left\|y^0_k-y_d\right\|^2_{L^2(\Omega)}+\int_\Omega\left(\nabla y^0_k, A^{0,sym}_k\nabla y^0_k\right)_{\mathbb{R}^N}\,dx\right]\\
\ge
\left\|\widehat{y}-y_d\right\|^2_{L^2(\Omega)}+\int_\Omega\left(\nabla \widehat{y}, \widehat{A}^{sym}\nabla \widehat{y}\right)_{\mathbb{R}^N}\,dx.
\end{gather*}
Since the pair $(\widehat{A},\widehat{y}\,)$ is admissible for the problem \eqref{2.3}--\eqref{2.3a} (see Theorem~\ref{Th 4.1}), it follows that $(\widehat{A},\widehat{y}\,)$ is an optimal pair, that is, in view of \eqref{4.16'}, it gives
\begin{align}
\inf_{(A,y)\in\,\Xi}I(A,y)=& I\left(\widehat{A},\widehat{y}\right):=\left\|\widehat{y}-y_d\right\|^2_{L^2(\Omega)}+\int_\Omega\left(\nabla \widehat{y}, \widehat{A}^{sym}\nabla \widehat{y}\right)_{\mathbb{R}^N}\,dx
 \notag\\
 =&\lim_{k\to \infty}\inf_{(A_k,y_k)\in\,{\Xi}_k}
{I}_k(A_k,y_k)=\lim_{k\to \infty}{I}_k(A^0_k,y^0_k)\notag\\
=&\lim_{k\to\infty} \left[\left\|y^0_k-y_d\right\|^2_{L^2(\Omega)}+ \int_\Omega\left(\nabla y^0_k, A^{0,sym}_k\nabla y^0_k\right)_{\mathbb{R}^N}\,dx\right].\label{4.16}
\end{align}
Hence, \eqref{4.14} is a direct consequence of properties \eqref{4.15}--\eqref{4.16}.
As a result, we get
\begin{align*}
0\,\stackrel{\text{by \eqref{4.3}}}{=}&\,\lim_{k\to\infty}[y^0_k,y^0_k]_{A^0_k}
\stackrel{\text{by \eqref{4.5}}}{=}\,-\lim_{k\to\infty}\int_\Omega\left(\nabla y^0_k,A^{0,sym}_k\nabla y^0_k\right)_{\mathbb{R}^N}\,dx\\
&+\lim_{k\to\infty}\left<f,y^0_k\right>_{H^{-1}(\Omega);H^1_0(\Omega)}
\,\stackrel{\text{by \eqref{4.14} and \eqref{4.16}}}{=}\,-\int_\Omega\left(\nabla \widehat{y},\widehat{A}^{sym}\nabla \widehat{y}\right)_{\mathbb{R}^N}\,dx\\&+
\left<f,\widehat{y}\,\right>_{H^{-1}(\Omega);H^1_0(\Omega)}
\,\stackrel{\text{by \eqref{2.12}}}{=}\,[\widehat{y},\widehat{y}]_{\widehat{A}}.
\end{align*}
\end{proof}

\begin{remark}
\label{Rem 4.17}
Since for some matrices $A\in L^2\big(\Omega;\mathbb{M}^N\big)$ the weak solutions to the boundary value problem \eqref{2.1}--\eqref{2.2} are not unique in general, it follows from Remark \ref{Rem 4.7} and Proposition \ref{Prop 4.13} that even if the OCP \eqref{2.3}--\eqref{2.3a} has a unique solution $(A^0,y^0)$ and this solution  possesses the property $[y^0,y^0]\ge 0$, it does not ensure that the pair $(A^0,y^0)$ is the variational solution to the above problem. Let $A\in \mathfrak{A}_{ad}$ be a fixed matrix and let
$L(A)$ be a subspace of $H^1_0(\Omega)$ such that
\begin{equation}
\label{4.17aa}
L(A)=\left\{h\in D(A)\ :\ \int_\Omega \big(\nabla \varphi,A\nabla h\big)_{\mathbb{R}^N}\,dx=0\ \forall\,\varphi \in C^\infty_0(\mathbb{R}^N)\right\},
\end{equation}
i.e., $L(A)$ is the set of all weak solutions of the homogeneous problem
\begin{equation}
\label{4.17a}
\begin{array}{c}
-\div\big(A\nabla y\big) = 0\quad\text{in }\ \Omega,\\[1ex]
y=0\text{ on }\partial\Omega.
\end{array}
\end{equation}
Since $L(A)$ can contain non-trivial elements in general, it follows that the set
$$\Lambda:=\left\{(A^0,y^0+h)\ \forall\,h\in L(A^0)\right\}$$
is not a singleton in $\Xi$.

Let  $\left\{(A_k^0,y_k^0)\right\}_{k\in \mathbb{N}}$ be optimal solutions to the regularized OCPs \eqref{4.0}--\eqref{4.0a}. Let $A^0$ be a strong limit in $L^2(\Omega;\mathbb{M}^N)$ of  $\left\{A_k^0\right\}_{k\in \mathbb{N}}$. Then Theorem \ref{Th 4.1} implies that $y^0_k\rightharpoonup y^\ast$ in $H^1_0(\Omega)$ and $(A^0,y^\ast)\in\Xi$. However, it means that $(A^0,y^\ast)\in \Lambda$ rather than $y^\ast=y^0$.
Moreover, the existence of  $(\Gamma,\delta)$-realizing sequence (see Definition \ref{Def 1.4}) for the pair $(A^0,y^0)\in\Xi$ is an open problem. In other words, the existence at least one approximation $\left\{A^\ast_k\right\}_{k\in \mathbb{N}}\subset L^\infty(\Omega;\mathbb{S}^N_{skew})$ in \eqref{4.0}--\eqref{4.0a} leading to the pair $(A^0,y^0)$ in the sense of conditions \eqref{1.7a}--\eqref{1.7b} is not established. As follows from our further analysis (see Section \ref{Sec 5}), such solutions can be attained through other structures of regularized OCPs than in \eqref{4.0}--\eqref{4.0a}.
\end{remark}

We are now in a position to discuss the existence of variational solutions to the OCP \eqref{2.3}--\eqref{2.3a}.
\begin{theorem}
\label{Th 4.18}
Assume that for every  matrix $A\in \mathfrak{A}_{ad}\subset L^2\big(\Omega;\mathbb{M}^N\big)$, we have
\begin{equation}
\label{4.19}
[y,y]_{A}=0\quad\forall\,y\in  D(A).
\end{equation}
Then  the OCP \eqref{2.3}--\eqref{2.3a} has variational solutions.
\end{theorem}
\begin{proof}
Let us consider $\left\{A^\ast_k\right\}_{k\in \mathbb{N}}\subset L^\infty(\Omega;\mathbb{S}^N_{skew})$ and $A^\ast\in L^2\big(\Omega;\mathbb{S}^N_{skew}\big)$ such that $A^\ast_k\rightarrow A^\ast$ strongly in $L^2(\Omega;\mathbb{S}^N_{skew})$. With each matrix $A^\ast_k$ we associate the constrained minimization problem
$$\left<\inf_{(A,y)\in\Xi_k} I_k(A,y)\right>,$$
where the cost functional $I_k$ and the set $\Xi_k$ are defined by \eqref{4.0aa}--\eqref{4.0a}.

Let $\left\{(A_k,y_k)\right\}_{k\in \mathbb{N}}$ be a sequence in $L^2(\Omega;\mathbb{M}^N)\times H^1_0(\Omega)$ with the
following properties:
\begin{enumerate}
\item[(a)] $(A_k,y_k)\in \Xi_{n_k}$ for every $k\in \mathbb{N}$, where $\{n_k\}_{k\in \mathbb{N}}$ is a subsequence
converging to $\infty$ as $k$ tends to $\infty$;
\item[(aa)] $y_k\rightharpoonup y$ in $H^1_0(\Omega)$ and $A_k\rightarrow A$ in $L^2(\Omega;\mathbb{M}^N)$ with additional properties as in \eqref{4.5.1a}--\eqref{4.5.1c}.
\end{enumerate}
Then proceeding as in the proof of Theorem \ref{Th 4.1}, it can be shown that the limit pair $(A,y)$ is admissible to the original OCP \eqref{2.3}--\eqref{2.3a}. Hence, this problem is regular and, therefore, it is solvable by Theorem \ref{Th 2.14}.  Our aim is to show that this problem can be interpreted as the variational limit of the sequence of constrained minimization problems \eqref{4.0}. To do so, we have to verify the fulfilment of all conditions of Definition~\ref{Def 1.4}.

Indeed,  it is easy to see that in the case of space $L^2(\Omega;\mathbb{M}^N\big)\times H^1_0(\Omega)$, the condition (d) is obviously true with $\delta=0$.
As for the property (dd), it immediately follows from the following relation
\begin{align*}
\liminf_{k\to\infty}I_k(A_k,y_k)=\ &\liminf_{k\to\infty}\left[\left\|y_k-y_d\right\|^2_{L^2(\Omega)}+\int_\Omega\left(\nabla y_k, A^{sym}_k\nabla y_k\right)_{\mathbb{R}^N}\,dx\right]\\
&\stackrel{\text{by \eqref{2a.5a}}}{\ge}\, \left\|y-y_d\right\|^2_{L^2(\Omega)}+\int_\Omega\left(\nabla y, A^{sym}\nabla y\right)_{\mathbb{R}^N}\,dx
 = I(A,y),
\end{align*}
which holds true for any sequence $\left\{(A_k,y_k)\right\}_{k\in \mathbb{N}}\subset \mathfrak{A}_{ad}\times H^1_0(\Omega)$ with properties (a)--(aa).

We focus now on the verification of condition (ddd) of Definition \ref{Def 1.4}.
Let $(A^\sharp,y^\sharp)$ be an arbitrary admissible pair to the original problem.
Since $A^{\sharp,skew}\preceq A^\ast$, we make use of the hint of Lemma~\ref{Lemma 4.1} in order to construct a sequence of admissible controls $\left\{A_k\in \mathfrak{A}_{ad}^k\subset L^2(\Omega;\mathbb{M}^N)\right\}_{k\in \mathbb{N}}$.
Namely, we proceed as follows. Let $A^\ast_k=[a^{\ast,k}_{ij}(x)]_{i,j=1}^N$ and $A^{\sharp,skew}=[a^\sharp_{ij}(x)]_{i,j=1}^N$. Then we set:
\begin{equation}
\label{4.18c}
A_k^{sym}=A^{\sharp,sym}\,\ \forall\,k\in \mathbb{N}\quad  \text{for the symmetric parts of }\ A_k,
\end{equation}
and for the skew symmetric parts $A^{skew}_k(x)=[a^k_{ij}(x)]_{i,j=1}^N$, we put
\begin{gather}
\label{4.18a}
a^k_{ij}(x)=\left\{
\begin{array}{ll}
a^\sharp_{ij}(x), & \text{if }\ \left|a^\sharp_{ij}(x)\right|\le \left|a^{\ast,k}_{ij}(x)\right|\ \text{ and }\ x\in \mathfrak{W}(A^{\sharp,skew}),\\
a^{\ast,k}_{ij}(x), & \text{if }\ \left|a^\sharp_{ij}(x)\right|> \left|a^{\ast,k}_{ij}(x)\right|\ \text{ and }\ x\in \mathfrak{W}(A^{\sharp,skew}),\\
0, & \text{otherwise},
\end{array}
\right.
\end{gather}
for all $i,j\in\{1,\dots,N\}$ and $k\in \mathbb{N}$.

Since the strong convergence $A^\ast_k\rightarrow A^\ast$ in $L^2(\Omega;\mathbb{S}^N_{skew})$ implies (up to a subsequence) the pointwise convergence of this sequence a.e. in $\Omega$, and $A^{\sharp,skew}\preceq A^\ast$, it follows that conditions \eqref{4.18c}--\eqref{4.18b} lead us to the following conclusion:
\begin{align}
\label{4.19a}
A_k:=A_k^{sym}+A_k^{skew}&\rightarrow {A}^{\sharp,sym}+{A}^{\sharp,skew}=:{A}^\sharp\ \text{ in }\ L^2(\Omega;\mathbb{M}^N),\\
\label{4.19b}
A_k^{sym}&\rightarrow {A}^{\sharp,sym}\quad\text{in }\ L^p(\Omega;\mathbb{S}^N_{sym}),\ \forall\,p\in [1,+\infty),\\
\label{4.19d} A_k^{skew}&\rightarrow {A}^{\sharp,skew}\quad\text{in }\ L^2(\Omega;\mathbb{S}^N_{skew}).
\end{align}
Let
$\left\{y_k={y}(A_k,f)\right\}_{k\in \mathbb{N}}$ be the corresponding solutions to the regularized boundary value problems \eqref{4.0a}. Then by applying the arguments of the proof of Theorem \ref{Th 4.1}, it can be shown that the sequence $\left\{y_k\right\}_{k\in \mathbb{N}}$ is uniformly bounded in $H^1_0(\Omega)$ and there exists an element $\widehat{y}\in D(A^\sharp)$ such that $(A^\sharp,\widehat{y})\in \Xi$ and, within a subsequence,  $y_k\rightharpoonup \widehat{y}$ in $H^1_0(\Omega)$. Our aim is to show that $\widehat{y}=y^\sharp$ and that the following identity
\begin{equation}
\label{4.20}
I(A^\sharp,y^\sharp)= \limsup_{k\to \infty} I_{k}(A_k,y_k)
\end{equation}
holds true.

Indeed, since $(A^\sharp,y^\sharp)\in\Xi$ and $(A^\sharp,\widehat{y})\in \Xi$, it follows that $y=y^\sharp-\widehat{y}$ is a solution of the homogeneous problem \eqref{4.17a}. Following our initial assumptions, we have $[y,y]_A=0$ $\forall\,y\in  D(A)$ and for  each matrix $A\in \mathfrak{A}_{ad}\subset L^2\big(\Omega;\mathbb{M}^N\big)$.
Hence, the problem \eqref{4.17a} has only trivial solution, since for this solution we have
\[
\int_{\Omega}\big(\nabla y,A^{\sharp,sym}\nabla y\big)_{\mathbb{R}^N}\,dx=-[y,y]_{A^\sharp}=0.
\]
Thus, $y^\sharp=\widehat{y}$. To prove the equality \eqref{4.20}, we  use of the idea of D.Cioranescu and F.Murat  (see \cite{8}).
Taking into account the property \eqref{4.19a}, compactness of the embedding $H^1_0(\Omega)\hookrightarrow L^2(\Omega)$, and the energy identities \eqref{4.5} and \eqref{2.12}, we get
\begin{multline*}
\lim_{k\to \infty} I_{k}(A_k,y_k)=\lim_{k\to \infty}\left[\left\|y_k-y_d\right\|^2_{L^2(\Omega)}+\int_\Omega\left(\nabla y_k, A^{sym}_k\nabla y_k\right)_{\mathbb{R}^N}\,dx\right]
 \\
 =\left\|y^\sharp-y_d\right\|^2_{L^2(\Omega)}+\lim_{k\to \infty}\int_\Omega\left(\nabla y_k, A^{sym}_k\nabla y_k\right)_{\mathbb{R}^N}\,dx \\
 \stackrel{\text{by \eqref{4.5} and \eqref{4.19a}}}{=}\,\left\|y^\sharp-y_d\right\|^2_{L^2(\Omega)}+\lim_{k\to \infty}\left[\left<f,\widehat{y}_k\right>_{H^{-1}(\Omega);H^1_0(\Omega)}\right]\\
 \stackrel{\text{by \eqref{4.19}}}{=}\,\left\|y^\sharp-y_d\right\|^2_{L^2(\Omega)}+
 \left<f,y^\sharp\right>_{H^{-1}(\Omega);H^1_0(\Omega)}-[y^\sharp,y^\sharp]_{A^\sharp}\\
 \stackrel{\text{by \eqref{2.12}}}{=}\,\left\|y^\sharp-y_d\right\|^2_{L^2(\Omega)}+\int_\Omega\left(\nabla y^\sharp, A^{\sharp,sym}\nabla y^\sharp\right)_{\mathbb{R}^N}\,dx
  = I(A^\sharp,y^\sharp).
\end{multline*}
This concludes the proof.
\end{proof}

Our next observation shows that variational solutions do not exhaust the entire set of all possible solutions to the original OCP \eqref{2.3}--\eqref{2.3a}.
\begin{lemma}
\label{Lemma 4.21}
Assume that there exists a matrix $A_0\in \mathfrak{A}_{ad}$ and an element $v\in D(A_0)$ with property $[v,v]_{A_0}<0$. Then there are distributions
$f\in H^{-1}(\Omega)$ and $y_d\in L^2(\Omega)$ such that  the optimal control problem
\begin{gather}
\label{4.21a}
\text{Minimize } I(A,y)=\left\|y-y_d\right\|^2_{L^2(\Omega)}+\int_\Omega\left(\nabla y-\nabla y_d, A^{sym}(\nabla y-\nabla y_d)\right)_{\mathbb{R}^N}\,dx\\
\label{4.21b}
\text{subject to the constraints \eqref{2.1}--\eqref{2.2} and $A\in \mathfrak{A}_{ad}\subset L^2(\Omega;\mathbb{M}^N)$}
\end{gather}
has a non-variational solution in the sense of Definition~\ref{Def 4.8}.
\end{lemma}
\begin{proof}
We consider the OCP \eqref{4.21a}--\eqref{4.21b} with
\[
y_d=v\quad\text{and}\quad f=-\mathrm{div}\, A_0\nabla v.
\]
Since $v\in D(A_0)$, it follows that $v\in H^1_0(\Omega)$  and  $f\in H^{-1}(\Omega)$. It is easy to see that $y_d$ is a solution to the boundary value problem \eqref{2.1}--\eqref{2.2} under $A=A_0$. Moreover, since
$I(A_0,y_d)=0$. it follows that $(A_0,y_d)$ is the optimal pair to the above OCP.

By contradiction, we assume that $(A_0,y_d)$ is the variational solution.
As follows from Theorem \ref{Th 4.1} (see also Remark \ref{Rem 4.7}), each attainable solution $(A^\sharp,y^\sharp)$ to this OCP satisfies the inequality $[y^\sharp,y^\sharp]_{A^\sharp}\ge 0$. Since $[y_d,y_d]_{A_0}:=[v,v]_{A_0}<0$, it means that the pair $(A_0,y_d)$ is not attainable and we come into conflict with Definition \ref{Def 4.8}. The proof is complete.
\end{proof}

Taking this result into account, we adopt the following concept.
\begin{definition}
\label{Def 4.11} We say that a pair $(A^0,y^0)\in \Xi$ is a non-variational solution to OCP \eqref{2.3}--\eqref{2.3a} if
\begin{equation}
\label{4.12}
I(A^0,y^0)=\inf_{(A,y)\in\Xi}I(A,y),\quad (A^0,y^0)\in\Xi,\quad\text{and}\quad [y^0,y^0]_{A^0}\ne 0.
\end{equation}
\end{definition}

\begin{remark}
\label{Rem 4.22} As follows from Theorem\ref{Th 4.1}, Proposition \ref{Prop 4.13}, and Lemma \ref{Lemma 4.21} none of non-variational solutions can be attainable through the limit of optimal solutions to the regularized problem \eqref{4.0}--\eqref{4.0a}.
\end{remark}


\section{Optimality Conditions}
\label{Sec 33}

We consider the extremal problem \eqref{2.3}--\eqref{2.3a}, where, as above, the set of admissible pairs $\Xi$
is defined by relation
\begin{align*}
\Xi=&\Big\{(A,y)\in \mathfrak{A}_{ad}\times D(A)\subset L^2(\Omega;\mathbb{M}^N)\times H^1_0(\Omega)\ :\ I(A,y)<+\infty,\\
&\int_\Omega \big(\nabla \varphi,A^{sym}\nabla y+A^{skew}\nabla y\big)_{\mathbb{R}^N}\,dx
=\left<f,\varphi\right>_{H^{-1}(\Omega);H^1_0(\Omega)},\quad\forall\,\varphi\in C^\infty_0(\Omega)\Big\}.
\end{align*}
As usual, we determine a solution  $(A_0,y_0)\in\Xi$ to the problem \eqref{2.3}--\eqref{2.3a} as follows
\[
I(A_0,y_0)=\inf_{(A,y)\in\Xi}I(A,y).
\]
To derive the optimality conditions for optimal control problem \eqref{2.3}--\eqref{2.3a}, we set
$F(A,y)=-\div\big(A(x)\nabla y\big)$ and consider the Lagrange functional
\begin{equation}
\label{33.0}
L(A,y,\lambda,\varphi)=\lambda I(A,y)+\left<F(A,y),\varphi\right>_{H^{-1}(\Omega);H^1_0(\Omega)}-\langle
f,\varphi\rangle_{H^{-1}(\Omega);H^1_0(\Omega)},
\end{equation}
where $\lambda\in \mathbb{R}_{+}$, $\varphi\in C^\infty_0(\Omega)$, and
\begin{equation*}
\left<F(A,y),\varphi\right>_{H^{-1}(\Omega);H^1_0(\Omega)}=\int_\Omega \big(\nabla \varphi,A^{sym}\nabla y\big)_{\mathbb{R}^N}\,dx + \int_\Omega \big(\nabla \varphi,A^{skew}\nabla y\big)_{\mathbb{R}^N}\,dx.
\end{equation*}
\begin{remark}
\label{Rem 33.1}
It is worth to note that, in view of Remark~\ref{Rem 2.11}, the Lagrange functional $L(A,y,\lambda,\varphi)$, given by \eqref{33.0}, can be extended over all $\varphi\in H^1_0(\Omega)$ by continuity provided we apply the rule \eqref{2.8a}. As a result, for each $A\in \mathfrak{A}_{ad}$, the extended functional
$\widehat{L}: \mathfrak{A}_{ad}\times D(A)\times \mathbb{R}_{+}\times H^1_0(\Omega)\rightarrow \mathbb{R}$
takes the form
\begin{equation}
\label{33.1}
\widehat{L}(A,y,\lambda,p)=\lambda I(A,y)+\int_\Omega \big(\nabla p,A^{sym}\nabla y\big)_{\mathbb{R}^N}\,dx +[y,p\,]_A-\langle
f,p\rangle_{H^{-1}(\Omega);H^1_0(\Omega)}.
\end{equation}
Hence, because of unboundedness of the bilinear form $\int_\Omega \big(\nabla p,A^{skew}(x)\nabla y\big)_{\mathbb{R}^N}\,dx$,
the extended Lagrangian $\widehat{L}(A,y,\lambda,p)$ is not
G\^{a}teaux differentiable, in general. Moreover, we cannot even assert that the mapping $y\mapsto \widehat{L}(A,y,\lambda,p)$ has a 'right hand' directional derivative \cite{Phelps}
\[
\mathcal{D}^{+}_y\,\widehat{L}(A,y,\lambda,p,h):=\lim_{\theta\to{+}0} \frac{\widehat{L}(A,y+\theta h,\lambda,p)-\widehat{L}(A,y,\lambda,p)}{\theta}.
\]
Indeed, for given
$h\in H^1_0(\Omega)$ and $\theta\in[0,1]$, since the structure of the set $D(A)$ is unknown, we cannot claim that $y+\theta h$ is still an element of $D(A)$ even if $\theta$ is small enough. Hence, the term $[y+\theta h,p]_A$ can be undefined, in general.
\end{remark}
\begin{remark}
\label{Rem 33.2}
In view of Remark~\ref{Rem 33.1}, the characteristic feature of extremal problem \eqref{2.3}--\eqref{2.3a} is the fact that the set of admissible solutions $\Xi$ can contain pairs $(A,y)\in \Xi$ such that $A^{skew}\not\in L^\infty(\Omega;\mathbb{S}^N_{skew})$, $A^{skew}\in L^2(\Omega;\mathbb{S}^N_{skew})$, and, hence,  the mapping $y\mapsto F(A,y)$ is not continuously differentiable in any neighbourhood of $y$. As a result, in order to deduce an optimality system for the problem \eqref{2.3}--\eqref{2.3a}, we cannot apply the well-known results of Ioffe and Tikhomirov (see \cite{Fursik,IoTi}).
\end{remark}

In order to avoid this option, we make use of the following concepts.
\begin{definition}
\label{Def 33.2}
We say that the mapping $y\mapsto \widehat{L}(A,y,\lambda,p)$ has a generalized 'right hand' directional derivative at the point $(A,y,\lambda,p)\in \mathfrak{A}_{ad}\times D(A)\times \mathbb{R}_{+}\times H^1_0(\Omega)$ with respect to $y$ in the direction $h\in H^1_0(\Omega)$ if
the 'right hand' directional derivative $\mathcal{D}^{+}_y\,\widehat{L}(A,y,\lambda,p,\varphi)$ in a smooth direction $\varphi\in C^\infty_0(\Omega)$
can be extended by continuity for $\varphi=h\in H^1_0(\Omega)$, that is,
\[
\mathcal{D}^{+}_y\,\widehat{L}(A,y,\lambda,p,h)=\lim_{\e\to 0} \mathcal{D}^{+}_y\,\widehat{L}(A,y,\lambda,p,\varphi_\e)
\]
whatever $\left\{\varphi_\e\right\}_{\e>0}\subset C^\infty_0(\Omega)$ such that $\varphi_\e\rightarrow h$ strongly in $H^1_0(\Omega)$.
\end{definition}

We are now in a position to implement this concept to the study of differential properties of the Lagrangian $\widehat{L}(A,y,\lambda,p)$.
\begin{lemma}
\label{Lemma 33.6}
If a given tuple $(A,y,\lambda,p)\in \mathfrak{A}_{ad}\times D(A)\times \mathbb{R}_{+}\times H^1_0(\Omega)$ is such that $p\in D(A)$, then for each direction $h\in H^1_0(\Omega)$, the generalized 'right hand'
directional derivative $\mathcal{D}^{+}_y\,\widehat{L}(A,y,\lambda,p,h)$
exists and takes the form
\begin{align}
\notag
\mathcal{D}^{+}_y\,\widehat{L}(A,y,\lambda,p,h)=&\ 2\int_\Omega \left(y-y_d\right)h\,dx+2\int_\Omega
\left(\nabla h, A^{sym}\nabla  y\right)_{\mathbb{R}^N}\,dx\\
\label{33.8}
&+\int_\Omega \left(\nabla h, A^{sym}\nabla  p\right)_{\mathbb{R}^N}\,dx-[p,h]_A.
\end{align}
\end{lemma}
\begin{proof}
For a given tuple $(A,y,\lambda,p)$, let $\left\{\varphi_\e\right\}_{\e>0}\subset C^\infty_0(\Omega)$ be a sequence such that
$\varphi_\e\rightarrow h$ strongly in $H^1_0(\Omega)$.  Then the direct computations show that
\begin{align*}
\mathcal{D}^{+}_y\,\widehat{L}(A,y,\lambda,p,\varphi_\e)=\ &
\lim_{\theta\to{+}0} \frac{\widehat{L}(A,y+\theta \varphi_\e,\lambda,p)-\widehat{L}(A,y,\lambda,p)}{\theta}\\=\ &
2\lambda\int_\Omega \left(y-y_d\right)\varphi_\e\,dx
+2\lambda\int_\Omega
\left(\nabla \varphi_\e, A^{sym}\nabla  y\right)_{\mathbb{R}^N}\,dx\\
&+\int_\Omega \left(\nabla \varphi_\e, A^{sym}\nabla p\right)_{\mathbb{R}^N}\,dx+J,
\end{align*}
where
\begin{align}
\label{33.8a}
J=\ & \lim_{\theta\to{+}0} \frac{\lim_{\delta\to 0} [y+\theta\varphi_\e,\psi_\delta]_A
-[y,p\,]_A}{\theta},\\
\label{33.8b}
[y+\theta\varphi_\e,\psi_\delta]_A=\ & \int_\Omega \left(\nabla \psi_\delta, A^{skew}(\nabla  y +\theta \nabla\varphi_\e)\right)_{\mathbb{R}^N}\,dx,\\
\label{33.8c}
[y,p\,]_A :=\ &  \lim_{\delta\to 0}\int_\Omega \left(\nabla \psi_\delta, A^{skew}\nabla y\right)_{\mathbb{R}^N}\,dx
\end{align}
for any $\left\{\psi_\delta\right\}_{\delta>0}\subset C^\infty_0(\Omega)$ such that $\psi_\delta\rightarrow p$ strongly in $H^1_0(\Omega)$.

Since $y\in D(A)$ and
$$
\left|\int_\Omega \left(\nabla \psi_\delta, A^{skew} \nabla\varphi_\e\right)_{\mathbb{R}^N}\,dx\right|\le \|\varphi_\e\|_{C^1(\Omega)}
\|A^{skew}\|_{L^2(\Omega;\mathbb{S}^N_{skew})}\|\psi_\delta\|_{H^1_0(\Omega)},
$$
it follows that
\begin{align*}
\lim_{\delta\to 0} [y+\theta\varphi_\e,\psi_\delta]_A=\ & [y,p\,]_A+ \theta \int_\Omega \left(\nabla p, A^{skew}\nabla \varphi_\e\right)_{\mathbb{R}^N}\,dx\\=\ & [y,p\,]_A-
 \theta \int_\Omega \left(\nabla \varphi_\e, A^{skew}\nabla p\right)_{\mathbb{R}^N}\,dx.
\end{align*}
Summing up the previous transformations, we get
\begin{align}
\notag
\mathcal{D}^{+}_y\,\widehat{L}(A,y,\lambda,p,\varphi_\e)=\ &
2\lambda\int_\Omega \left(y-y_d\right)\varphi_\e\,dx
+2\lambda\int_\Omega
\left(\nabla \varphi_\e, A^{sym}\nabla  y\right)_{\mathbb{R}^N}\,dx\\
\notag
&+\int_\Omega \left(\nabla \varphi_\e, A^{sym}\nabla p\right)_{\mathbb{R}^N}\,dx\\
&-\int_\Omega \left(\nabla \varphi_\e, A^{skew}\nabla p\right)_{\mathbb{R}^N}\,dx, \quad\forall\,\varphi_\e\rightarrow h\ \text{ in }\ H^1_0(\Omega).
\label{33.9}
\end{align}
Since $p\in D(A)$ by the initial assumptions, it follows that
\begin{equation}
\label{33.3}
\big|[p,\varphi_\e]_A\big|:=\left|\int_\Omega \left(\nabla\varphi_\e, A^{skew}\nabla p\right)_{\mathbb{R}^N}\,dx\right|\le c(p,A)\|\varphi_\e\|_{H^1_0(\Omega)}.
\end{equation}
Hence, the 'right hand' directional derivative $\mathcal{D}^{+}_y\,\widehat{L}(A,y,\lambda,p,\varphi_\e)$ admits an extension by continuity for $\varphi\in H^1_0(\Omega)$. Therefore,
the limit passage in \eqref{33.9} as $\e\to 0$ immediately leads us to the representation \eqref{33.8}.
\end{proof}
\begin{corollary}
\label{Col 33.3}
The representation \eqref{33.8} for the generalized 'right hand'
directional derivative $\mathcal{D}^{+}_y\,\widehat{L}(A,y,\lambda,p,h)$ remains valid even if $y\not\in D(A)$ but rather $y\in H^1_0(\Omega)$.
\end{corollary}
\begin{proof}
Indeed, in this case we have
\begin{align*}
\widehat{L}(A,y,\lambda,p)=\ &\lambda I(A,y)+\int_\Omega \big(\nabla p,A^{sym}\nabla y\big)_{\mathbb{R}^N}\,dx\\
&-\langle f,p\rangle_{H^{-1}(\Omega);H^1_0(\Omega)}
 -\lim_{\e\to 0}\int_\Omega \big(\nabla \psi_\e,A^{skew}\nabla p\,\big)_{\mathbb{R}^N}\,dx\\
 =\ & \lambda I(A,y)+\int_\Omega \big(\nabla p,A^{sym}\nabla y\big)_{\mathbb{R}^N}\,dx-\langle
f,p\rangle_{H^{-1}(\Omega);H^1_0(\Omega)}
-[p,y]_{A}.
\end{align*}
for any $\left\{\psi_\delta\right\}_{\delta>0}\subset C^\infty_0(\Omega)$ such that $\psi_\delta\rightarrow y$ strongly in $H^1_0(\Omega)$. It remains to observe further that formulae \eqref{33.8a}--\eqref{33.8c} should be read in this case as follows
\begin{align}
\label{33.8d}
J=\ & \lim_{\theta\to{+}0} \frac{\lim_{\delta\to 0} [\psi_\delta+\theta\varphi_\e,p\,]_A
+[p,y]_A}{\theta},\notag\\
[\psi_\delta+\theta\varphi_\e,p\,]_A=\ & \int_\Omega \left(\nabla p, A^{skew}(\nabla  \psi_\delta +\theta \nabla\varphi_\e)\right)_{\mathbb{R}^N}\,dx\\
\label{33.8f}
=\ & -\int_\Omega \left(\nabla  \psi_\delta +\theta \nabla\varphi_\e,A^{skew}\nabla p\,\right)_{\mathbb{R}^N}\,dx,\\
\label{33.8g}
[p,y]_A :=\ &  \lim_{\delta\to 0}\int_\Omega \left(\nabla \psi_\delta, A^{skew}\nabla p\right)_{\mathbb{R}^N}\,dx
\end{align}
for any $\left\{\psi_\delta\right\}_{\delta>0}\subset C^\infty_0(\Omega)$ such that $\psi_\delta\rightarrow y$ strongly in $H^1_0(\Omega)$.

Since $p\in D(A)$, it follows that
\begin{equation*}
\lim_{\delta\to 0} [\psi_\delta+\theta\varphi_\e,p\,]_A= -[p,y]_A-
 \theta \int_\Omega \left(\nabla \varphi_\e, A^{skew}\nabla p\right)_{\mathbb{R}^N}\,dx.
\end{equation*}
In the rest, we have to follows the arguments of the proof of Lemma~\ref{Lemma 33.6}.
\end{proof}

As an evident consequence of these results, we can give the following specification of formula \eqref{33.8}.
\begin{corollary}
\label{Col 33.9}
Let  $(A,y,\lambda,p)\in \mathfrak{A}_{ad}\times  H^1_0(\Omega)\times \mathbb{R}_{+}\times H^1_0(\Omega)$ be a given tuple. Assume that
$\nabla p\in L^{\infty}(\Omega;\mathbb{R}^N)$ so that $p\in D(A)$. Then the generalized 'right hand'
directional derivative $\mathcal{D}^{+}_y\,\widehat{L}(A,y,\lambda,p,h)$
exists for each direction $h\in H^1_0(\Omega)$ and takes the form
\begin{align}
\notag
\mathcal{D}^{+}_y\,\widehat{L}(A,y,\lambda,p,h)=&\ 2\lambda\int_\Omega \left(y-y_d\right)h\,dx+2\lambda\int_\Omega
\left(\nabla h, A^{sym}\nabla  y\right)_{\mathbb{R}^N}\,dx\\
&+\int_\Omega \left(\nabla h, A^{sym}\nabla  p\right)_{\mathbb{R}^N}\,dx - \int_\Omega \left(\nabla h, A^{skew}\nabla  p\right)_{\mathbb{R}^N}\,dx.
\label{33.10}
\end{align}
\end{corollary}
\begin{proof}
Since $\nabla p\in L^{\infty}(\Omega;\mathbb{R}^N)$, it follows that the estimate \eqref{33.3} can be justified as follows
\[
\big|[p,\varphi_\e]_A\big|:=\left|\int_\Omega \left(\nabla\varphi_\e, A^{skew}\nabla p\,\right)_{\mathbb{R}^N}\,dx\right|\le
\|\nabla p\,\|_{L^\infty(\Omega)}\|A^{skew}\|_{L^2(\Omega;\mathbb{S}^N_{skew})}\|\varphi_\e\|_{H^1_0(\Omega)}.
\]
Hence, $p\in D(A)$ and, therefore, the existence of the generalized 'right hand'
directional derivative $\mathcal{D}^{+}_y\,\widehat{L}(A,y,\lambda,p,h)$ follows from Lemma~\ref{Lemma 33.6}. It remains to
observe that, for such a given $p$, we have $A^{skew}\nabla  p\in L^2(\Omega;\mathbb{R}^N)$. Hence, the  passage to the limit in the last term of \eqref{33.9} reads
\[
[p,h]_A=\lim_{\e\to 0} \int_\Omega \left(\nabla\varphi_\e, A^{skew}\nabla p\right)_{\mathbb{R}^N}\,dx=
\int_\Omega \left(\nabla h, A^{skew}\nabla  p\right)_{\mathbb{R}^N}\,dx.
\]
\end{proof}

\begin{corollary}
\label{Col 33.9a}
Assume that instead of the condition $\nabla p\in L^{\infty}(\Omega;\mathbb{R}^N)$ in Corollary~\ref{Col 33.9}, we have
\begin{equation}
\label{33.4}
A^{skew}\in L^{2+\frac{4}{\gamma}}(\Omega;\mathbb{S}^N_{skew})\quad\text{and}\quad
\nabla p\in L^{2+\gamma}(\Omega;\mathbb{R}^N)\ \text{for some $\gamma\in (0,\infty]$}.
\end{equation}
Then the assertion of Corollary~\ref{Col 33.9} remains true.
\end{corollary}
\begin{proof}
In order to preserve the correctness of the statement of Corollary~\ref{Col 33.9}, it is enough to show that assumptions \eqref{33.4} imply inclusion $p\in D(A)$. Indeed,
since
\begin{align*}
\big|[p,\varphi_\e]_A\big|:=&\left|\int_\Omega \left(\nabla\varphi_\e, A^{skew}\nabla p\right)_{\mathbb{R}^N}\,dx\right|\\
& \le
 \|\varphi_\e\|_{H^1_0(\Omega)}\left( \int_\Omega \|\nabla p\,\|^2_{\mathbb{R}^N}\|A^{skew}\|^2_{\mathbb{S}^N_{skew}}\,dx\right)^{1/2},
\end{align*}
it remains to apply the H\"{o}lder's inequality with conjugates $r=(2+\gamma)/2$ and $q=1+2/\gamma$, ($r^{-1}+q^{-1}=1$). We finally get
\begin{align*}
\int_\Omega \|\nabla p\,\|^2_{\mathbb{R}^N}\|A^{skew}\|^2_{\mathbb{S}^N_{skew}}\,dx\le &
\left(\int_\Omega\|A^{skew}\|^{2+\frac{4}{\gamma}}_{\mathbb{S}^N_{skew}}\,dx\right)^{\frac{\gamma}{2+\gamma}}
\left(\int_\Omega\|\nabla p\,\|^{2+\gamma}_{\mathbb{R}^N}\,dx\right)^{\frac{2}{2+\gamma}}\\
\le & \|A^{skew}\|^{2}_{L^{2+\frac{4}{\gamma}}(\Omega;\mathbb{S}^N_{skew})} \|\nabla p\,\|^2_{L^{2+\gamma}(\Omega)}<+\infty.
\end{align*}
Hence, $p\in D(A)$ and this concludes the proof.
\end{proof}

As obvious consequence of these assertions, we have the following result.

\begin{lemma}
\label{Lemma 33.10}
Let $A\in \mathfrak{A}_{ad}$,
 $y\in  H^1_0(\Omega)$,
$\lambda\in \mathbb{R}_{+}$, and $p\in H^1_0(\Omega)$ be given distributions. If
\begin{equation}
\label{33.6}
\nabla p\in L^{\infty}(\Omega;\mathbb{R}^N)\ \text{ or }\ \left[A^{skew}\in L^{2+\frac{4}{\gamma}}(\Omega;\mathbb{S}^N_{skew})\quad\text{and}\quad
\nabla p\in L^{2+\gamma}(\Omega;\mathbb{R}^N)\right]
\end{equation}
for some $\gamma\in (0,\infty]$,
then the mapping $H^1_0(\Omega)\ni v\mapsto \widehat{L}(A,v,\lambda,p)\in \mathbb{R}$ is G\^{a}teaux differentiable and its G\^{a}teaux differential
takes the form
\[
\left\langle\mathcal{D}_y\,\widehat{L}(A,y,\lambda,p),h\right\rangle_{H^{-1}(\Omega);H^1_0(\Omega)}=\mathcal{D}^{+}_y\,\widehat{L}(A,y,\lambda,p,h),
\quad\forall\,h\in  H^1_0(\Omega).
\]
\end{lemma}
\begin{proof}
Let $(A,y,\lambda,p)\in \mathfrak{A}_{ad}\times  H^1_0(\Omega)\times \mathbb{R}_{+}\times H^1_0(\Omega)$ be a given tuple. Since $p\in D(A)$ by Corollaries~\ref{Col 33.9}--\ref{Col 33.9a}, it follows from Lemma~\ref{Lemma 33.6} that the value $[y+\theta h,p\,]_A$ is well defined for each $h\in H^1_0(\Omega)$.
Moreover, the properties \eqref{33.6} imply the validity of the representation \eqref{33.10} which, of
course, implies
\[
\mathcal{D}^{+}_y\,\widehat{L}(A,y,\lambda,p,h)=-\mathcal{D}^{+}_y\,\widehat{L}(A,y,\lambda,p,-h),\quad\forall\,h\in H^1_0(\Omega),
\]
and thus ascertains the G\^{a}teaux differentiability.
\end{proof}

In what follows, we need the following auxiliary result.
\begin{lemma}
\label{Lemma 33.11}
Let $A\in \mathfrak{A}_{ad}$,
$v\in H^1_0(\Omega)$, $y\in  H^1_0(\Omega)$,
$\lambda\in \mathbb{R}_{+}$, and $p\in H^1_0(\Omega)$ be given distributions.
Assume that the property \eqref{33.6} holds true.
Then there exists a positive value $\e\in [0,1]$ such that
\begin{align}
\widehat{L}(A,v,\lambda,p)-&\widehat{L}(A,y,\lambda,p)=
\left\langle\mathcal{D}_y\,\widehat{L}(A,y+\e(v-y),\lambda,p),v-y\right\rangle_{H^{-1}(\Omega);H^1_0(\Omega)}\notag\\
=&\ 2\lambda\int_\Omega \left(y+\e(v-y)-y_d\right)(v-y)\,dx\notag\\
&+2\lambda\int_\Omega
\left(\nabla (v-y), A^{sym}(\nabla y+\e(\nabla v-\nabla y))\right)_{\mathbb{R}^N}\,dx\notag\\
&+\int_\Omega \left(\nabla (v-y), A^{sym}\nabla  p\right)_{\mathbb{R}^N}\,dx\notag\\
&- \int_\Omega \left(\nabla (v-y), A^{skew}\nabla  p\right)_{\mathbb{R}^N}\,dx.
\label{33.12}
\end{align}
\end{lemma}

\begin{proof}
For given $A, \lambda, p, y$, and $v$, let us consider the scalar function
$$\varphi(t)=\widehat{L}(A,y+t(v-y),\lambda,p).$$
Since by Lemma~\ref{Lemma 33.10} and Corollary~\ref{Col 33.3}, the mapping $H^1_0(\Omega)\ni\xi\mapsto \widehat{L}(\mathcal{U},\xi,\lambda,p)$ is G\^{a}teaux differentiable at each point of the segment
$$[y,v]:=\left\{y+\alpha(v-y)\ :\ \forall\,\alpha\in[0,1]\right\}\subset H^1_0(\Omega),$$
it follows that the function
$\varphi=\varphi(t)$ is differentiable on $[0,1]$ and
\[
\varphi^\prime(t)=\left\langle\mathcal{D}_y\,\widehat{L}(A,y+t(v-y),\lambda,p),v-y\right\rangle_{H^{-1}(\Omega);H^1_0(\Omega)},
\quad\forall\,t\in[0,1].
\]
To conclude the proof, it remains to take into account the representation \eqref{33.10} and apply the Rolle's Theorem:
$$
\varphi(1)-\varphi(0)=\varphi^\prime(\e)\quad\text{ for some }\quad \e\in   [0,1].
$$
\end{proof}

In what follows, we make use of the following concept.
\begin{definition}
\label{Def 33.16}
Let $A_\theta\in \mathfrak{A}_{ad}$ ($\theta\in [0,1]$) be an admissible control, let $y(A_\theta)$ be a weak solution of the problem \eqref{2.1}-\eqref{2.2}, and let
$\e_\theta\in[0,1]$ be a given value. Let $(A_0,y_0)\in\Xi$ be an optimal pair to the problem \eqref{2.3}--\eqref{2.3a}.
We say that a distribution $\psi_\theta$ is the quasi-adjoint state to $y_0\in H^1_0(\Omega)$ for fixed $\theta\in[0,1]$ and $\e_\theta\in [0,1]$, if $\psi_\theta$ satisfies the following integral identity:
\begin{multline}
\label{33.17}
\int_\Omega \big(\nabla \varphi,A_\theta^{sym}\nabla \psi_\theta-A_\theta^{skew}\nabla \psi_\theta\big)_{\mathbb{R}^N}\,dx
=-2\lambda\int_\Omega \left(y_\theta-y_d\right)\varphi\,dx\\-2\lambda\int_\Omega
\left(\nabla \varphi, A_\theta^{sym}\nabla y_\theta\right)_{\mathbb{R}^N}\,dx,
\quad\forall \varphi\in  H^1_0(\Omega),
\end{multline}
where $y_\theta=y_0-\e_\theta(y(A_\theta)-y_0)$.
\end{definition}

By analogy with Proposition~\ref{Prop 2.9}, we can give the following conclusion.
\begin{proposition}
\label{Prop 33.18}
Let $(A_0,y_0)\in\Xi$ be an optimal pair for problem \eqref{2.3}--\eqref{2.3a}, and let
$(A_\theta,y(A_\theta))$ be an admissible pair to this problem. Then, for given $\theta\in[0,1]$ and $\e_\theta\in 0,1]$, $\psi_\e$ is the quasi-adjoint state to $y_0\in H^1_0(\Omega)$ if $\psi_\theta\in D(A_\theta)\subset H^1_0(\Omega)$ is a weak solution to the boundary value problem
\begin{gather}
\label{33.19}
-\div\big(A_\theta^t\nabla \psi_\theta\big) = 2\lambda\,\div\big(A_\theta^{sym}\nabla y_\theta\big) -2\lambda \left(y_\theta-y_d\right) \quad\text{in }\ \Omega,\\[1ex]
\label{33.20}
\psi_\theta=0\text{ on }\partial\Omega.
\end{gather}
\end{proposition}

As for the proof of this assertion, it is enough to observe that $A_\theta^t=A_\theta^{sym}-A_\theta^{skew}$ and $D(A_\theta)=D(A_\theta^t)$ by Definition~\ref{Def 2.7}, and then repeat  the arguments of the proof of Proposition~\ref{Prop 2.9}. In particular, as a consequence of this result, we have the following equalities for the quasi-adjoint state $\psi_\theta$:
\begin{gather}
\int_\Omega \big(\nabla \varphi,A_\theta^{sym}\nabla \psi_\theta\big)_{\mathbb{R}^N}\,dx - [\psi_\theta,\varphi]_{A_\theta}
=-2\lambda\int_\Omega \left(y_\theta-y_d\right)\varphi\,dx\notag\\
\label{33.21}
-2\lambda\int_\Omega
\left(\nabla \varphi, A_\theta^{sym}\nabla y_\theta\right)_{\mathbb{R}^N}\,dx
\,\quad\forall\,\varphi\in H^1_0(\Omega),\\
\int_\Omega \big(A_\theta^{sym}\nabla \psi_\theta,\nabla \psi_\theta\big)_{\mathbb{R}^N}\,dx - [\psi_\theta,\psi_\theta]_{A_\theta}
=-2\lambda\int_\Omega \left(y_\theta-y_d\right)\psi_\theta\,dx\notag\\
\label{33.22}
-2\lambda\int_\Omega
\left(\nabla \psi_\theta, A_\theta^{sym}\nabla y_\theta\right)_{\mathbb{R}^N}\,dx.
\end{gather}

Let us assume the fulfilment of the following hypotheses:
\begin{enumerate}
\item[(H1)] If $A\in \mathfrak{A}_{ad}$ is an admissible matrix-valued control, then $D(A)\supseteq D(A^\ast)$, where the skew symmetric matrix $A^\ast\in L^2(\Omega;\mathbb{S}^N_{skew})$ is given in \eqref{2.3c}.

\item[(H2)] There exists an optimal pair $(A_0,y_0)\in\Xi$  to the problem \eqref{2.3}--\eqref{2.3a} such that $y_0\in D(A^\ast)$.

\item[(H3)] Let $(A_0,y_0)\in\Xi$ be an optimal pair to the problem \eqref{2.3}--\eqref{2.3a} such that $y_0\in D(A^\ast)$. Let $\widehat{A}\in \mathfrak{A}_{ad}$ be an arbitrary admissible control, and let $A_\theta=A_0+\theta(\widehat{A}-A_0)$ for each $\theta\in [0,1]$. Then
    there exists a sequence of corresponding solutions
to the boundary value problem problem \eqref{2.1}-\eqref{2.2} $\left\{y_\theta:=y\left(A_\theta\right)=y\left(A_0+\theta(\widehat{A}-A_0)\right)\right\}_{\theta\to 0}$ such that $y_\theta\rightharpoonup y(A_0)$ in $H^1_0(\Omega)$ as $\theta\to 0$, and $y_\theta\in D(A^\ast)$ for $\theta$ small enough.

\item[(H4)]  Let $(A_0,y_0)\in\Xi$ be an optimal pair to the problem \eqref{2.3}--\eqref{2.3a} such that $y_0\in D(A^\ast)$. Let $\widehat{A}\in \mathfrak{A}_{ad}$ be an arbitrary admissible control, and let $A_\theta=A_0+\theta(\widehat{A}-A_0)$ for each $\theta\in [0,1]$. Then, for each $\theta\in [0,1]$ and $\e_\theta\in [0,1]$, there exists $\gamma\in (0,\infty]$ such that
\begin{equation}
\label{33.24c}
\nabla \psi_\theta\in L^{\infty}(\Omega;\mathbb{R}^N)\ \text{ or }\ \left[A^\ast\in L^{2+\frac{4}{\gamma}}(\Omega;\mathbb{S}^N_{skew})\ \text{and}\
\nabla \psi_\theta\in L^{2+\gamma}(\Omega;\mathbb{R}^N)\right]
\end{equation}
and the sequence of quasi-adjoint states
$\left\{\psi_{\theta}\right\}_{\theta\to 0}$ is relatively compact with respect to the strong topology of $H^1_0(\Omega)$.
\end{enumerate}
\begin{remark}
It is worth to note that due to the condition \eqref{2.3c} and Definition~\ref{Def 2.7},
Hypothesis (H1) sounds realistic. Indeed, in view of estimate \eqref{2.8}, it is reasonably to suppose that the set
$D(A^\ast)\subseteq H^1_0(\Omega)$ strictly depends on the location of 'unboundedness zones' of matrix $A^\ast$. Since $A^{skew}\preceq A^\ast$ for all $A\in \mathfrak{A}_{ad}$, it follows that each of matrices $A\in \mathfrak{A}_{ad}$ inherits the integration properties of
$A^\ast$ in spite of the fact that, as was mentioned earlier, the structure of the set $D(A)$ is still unknown for the time being.
\end{remark}
\begin{remark}
\label{Rem 33.31}
Hypothesis (H3) can be reformulated in other form. Namely, instead of the given variant of (H3), we assume that $y_0$ is a unique solution of the boundary value problem \eqref{2.1}-\eqref{2.2} under $A=A_0$, $y_0\in D(A^\ast)$, and $y_\theta\in D(A^\ast)$ for $\theta$ small enough, where the sequence $\left\{y_\theta:=y\left(A_\theta\right)\right\}_{\theta\to 0}$ is defined as in (H3). Let us show that the remaining conditions of Hypothesis (H2) is carried out in this case.
Indeed, since the boundary value problem \eqref{2.1}-\eqref{2.2} is ill-posed, in general, it means that for a fixed $\theta\in [0,1]$ this problem can admit non-unique solution $y_\theta:=y\left(A_\theta\right)$. In view of Remark~\ref{Rem 4.7}, we can take $y_\theta$ as a weak limit in $H^1_0(\Omega)$ of approximated solutions to \eqref{2.1}-\eqref{2.2}. With that in mind, it is enough to apply an arbitrary approximation $\left\{A^{skew}_{k,\theta}\right\}_{k\in \mathbb{N}}$ of the  matrix $A_\theta^{skew}\in L^2\big(\Omega;\mathbb{S}^N_{skew}\big)$ with properties $\left\{A^{skew}_{k,\theta}\right\}_{k\in \mathbb{N}}\subset L^\infty(\Omega;\mathbb{S}^N_{skew})$ and
$A^{skew}_{k,\theta}\rightarrow A_\theta^{skew}$ strongly in $L^2(\Omega;\mathbb{S}^N_{skew})$ for each $\theta\in [0,1]$. As a result, the sequence $\left\{y_\theta\right\}_{\theta\to 0}$ is uniformly bounded in $H^1_0(\Omega)$ by Theorem~\ref{Th 4.1} (see estimate \eqref{4.5.0}). Therefore, within a subsequence, we can suppose that there exists a function $y^\ast\in H^1_0(\Omega)$ such that $(A_\theta,y_\theta)\,\stackrel{\tau}{\rightarrow}\, (A_0,y^\ast\,)$ as $\theta\to 0$. Moreover, the limit passage in the integral identity
\[
\int_\Omega \big(\nabla \varphi,A_\theta\nabla y_\theta\big)_{\mathbb{R}^N}\,dx
=\left<f,\varphi\right>_{H^{-1}(\Omega);H^1_0(\Omega)},\quad\forall\, \varphi\in C^\infty_0(\Omega)
\]
as $\theta\to 0$, immediately leads us to the conclusion that $(A_0,y^\ast)$ is an admissible pair for optimal control problem \eqref{2.3}--\eqref{2.3a}. Since boundary value problem \eqref{2.1}-\eqref{2.2} has a unique solution under $A=A_0$, we get $y_0=y^\ast$ (see Remark~\ref{Rem 33.25}). Hence, we obtain the required property: $(A_\theta,y_\theta)\,\stackrel{\tau}{\rightarrow}\, (A_0,y_0\,)$.
\end{remark}
\begin{remark}
\label{Rem 33.25}
The assumption in Remark~\ref{Rem 33.31} that $y_0$ has to be a unique solution of the boundary value problem \eqref{2.1}-\eqref{2.2} under $A=A_0$, can be omitted if, instead of this, we assume that a given optimal pair $(A_0,y_0)$ can be attained through $\tau$-limit of the sequence
$$
\left\{(A_\theta,y_\theta):=\left(A_0+\theta(\widehat{A}-A_0),y\left(A_0+\theta(\widehat{A}-A_0)\right)\right))\right\}_{\theta\to 0},
$$
where $\widehat{A}\in \mathfrak{A}_{ad}$ is an arbitrary matrix.
\end{remark}

We are now in a position to derive the optimality conditions for the optimal control problem \eqref{2.3}--\eqref{2.3a}.
\begin{theorem}
\label{Th 33.24}
Let $f\in H^{-1}(\Omega)$ and  $y_d\in  L^2(\Omega)$ be given distributions.
Let $(A_0,y_0)$ be an optimal pair to the problem \eqref{2.3}--\eqref{2.3a} satisfying Hypothesis (H2).
Then the fulfilment of the Hypotheses (H1), (H3)--(H4) implies the existence of elements $\lambda\in \mathbb{R}_{+}$ and $\overline{\psi}\in H^1_0(\Omega)$ such that $\lambda$ and $\overline{\psi}$ are non-zero simultaneously, and

\begin{align}
\label{33.26}\int_\Omega \big(\nabla \varphi,A_0\nabla y_0\big)_{\mathbb{R}^N}\,dx
=\ &\left<f,\varphi\right>_{H^{-1}(\Omega);H^1_0(\Omega)},\ \forall \varphi\in
C^\infty_0(\Omega),\\
\notag
\int_\Omega \big(\nabla \varphi,\, A_0^t\nabla \overline{\psi}\big)_{\mathbb{R}^N}\,dx
=\ &-2\lambda\int_\Omega \left(y_0-y_d\right)\varphi\,dx\\
& -2\lambda\int_\Omega
\left(\nabla \varphi, A_0^{sym}\nabla y_0\right)_{\mathbb{R}^N}\,dx,
\quad\forall \varphi\in  C^\infty_0(\Omega),\label{33.27}\\
\int_\Omega \big(\nabla y_0,\big(A^{sym}-A^{sym}_0\big)&\left(\lambda\nabla y_0+\nabla\overline{\psi}\right)\big)_{\mathbb{R}^N}\,dx\notag\\
\ge\ & [y_0,\overline{\psi}]_{A_0} - [y_0,\overline{\psi}]_{A},\ \forall A\in \mathfrak{A}_{ad}.
\label{33.25}
\end{align}
\end{theorem}

\begin{proof}
Let $(\widehat{A}, \widehat{y})\in \Xi$ be an admissible pair.
We set $A_\theta=A_0+\theta(\widehat{A}-A_0)$, where $\theta\in[0,1]$.
By Proposition~\ref{Prop 2.3f}, we have $A_\theta\in \mathfrak{A}_{ad}$ for all $\theta\in [0,1]$.
Moreover, it is easy to see that $A_\theta\rightarrow A_0$ in the sense \eqref{2.3m}--\eqref{2.3n}.
Let
$y_\theta:=y\left(A_\theta\right)=y\left(A_0+\theta(\widehat{A}-A_0)\right)$ be a solution
of the boundary value problem problem \eqref{2.1}-\eqref{2.2}.  Then Hypothesis (H3) ensures that
\begin{equation}
\label{33.27a}
\text{$(A_\theta,y_\theta)\,\stackrel{\tau}{\rightarrow}\, (A_0,y_0\,)$ as $\theta\to 0$}.
\end{equation}

It is clear that
\begin{align}
\notag
\Delta \widehat{L}=\ &\widehat{L}(A_\theta,y_\theta,\lambda,p)-\widehat{L}(A_0,y_0,\lambda,p)
=\widehat{L}(A_\theta,y_\theta,\lambda,p)-\widehat{L}(A_\theta,y_0,\lambda,p)\\
\notag &+\widehat{L}(A_\theta,y_0,\lambda,p)-\widehat{L}(A_0,y_0,\lambda,p)\\
\notag =\ & \Delta_y
\,\widehat{L}(A_\theta,y_0,\lambda,p)
+\Delta_A\,\widehat{L}(A_0,y_0,\lambda,p)
\ge
0,\\\ &\quad\forall\, \theta\in [0,1],\quad \forall\,(\lambda,p)
\in \mathbb{R}_{+}\times H^1_0(\Omega).\label{33.13}
\end{align}

Taking into account the representations \eqref{33.0}-\eqref{33.1}, we obtain
\begin{align}
\notag
\Delta_A\,\widehat{L}(A_0,y_0,\lambda,p):=\ &\widehat{L}(A_\theta,y_0,\lambda,p)-\widehat{L}(A_0,y_0,\lambda,p)\\
=\ &\lim_{\delta\to 0}\big[L(A_\theta,y_0,\lambda,\varphi_\delta)-L(A_0,y_0,\lambda,\varphi_\delta)\big]\notag\\
=\ &\theta\lambda \int_\Omega \big(\nabla y_0,\left(\widehat{A}^{sym}-A^{sym}_0\right)\nabla y_0\big)_{\mathbb{R}^N}\,dx\notag\\
&+
\theta\int_\Omega \big(\nabla p,\left(\widehat{A}^{sym}-A^{sym}_0\right)\nabla y_0\big)_{\mathbb{R}^N}\,dx \notag\\
&+\theta\lim_{\delta\to 0} \int_\Omega \big(\nabla \varphi_\delta,\left(\widehat{A}^{skew}-A^{skew}_0\right)\nabla y_0\big)_{\mathbb{R}^N}\,dx,
\label{33.14}
\end{align}
for any sequence $\left\{\varphi_\delta\right\}_{\delta>0}\subset C^\infty_0(\Omega)$ such that $\varphi_\delta\rightarrow p$ in $H^1_0(\Omega)$.

It is worth to notice that due to the property $y_0\in D(A^\ast)$, Hypothesis~(H1) implies that, in general, for an arbitrary $p\in H^1_0(\Omega)$, the last term in \eqref{33.14} can be written
$\theta\left([y_0,p\,]_{\widehat{A}}-[y_0,p\,]_{A_0}\right)$.

As for the term $\Delta_y
\,\widehat{L}(A_0,y_0,\lambda,p)$ in \eqref{33.13},
we temporary assume that the distribution $p\in H^1_0(\Omega)$ satisfies the property \eqref{33.6} with $A=A^\ast$.
Then by Lemma~\ref{Lemma 33.11},
there exists a  positive value $\e_\theta\in [0,1]$ such that
\begin{multline}
\label{33.15}
\Delta_y
\,\widehat{L}(A_0,y_0,\lambda,p)=
\widehat{L}(A_\theta,y_\theta,\lambda,p)-\widehat{L}(A_\theta,y_0,\lambda,p)\\=
\left\langle\mathcal{D}_y\,\widehat{L}(A_\theta,y_0+\e_\theta(y_\theta-y_0),\lambda,p),y_\theta-
y_0\right\rangle_{H^{-1}(\Omega);H^1_0(\Omega)}.
\end{multline}
As a result, combining \eqref{33.14} and \eqref{33.15}, and taking into account property \eqref{33.6} with $A=A^\ast$ and
Lemma~\ref{Lemma 33.11}, we can represent inequality \eqref{33.13}
as follows
\begin{align}
\notag
\Delta \widehat{L}=&\widehat{L}(A_\theta,y_\theta,\lambda,p)-\widehat{L}(A_0,y_0,\lambda,p)=2\lambda\int_\Omega \left(y_0+\e_\theta(y_\theta-y_0)-y_d\right)(y_\theta-y_0)\,dx\\
&+2\lambda\int_\Omega
\left(\nabla (y_\theta-y_0), A^{sym}_\theta(\nabla y_0+\e_\theta(\nabla y_\theta-\nabla y_0))\right)_{\mathbb{R}^N}\,dx\notag\\
\notag
&+\int_\Omega \left(\nabla (y_\theta-y_0), A^{sym}_\theta\nabla  p\,\right)_{\mathbb{R}^N}\,dx
- \int_\Omega \left(\nabla (y_\theta-y_0), A^{skew}_\theta\nabla  p\,\right)_{\mathbb{R}^N}\,dx\\
\notag
&+\theta\lambda \int_\Omega \big(\nabla y_0,\left(\widehat{A}^{sym}-A^{sym}_0\right)\nabla y_0\big)_{\mathbb{R}^N}\,dx\\
&+
\theta\int_\Omega \big(\nabla y_0,\left(\widehat{A}^{sym}-A^{sym}_0\right)\nabla p\,\big)_{\mathbb{R}^N}\,dx \notag\\
&-\theta  \int_\Omega \big(\nabla y_0,\left(\widehat{A}^{skew}-A^{skew}_0\right)\nabla p\,\big)_{\mathbb{R}^N}\,dx.
\label{33.28}
\end{align}

In view of the property \eqref{33.24c}, let us define the element $p$ in \eqref{33.28} as the quasi-adjoint state to $y_0\in H^1_0(\Omega)$, that is, we set $p=\psi_{\theta}$, where
$\psi_{\theta}$ satisfies the following integral identity:
\begin{align}
\int_\Omega \big(\nabla \varphi,\,& A^{sym}_\theta\nabla \psi_{\theta}\big)_{\mathbb{R}^N}\,dx
-\int_\Omega \big(\nabla \varphi,A^{skew}_\theta\nabla \psi_{\theta}\big)_{\mathbb{R}^N}\,dx\notag\\
=\ &-2\lambda\int_\Omega \left(y_0+\e_\theta(y_\theta-y_0)-y_d\right)\varphi\,dx\notag\\&-2\lambda\int_\Omega
\left(\nabla \varphi, A^{sym}_\theta(\nabla y_0+\e_\theta(\nabla y_\theta-\nabla y_0)\right)_{\mathbb{R}^N}\,dx,\ \forall\,\varphi\in  H^1_0(\Omega).
\label{33.29}
\end{align}
As a result,  we can justify relation \eqref{4.5} to the form
\begin{align}
\notag
\frac{\Delta \widehat{L}}{\theta}=\ &\frac{\widehat{L}(A_\theta,y_\theta,\lambda,\psi_{\theta})-\widehat{L}(A_0,y_0,\lambda,\psi_{\theta})}{\theta}\\ \notag
=\ &\lambda \int_\Omega \big(\nabla y_0,\left(\widehat{A}^{sym}-A^{sym}_0\right)\nabla y_0\big)_{\mathbb{R}^N}\,dx\\
&+
\int_\Omega \big(\nabla y_0,\left(\widehat{A}^{sym}-A^{sym}_0\right)\nabla \psi_{\theta}\big)_{\mathbb{R}^N}\,dx \notag\\
&+  \int_\Omega \big(\nabla \psi_{\theta},\left(\widehat{A}^{skew}-A^{skew}_0\right)\nabla y_0\big)_{\mathbb{R}^N}\,dx\ge 0,\quad\forall\, \widehat{A}\in \mathfrak{A}_{ad},
\label{33.30}
\end{align}
where the last term has a sense by property \eqref{33.24c} and Hypotheses (H1)--(H2).

It remains to pass to the limit in \eqref{33.29}--\eqref{33.30} as $\theta\to {+}0$. To this end, we note that (see \eqref{33.27a})
\begin{enumerate}
         \item[($A_1$)] $A_\theta\rightarrow A_0$ in the sense of \eqref{2.3m}--\eqref{2.3l}
 as $\theta\to 0$;
         \item[($A_2$)] $y_\theta\rightharpoonup y_0$ in $H^1_0(\Omega)$ as $\theta\to 0$ by Hypothesis (H3);
         \item[($A_3$)] there exists an element $\overline{\psi}\in H^1_0(\Omega)$ such that
         (within a subsequence) $\psi_{\theta}\rightarrow\overline{\psi}$ in $H^1_0(\Omega)$ as $\theta\to 0$ by Hypothesis (H4).
\end{enumerate}
Then, taking into account the fact that $\big(\widehat{A}^{sym}-A^{sym}_0\big)\in L^\infty(\Omega;\mathbb{S}^N_{sym})$, the limit passage in \eqref{33.30} gives
\begin{align}
\notag
\lim_{\theta\to 0}\frac{\Delta \widehat{L}}{\theta}&\,\stackrel{\text{by ($A_1$)--($A_3$)}}{=}
\lambda \int_\Omega \big(\nabla y_0,\left(\widehat{A}^{sym}-A^{sym}_0\right)\nabla y_0\big)_{\mathbb{R}^N}\,dx\\
\notag &+
\int_\Omega \big(\nabla y_0,\left(\widehat{A}^{sym}-A^{sym}_0\right)\nabla \overline{\psi}\big)_{\mathbb{R}^N}\,dx\\
&+ \lim_{\theta\to 0} \int_\Omega \big(\nabla\psi_{\theta},\left(\widehat{A}^{skew}-A^{skew}_0\right)\nabla y_0 \big)_{\mathbb{R}^N}\,dx\ge 0.
\label{33.31}
\end{align}
For each $\theta\in [0,1]$, we define a sequence $\left\{\varphi_{\delta,\theta}\right\}_{\delta\to 0}\subset C^\infty_0(\Omega)$ such that $\varphi_{\delta,\theta}\rightarrow \psi_{\theta}$ in $H^1_0(\Omega)$ as $\delta\to 0$. Further, we note that
\begin{multline*}
\left|\int_\Omega \big(\nabla (\psi_{\theta}-\varphi_{\delta,\theta}),\left(\widehat{A}^{skew}-A^{skew}_0\right)\nabla y_0 \big)_{\mathbb{R}^N}\,dx\right|\\
\stackrel{\text{by (H1) and \eqref{2.8}}}{\le}\, c(y_0,A^\ast)\|\psi_{\theta}-\varphi_{\delta,\theta}\|_{H^1_0(\Omega)}\rightarrow 0\ \text{as }\ \delta \to 0.
\end{multline*}
Hence,  there exists a monotonically decreasing to $0$ sequence $\left\{\delta(\theta)\right\}_{\theta\to 0}$ such that $\varphi_{\delta(\theta),\theta}-\psi_{\theta}\rightarrow 0$ in $H^1_0(\Omega)$ as $\theta\to 0$, and, therefore,
\begin{align*}
\lim_{\theta\to 0} \int_\Omega \big(\nabla\psi_{\theta},&\left(\widehat{A}^{skew}-A^{skew}_0\right)\nabla y_0 \big)_{\mathbb{R}^N}\,dx\\=\ &
\lim_{\theta\to 0} \int_\Omega \big(\nabla\varphi_{\delta(\theta),\theta},\left(\widehat{A}^{skew}-A^{skew}_0\right)\nabla y_0 \big)_{\mathbb{R}^N}\,dx\\
=\ &\lim_{\theta\to 0}\left([y_0,\varphi_{\delta(\theta),\theta}]_{\widehat{A}} - [y_0,\varphi_{\delta(\theta),\theta}]_{A_0}\right)\\
=\ &(\text{by continuity})=
[y_0,\overline{\psi}]_{\widehat{A}} - [y_0,\overline{\psi}]_{A_0}.
\end{align*}
Combining this result with \eqref{33.31}, we immediately arrive at the inequality  \eqref{33.25}.

As for the limit passage in \eqref{33.29} as $\theta\to 0$, we see that the numerical sequence $\left\{\e_\theta\right\}_{\theta\to 0}$ is bounded and
\[
\begin{array}{rll}
A_\theta^{sym}\rightarrow A_0^{sym}\ \text{and}\
A_\theta^{skew}&\rightarrow A_0^{skew} &\text{ strongly in $L^2(\Omega;\mathbb{M}^N)$ by ($A_1$)-condition},\\
\left(y_0+\e_\theta(y_\theta-y_0)\right)&\rightarrow y_0  &\text{ strongly in $L^2(\Omega)$ by ($A_2$)-condition},\\
(\nabla y_0+\e_\theta(\nabla y_\theta-\nabla y_0)&\rightharpoonup \nabla y_0  &\text{ weakly in $L^2(\Omega)$ by ($A_2$)-condition}.
\end{array}
\]
As a result, the limit passage in \eqref{33.29} as $\theta\to 0$ leads us to integral identity \eqref{33.27}. This concludes the proof.
\end{proof}
\begin{remark}
As follows from Remark~\ref{Rem 4.17}, the optimality system given by Theorem~\ref{Th 33.24} admits the existence of zero Lagrange multiplier  $\lambda=0$. Indeed, the set of all weak solutions of the adjoint problem under $\lambda=0$
\begin{equation}
\label{33.32}
\begin{array}{c}
-\div\big(A_0^t\nabla \overline{\psi}\big) = 0\quad\text{in }\ \Omega,\\[1ex]
\overline{\psi}=0\text{ on }\partial\Omega.
\end{array}
\end{equation}
may contain non-trivial elements, and, hence, the inequality
\[
 \int_\Omega \big(\nabla y_0,\big(A^{sym}-A^{sym}_0\big)\nabla\overline{\psi}\big)_{\mathbb{R}^N}\,dx
\ge [y_0,\overline{\psi}]_{A_0} - [y_0,\overline{\psi}]_{A}
\]
can remain valid for all $A\in \mathfrak{A}_{ad}$. At the same time, following Zhikov (see \cite{Zhik1_04}), it can be
easily shown that the boundary value problem \eqref{33.32} has a unique trivial solution if only
\begin{equation}
\label{33.33}
\lim_{p\to\infty} p^{-1}\|A_0^{skew}\|_{L^p(\Omega;\mathbb{S}^N_{skew})}=0.
\end{equation}
It is east to see that this condition can be interpreted as some extension of
the case of $L^\infty(\Omega;\mathbb{S}^N_{skew})$-matrices because the fulfilment of property \eqref{33.33} is obvious for the skew-symmetric matrices
with entries like e.g. $\ln\ln\|x\|_{\mathbb{R}^N}$.
\end{remark}

Taking this remark into account, we can precise Theorem~\ref{Th 33.24} as follows.
\begin{theorem}
\label{Th 33.33}
Let $f\in H^{-1}(\Omega)$ and  $y_d\in  L^2(\Omega)$ be given distributions.
Let $(A_0,y_0)$ be an optimal pair to the problem \eqref{2.3}--\eqref{2.3a} satisfying Hypothesis (H2) and property \eqref{33.33}.
Then the fulfilment of the Hypotheses (H1), (H3)--(H4) implies the existence of an element  $\overline{\psi}\in H^1_0(\Omega)$ such that
\begin{align*}
\int_\Omega \big(\nabla \varphi,A_0\nabla y_0\big)_{\mathbb{R}^N}\,dx
=\ &\left<f,\varphi\right>_{H^{-1}(\Omega);H^1_0(\Omega)},\ \forall \varphi\in
C^\infty_0(\Omega),\\
\int_\Omega \big(\nabla \varphi,\, A_0^t\nabla \overline{\psi}\big)_{\mathbb{R}^N}\,dx
=\ &-2\int_\Omega \left(y_0-y_d\right)\varphi\,dx\\
& -2\int_\Omega
\left(\nabla \varphi, A_0^{sym}\nabla y_0\right)_{\mathbb{R}^N}\,dx,
\quad\forall \varphi\in  C^\infty_0(\Omega),\\
\int_\Omega \big(\nabla y_0,\big(A^{sym}-A^{sym}_0\big)&\left(\nabla y_0+\nabla\overline{\psi}\right)\big)_{\mathbb{R}^N}\,dx\\
\ge\ & [y_0,\overline{\psi}]_{A_0} - [y_0,\overline{\psi}]_{A},\ \forall A\in \mathfrak{A}_{ad}.
\end{align*}
\end{theorem}

\section{Example of a non-variational optimal solution}
\label{Sec A3a}

Let $\Omega$ be the unit ball in $\mathbb{R}^3$, $\Omega=\left\{x\in \mathbb{R}^3\ :\ \|x\|_{\mathbb{R}^3}<1\right\}$.
Let us consider the following OCP:
\begin{gather}
\label{A3a.1}
\text{Minimize } I(A,y)==\left\|y-y_d\right\|^2_{L^2(\Omega)}+\int_\Omega\left(\nabla y-\nabla y_d, A^{sym}(\nabla y-\nabla y_d)\right)_{\mathbb{R}^N}\,dx\\
\label{A3aa.1}
\text{subject to the constraints \eqref{2.1}--\eqref{2.2} and $A\in \mathfrak{A}_{ad}\subset L^2(\Omega;\mathbb{M}^N)$,}
\end{gather}
where the distributions $A^\ast\in L^2(\Omega,\mathbb{S}^3_{skew})$, $f\in H^{-1}(\Omega)$, and $y_{d}\in H^1_0(\Omega)$ will be defined later on. Our intention is to show that in this case the above problem admits a non-variational solution, i.e. there exists an admissible pair $(A^0,y^0)\in \mathfrak{A}_{sd}\times H^1_0(\Omega)$ such that
\begin{equation}
\label{A3a.2}
I(A^0,y^0)=0=\ds\inf_{(A,y)\in \Xi} I(A,y)\quad\text{and}\quad [y^0,y^0]_A=-\frac{\zeta}{2}<0,
\end{equation}
where $\zeta$ is a given positive value.

We divide our analysis into several steps. At the first step we define a skew-symmetric matrix $A^\ast$ as follows
\begin{equation}
\label{A3a.5b}
A^\ast(x)=\left(
       \begin{array}{ccc}
         0 & a(x) & 0 \\
         -a(x) & 0 & -b(x) \\
         0 & b(x) & 0 \\
       \end{array}
     \right),
\end{equation}
where
$a(x)=\ds \frac{x_1}{2\|x\|^2_{\mathbb{R}^3}}$ and $ b(x)=\ds \frac{x_3}{2\|x\|^2_{\mathbb{R}^3}}$.
Since
\begin{align*}
\left\|a\right\|^2_{L^2(\Omega)} &= \int_\Omega \left(\frac{x_1}{2\|x\|^2_{\mathbb{R}^3}}\right)^2 \,dx\\
&= \int_0^1 \int_0^{2\pi}\int_0^\pi \frac{\rho^2\cos^2\varphi\sin^2\psi}{\rho^4} \rho^2\sin\varphi\,d\psi\,d\varphi\,d\rho<+\infty,
\end{align*}
it follows that $a\in L^2(\Omega)$. By analogy, it can be shown that $b\in L^2(\Omega)$.  Moreover, it is easy to see that the skew-symmetric matrix $A^\ast$, define by \eqref{A3a.5b}, satisfies the property
$A^\ast\in H(\Omega,\mathrm{div};\mathbb{S}^3)$, i.e. $A^\ast\in L^2(\Omega;\mathbb{S}^3_{skew})$ and $\div A^\ast\in L^1(\Omega;\mathbb{R}^3)$.
Indeed, in view of the definition of the divergence $\div A^\ast$ of a skew-symmetric matrix, we have
$\div A^\ast=\left[
\begin{array}{c}
d_1\\d_2\\d_3
\end{array}
\right]$,
where $d_i=\mathrm{div}\,a^\ast_i=\ds\frac{x_i x_2}{\|x\|_{\mathbb{R}^3}^4}$ and $a^\ast_i$ is $i$-th column of $A^\ast$. As a result, we get
\[
\|\mathrm{div}\,a^\ast_i\|_{L^1(\Omega)}=\int_0^1 \int_0^{2\pi}\int_0^\pi \left|\frac{\rho^2 f_i(\varphi,\psi)\sin\varphi\sin\psi}{\rho^4} \right| \rho^2\sin\varphi\,d\psi\,d\varphi\,d\rho<+\infty,
\]
for the corresponding $f_i=f_i(\varphi,\psi)$ $(i=1,2,3)$. Therefore, $\div A^\ast\in L^1(\Omega;\mathbb{R}^3)$.

Step~2 deals with the choice of the function $y_{d}\in H^1_0(\Omega)$. We define it by the rule
\begin{equation}
\label{A3a.7}
y_{d}= \frac{\sqrt{\zeta}}{\pi}
\left(1-\|x\|_{\mathbb{R}^3}^5\right)\sqrt{4\pi-\mathrm{atan2}\,\left(\frac{x_2}{\|x\|_{\mathbb{R}^3}},\frac{x_1}{\|x\|_{\mathbb{R}^3}}\right)}
\quad\text{ in }\ \Omega,
\end{equation}
where the two-argument function $\mathrm{atan2}(y,x)$ is defined as follows
\[
\mathrm{atan2}\,(y,x)=\left\{
\begin{array}{ll}
\arctan\left(\frac{y}{x}\right)+\pi, & x<0,\\
\arctan\left(\frac{y}{x}\right)+2\pi, & y< 0, x>0,\\
\arctan\left(\frac{y}{x}\right), & y\ge 0, x>0,\\
{\pi}/{2}, & y>0, x=0,\\
{3\pi}/{2}, & y<0, x=0,\\
0, & y=0, x=0.
\end{array}
\right.
\]
It is easy to see that the range of $\mathrm{atan2}(y,x)$ is $[0,2\pi]$ and
\[
v^2_0\Big(\frac{x}{\|x\|_{\mathbb{R}^3}}\Big):=
\frac{\zeta}{\pi^2}
\left(4\pi-\mathrm{atan2}\,\left(\frac{x_2}{\|x\|_{\mathbb{R}^3}},\frac{x_1}{\|x\|_{\mathbb{R}^3}}\right)\right)
=\frac{\zeta}{\pi^2}\left(4\pi-\varphi\right),\quad\forall\,\varphi\in [0,2\pi]
\]
with respect to the spherical coordinates. Hence, $v_0\in C^\infty(\partial\Omega)$, and, as
immediately follows from \eqref{A3a.7}, it provides that
\[
y_{d}\in L^2(\Omega)\quad\text{and}\quad y_{d}=0\ \text{ on }\ \partial\Omega.
\]
By direct computations, we get
\begin{equation}
\label{A3a.6''}
\nabla v_0\Big(\frac{x}{\|x\|_{\mathbb{R}^3}}\Big)=\frac{1}{\|x\|_{\mathbb{R}^3}^3}
\left[
\begin{array}{c}
\frac{\partial v_0}{\partial z_1}\left(\|x\|_{\mathbb{R}^3}^2-x_1^2\right) - \frac{\partial v_0}{\partial z_2}x_1x_2\\
\frac{\partial v_0}{\partial z_2}\left(\|x\|_{\mathbb{R}^3}^2-x_2^2\right) - \frac{\partial v_0}{\partial z_1}x_1x_2\\
- \frac{\partial v_0}{\partial z_1}x_1x_3 - \frac{\partial v_0}{\partial z_2}x_2x_3
\end{array}
\right],\ \forall\,x\ne 0.
\end{equation}
Hence, there exists a constant $C^\ast>0$ such that
\[
\left\|\nabla v_0\Big(\frac{x}{\|x\|_{\mathbb{R}^3}}\Big)\right\|_{\mathbb{R}^3}\le \frac{C^\ast}{\|x\|_{\mathbb{R}^3}}.
\]
Thus,
\begin{align*}
\|\nabla y_{d}\|_{\mathbb{R}^3}\le & \left|v_0\Big(\frac{x}{\|x\|_{\mathbb{R}^3}}\Big)\right|\left\|\nabla\left(1-\|x\|_{\mathbb{R}^3}^5\right)\right\|_{\mathbb{R}^3}\\ &+ \left(1-\|x\|_{\mathbb{R}^3}^5\right)\left\|\nabla v_0\Big(\frac{x}{\|x\|_{\mathbb{R}^3}}\Big)\right\|_{\mathbb{R}^3}\le C_1+\frac{C_2}{\|x\|_{\mathbb{R}^3}}.
\end{align*}
As a result, we infer that $\nabla y_{d}\in L^2(\Omega;\mathbb{R}^3)$, i.e. we finally have $y_{d}\in H^1_0(\Omega)$.

Step~3. We show that the function $y_{d}$, which was introduced before, belongs to the set $D(A^\ast)$. To do so, we have to prove the estimate
\begin{equation}
\label{A3a.8}
\left|\int_\Omega \big(\nabla \varphi,A^\ast(x)\nabla y_{d}\big)_{\mathbb{R}^3}\,dx\right|\le \widetilde{C}(y_{d}) \left(\int_\Omega |\nabla \varphi|^2_{\mathbb{R}^3}\right)^{1/2}\  \forall\,\varphi\in C^\infty_0(\mathbb{R}^3).
\end{equation}
To this end, we make use of the following transformations
\begin{align*}
\int_\Omega \big(\nabla\varphi,\,&A^\ast\nabla \psi\big)_{\mathbb{R}^3}\,dx=
-\big<\mathrm{div}\,\left(A^\ast\nabla \psi\right),\varphi\big>_{H^{-1}(\Omega);H^1_0(\Omega)}\\
=\ &\Big< \mathrm{div}\,\left[
\begin{array}{c}
(a^\ast_{1})^t\nabla \psi\\ (a^\ast_{2})^t \nabla \psi\\(a^\ast_{3})^t\nabla \psi
\end{array}\right],\varphi\Big>_{H^{-1}(\Omega);H^1_0(\Omega)}\\
=\ &
\sum_{i=1}^3\left<\mathrm{div}\,a^\ast_{i}, \varphi\frac{\partial\psi}{\partial x_i}\right>_{H^{-1}(\Omega);H^1_0(\Omega)}
+ \underbrace{\int_\Omega
\sum_{i=1}^3 \sum_{j=1}^3 \left(a^\ast_{ij} \frac{\partial^2 \psi}{\partial x_i\partial x_j}\right) \varphi\,dx}_{=0\ \atop {\text{ since }\ A^\ast\in L^2(\Omega;\mathbb{S}^3_{skew})}}\\
&(\text{due to the fact that }\text{$\div A^\ast\in L^\infty(\Omega;\mathbb{R}^3)$})
\\
=\ &\int_{\Omega} \left(\,\div A^\ast,\nabla \psi\right)_{\mathbb{R}^3} \varphi\,dx,
\end{align*}
which are obviously true for all $\psi,\varphi\in C^\infty_0(\Omega)$. Since
\[
\left|\int_{\Omega} \left(\,\div A^\ast,\nabla \psi\right)_{\mathbb{R}^3} \varphi\,dx\right|=\left|\int_\Omega \big(\nabla\varphi,\,A^\ast\nabla \psi\big)_{\mathbb{R}^3}\,dx\right|\le C\|A^\ast\|_{L^2(\Omega;\mathbb{S}^3_{skew})}\|\psi\|_{H^1_0(\Omega)},
\]
it follows that,
using the continuation principle, we can extend the previous equality
with respect to $\psi$ to the following one
\begin{equation}
\label{A3a.13}
\int_\Omega \left(\nabla\varphi,A^\ast\nabla y_{d}\right)_{\mathbb{R}^3}\,dx=
\int_\Omega \varphi \left(\div A^\ast,\nabla y_{d}\right)_{\mathbb{R}^3}\,dx\quad \forall\,\varphi\in C^\infty_0(\Omega).
\end{equation}

Let us show that $\left(\div A^\ast,\nabla y_{d}\right)_{\mathbb{R}^3}\in L^\infty(\Omega)$. In this case,
relation \eqref{A3a.13} implies the estimate
\begin{align*}
\left|\int_\Omega \left(\nabla\varphi,A^\ast\nabla y_{d}\right)_{\mathbb{R}^3}\,dx\right|\le&
\left\|\left(\div A^\ast,\nabla y_{d}\right)_{\mathbb{R}^3}\right\|_{L^\infty(\Omega)}\int_\Omega |\varphi|\,dx\\
\le&
\widetilde{C}(y_{d}) \left(\int_\Omega |\nabla \varphi|^2_{\mathbb{R}^N}\right)^{1/2}\  \forall\,\varphi\in C^\infty_0(\mathbb{R}^N),
\end{align*}
which means that the element $y_{d}$ belongs to the set $D(A^\ast)$.

 Indeed, as follows from \eqref{A3a.6''}, we have the equality
\begin{equation}
\label{A3a.14}
\left(\nabla v_0\Big(\frac{x}{\|x\|_{\mathbb{R}^3}}\Big),\frac{x}{\|x\|_{\mathbb{R}^3}^3}\right)_{\mathbb{R}^3}=0.
\end{equation}
Thus, the gradient of the function $\nabla v_0(\frac{x}{\|x\|_{\mathbb{R}^3}})$ is orthogonal to the vector field $Q={x}/{\|x\|_{\mathbb{R}^3}^3}$ outside the origin. Therefore,
\begin{align}
\notag
\left(\nabla y_{d},\div A^\ast\right)_{\mathbb{R}^3}:=& 
\left(\nabla \left[\left(1-\|x\|_{\mathbb{R}^3}^5\right)v_0\Big(\frac{x}{\|x\|_{\mathbb{R}^3}}\Big)\right], \frac{x}{\|x\|_{\mathbb{R}^3}^3}\,\frac{x_2}{\|x\|_{\mathbb{R}^3}} \right)_{\mathbb{R}^3}\\ \notag
=& \left(\nabla \left(1-\|x\|_{\mathbb{R}^3}^5\right),\frac{x}{\|x\|_{\mathbb{R}^3}^3}\right)_{\mathbb{R}^3} v_0\Big(\frac{x}{\|x\|_{\mathbb{R}^3}}\Big)\frac{x_2}{\|x\|_{\mathbb{R}^3}}\\
& +
\left(1-\|x\|_{\mathbb{R}^3}^5\right)\left( \nabla v_0\Big(\frac{x}{\|x\|_{\mathbb{R}^3}}\Big),\frac{x}{\|x\|_{\mathbb{R}^3}^3}\right)_{\mathbb{R}^3}\frac{x_2}{\|x\|_{\mathbb{R}^3}}=
I_1+I_2,
\label{A3a.15}
\end{align}
where $I_2=0$ by \eqref{A3a.14}. Since $\nabla \left(1-\|x\|_{\mathbb{R}^3}^5\right)=-5\|x\|_{\mathbb{R}^3}^3x$, $\frac{x_2}{\|x\|_{\mathbb{R}^3}}=\sin\varphi \sin\psi$ with respect to the spherical coordinates, and function $v_0$ is smooth, it follows that there exists a constant $C_0>0$ such that $\left|\left(\nabla y_{d},\div A^\ast\right)_{\mathbb{R}^3}\right|\le C_0$ almost everywhere in $\Omega$. Thus, $\left(\div A^\ast,\nabla y_{d}\right)_{\mathbb{R}^3}\in L^\infty(\Omega)$ and we have obtained the required property.

Step~4. Using results of the previous steps, we show that the function $y_{d}$ satisfies
the condition $[y_d,y_d]_{A^\ast}=-\frac{\zeta}{2}<0$.
Indeed, let $\left\{\varphi_\e\right\}_{\e\to 0}\subset C^\infty_0(\Omega)$ be a sequence such that
$\varphi_\e\rightarrow y_{d}$ strongly in $H^1_0(\Omega)$. Then by continuity, we have
\begin{align*}
[y_d,y_d]_{A^\ast}=& \lim_{\e\to 0} \int_\Omega \left(\nabla\varphi_\e,A^\ast\nabla y_{d}\right)_{\mathbb{R}^3}\,dx\\
\stackrel{\text{by \eqref{A3a.13}}}{=}&\lim_{\e\to 0}
\int_\Omega \varphi_\e \left(\div A^\ast,\nabla y_{d}\right)_{\mathbb{R}^3}\,dx
\end{align*}
Since $\left(\div A^\ast,\nabla y_{d}\right)_{\mathbb{R}^3}\in L^\infty(\Omega)$ and $\varphi_\e\rightarrow y_{d}$ strongly in $H^1_0(\Omega)$, we can pass to the limit in the right-hand side of this relation. As a result, we get
\begin{equation}
\label{A3a.16}
[y_d,y_d]_{A^\ast}=\int_\Omega y_{d} \left(\div A^\ast,\nabla y_{d}\right)_{\mathbb{R}^3}\,dx=\frac{1}{2}
\int_\Omega  \left(\div A^\ast,\nabla y^2_{d}\right)_{\mathbb{R}^3}\,dx.
\end{equation}
Let $\Omega_\e=\left\{x\in \mathbb{R}^3\ |\ \e<\|x\|_{\mathbb{R}^3}<1\right\}$ and let $\Gamma_\e=\left\{\|x\|_{\mathbb{R}^3}=\e\right\}$ be the sphere of radius $\e$ centered at the origin. Then
\begin{align*}
\int_{\Omega_\e} \left(\div A^\ast,\nabla y^2_{d}\right)_{\mathbb{R}^3}&\,dx\stackrel{\text{since }\ y_{d}\in H^1_0(\Omega)}{=}
\int_{\Gamma_\e} \left(\div A^\ast,\nu \right)_{\mathbb{R}^3}y^2_{d}\,d \mathcal{H}^2\\&=
\int_{\Gamma_\e} \left(\div A^\ast,\nu \right)_{\mathbb{R}^3} \left(1-\|x\|_{\mathbb{R}^3}^5\right)^2 v^2_0\Big(\frac{x}{\|x\|_{\mathbb{R}^3}}\Big)\,d \mathcal{H}^2\\
&=\int_{\Gamma_\e} \left(\div A^\ast,\nu \right)_{\mathbb{R}^3} v^2_0\Big(\frac{x}{\|x\|_{\mathbb{R}^3}}\Big)\,d \mathcal{H}^2 + o(1)\\
&=\int_{\Gamma_\e} \left(\frac{x}{\|x\|^3_{\mathbb{R}^3}},\Big(-\frac{x}{\|x\|_{\mathbb{R}^3}}\Big)\right)_{\mathbb{R}^3}\frac{ x_2}{\|x\|_{\mathbb{R}^3}}v^2_0\Big(\frac{x}{\|x\|_{\mathbb{R}^3}}\Big)\,d \mathcal{H}^2+o(1)
\\&=-
\e^{-2} \int_{\Gamma_\e} \frac{x_2}{\|x\|_{\mathbb{R}^3}} v^2_0\Big(\frac{x}{\|x\|_{\mathbb{R}^3}}\Big)\,d \mathcal{H}^2+o(1)\\
&=-\int_{\Gamma} b_0(x)v^2_0(x)\,d \mathcal{H}^2+o(1),
\end{align*}
where $b_0=\sin\varphi \sin\psi\quad\text{and}$ and $v_0^2=\frac{\zeta}{\pi^2}\left(4\pi-\varphi\right)$.
Since
\begin{equation*}
\int_{\partial\Omega} b_0v^2_0\,d\mathcal{H}^2=\frac{\zeta}{\pi^2}\int_0^{2\pi} \sin\varphi\left(4\pi-\varphi\right)\,d\varphi \int_0^\pi \sin^2\psi\,d\psi
=\zeta>0,
\end{equation*}
it remains to combine this result with \eqref{A3a.16} and relation
\[
\int_\Omega  \left(\div A^\ast,\nabla y^2_{d}\right)_{\mathbb{R}^3}\,dx=\lim_{\e\to 0}\int_{\Omega_\e } \left(\div A^\ast,\nabla y^2_{d}\right)_{\mathbb{R}^3}\,dx.
\]
As a result, we infer $[y_d,y_d]_{A^\ast}=-\zeta/2<0$.

Step~5. This is the last step in our analysis of OCP \eqref{A3a.1}--\eqref{A3aa.1}. Let us define the distribution $f\in H^{-1}(\Omega)$ as follows
\begin{gather}
\label{A3a.17a}
f=-\mathrm{div}\,\left(A^{\sharp}\nabla y_{d}+A^\ast\nabla y_{d}\right),
\end{gather}
where $A^{\sharp}$ is an arbitrary symmetric matrix such that $A^{\sharp}\in \mathfrak{A}_{ad,1}$.

Assume that a compact set $Q$ of $L^2(\Omega;\mathbb{S}^3_{skew})$ contains matrix $A^\ast$. Then it is obvious that
the matrix $A^0=A^{\sharp}+A^\ast$ is admissible for OCP \eqref{A3a.1}--\eqref{A3aa.1}, i.e. $A^0\in \mathfrak{A}_{ad}$.

Since $y_d\in D(A^\ast)$, it follows that $f\in H^{-1}(\Omega)$. Hence, $(A^0,y_{d})$ is an admissible pair to the problem \eqref{A3a.1}--\eqref{A3aa.1}. Taking into account that $I(A^0,y_{d})=0$, we finally conclude: the pair $(A^0,y^0):=(A^{\sharp}+A^\ast,y_{d})$ is a non-variational solution to OCP \eqref{A3a.1}--\eqref{A3aa.1}.


\section{On approximation of non-variational solutions to OCP (\ref{2.3})--(\ref{2.3a})}
\label{Sec 5}
We begin this section with some auxiliary results and notions. Let $\e$ be a small parameter.
Assume that the parameter $\e$ varies within a strictly decreasing sequence of positive real numbers which converge to $0$.
Hereinafter in this section,  for any subset $E\subset\Omega$, we denote by
$|E|$ its $N$-dimensional Lebesgue measure $\mathcal{L}^N(E)$.

For every $\e>0$, let $T_\e:\mathbb{R}\rightarrow \mathbb{R}$ be the truncation function defined by
\begin{equation}
\label{5.00}
T_\e(s)=\max\left\{\min\left\{s,\e^{-1}\right\},-\e^{-1}\right\}.
\end{equation}
The following property of $T_\e$ is well known (see \cite{Kind-Sta}). Let $g\in L^2(\Omega)$ be an arbitrary function. Then we have:
\begin{equation}
\label{5.0}
T_\e(g)\in L^\infty(\Omega)\ \forall\,\e>0\quad\text{and}\quad T_\e(g)\rightarrow g\ \text{strongly in }\ L^2(\Omega).
\end{equation}

Let $A^\ast\in L^2\big(\Omega;\mathbb{S}^N_{skew}\big)$ be a matrix mentioned in the control constraints \eqref{2.3c}. For a given sequence $\left\{\e>0\right\}$, we define the cut-off operators ${T}_\e:\mathbb{S}^N_{skew}\rightarrow \mathbb{S}^N_{skew}$ as follows
${T}_\e (A^\ast)=\left[T_\e(a^\ast_{ij})\right]_{i,j=1}^N$ for every $\e>0$. We associate with such operators the following
set of subdomains $\left\{\Omega_\e\right\}_{\e>0}$ of $\Omega$
\begin{equation}
\label{5.0a}
\Omega_\e=\Omega\setminus Q_\e,\quad\forall\,\e>0,
\end{equation}
where
\begin{equation}
\label{5.0b}
Q_\e=\mathrm{closure}\,\left\{x\in\Omega\ :\ \|A^\ast(x)\|_{\mathbb{S}^N_{skew}}:=\max_{1\le i<j\le N}\left|a^\ast_{ij}(x)\right|\ge\e^{-1}\right\}.
\end{equation}
\begin{definition}
\label{Def 5.1}
We say that a matrix $A^\ast\in L^2\big(\Omega;\mathbb{S}^N_{skew}\big)$ is of the $\mathfrak{F}$-type, if there exists a strictly decreasing sequence of positive real numbers $\left\{\e\right\}$ converging to $0$ such that the corresponding collection of
sets $\left\{\Omega_\e\right\}_{\e>0}$, defined by \eqref{5.0a}, possesses the following properties:
\begin{enumerate}
\item[(i)] $\Omega_\e$ are  open connected subsets of $\Omega$ with Lipschitz boundaries for which there exists a positive value $\delta>0$ such that
    $$
    \partial\Omega\subset \partial\Omega_\e\quad\text{ and }\quad\mathrm{dist}\, (\Gamma_\e,\partial\Omega)>\delta,\quad \forall\,\e>0,
    $$
    where $\Gamma_\e=\partial\Omega_\e\setminus\partial\Omega$.
\item[(ii)]  The surface measure of the boundaries of holes $Q_\e=\Omega\setminus\Omega_\e$ is small enough in the following sense:
\begin{equation}
\label{5.13c}
\mathcal{H}^{N-1}(\Gamma_\e)= o(\e)\quad \forall\,\e>0.
\end{equation}
\item[(iii)] For each matrix $A\in L^2(\Omega;\mathbb{M}^N)$ such that $A^{skew}\preceq A^\ast$ a.e. in $\Omega$, and for each element
$h\in D(A)$, there is a constant $c=c(h)$ depending on $h$ and independent of $\e$ such that
\begin{equation}
\label{5.0c}
\left|\int_{\Omega\setminus\Omega_\e} \big(\nabla \varphi,A^{skew}\nabla h\big)_{\mathbb{R}^N}\,dx\right|\le c(h)\sqrt{\frac{|\Omega\setminus\Omega_\e|}{\e}} \left(\int_{\Omega\setminus\Omega_\e} |\nabla \varphi|^2_{\mathbb{R}^N}\,dx\right)^{1/2}
\end{equation}
 for all $\varphi\in C^\infty_0(\mathbb{R}^N)$.
\end{enumerate}
\end{definition}

Thus, if $A^\ast$ is of the $\mathfrak{F}$-type, each of the sets $\Omega_\e$ is locally located on one side of its Lipschitz boundary $\partial\Omega_\e$. Moreover, in this case the boundary $\partial\Omega_\e$ can be divided into two parts $\partial\Omega_\e=\partial\Omega\cup\Gamma_\e$. Observe also that  if $A^\ast\in L^\infty\big(\Omega;\mathbb{S}^N_{skew}\big)$ then the estimate \eqref{5.0c} is obviously true for all matrices $A\in L^2(\Omega;\mathbb{M}^N)$ such that $A^{skew}\preceq A^\ast$.

\begin{remark}
\label{Rem 5.1.0}
As immediately follows from Definition~\ref{Def 5.1}, the sequence of perforated domains $\left\{\Omega_\e\right\}_{\e>0}$ is monotonically expanding, i.e., $\Omega_{\e_k} \subset \Omega_{\e_{k+1}}$ for all $\e_k>\e_{k+1}$, and perimeters of $Q_\e$ tend to zero as $\e\to 0$.
Moreover, because of the structure of subdomains $Q_\e$ (see (\ref{5.0b})) and $L^2$-property of the matrix $A^\ast$, we have
$$
\frac{\left|\Omega\setminus\Omega_\e\right|}{\e^2}\le \int_{\Omega\setminus\Omega_\e} \|A^\ast(x)\|^2_{\mathbb{S}^N_{skew}}\,dx,\ \forall\,\e>0\quad\text{and}\quad
\lim_{\e\to 0}\|A^\ast\|_{L^2\left(\Omega\setminus\Omega_\e;\mathbb{S}^N_{skew}\right)}=0.
$$
This entails the property: $\left|\Omega\setminus\Omega_\e\right|=o(\e^2)$ and, hence,
$\lim_{\e\to 0}\left|\Omega_\e\right|=|\Omega|$.  Besides, in view of the condition~(ii) of Definition~\ref{Def 5.1}, we have
\begin{equation}
\label{5.13cc}
\frac{\e\mathcal{H}^{N-1}(\Gamma_\e)}{|\Omega\setminus\Omega_\e|}=O(1).
\end{equation}
\end{remark}

\begin{remark}
\label{Rem 5.1}
As follows from \cite{Cioran}, $\mathfrak{F}$-property of the skew-symmetric matrix $A^\ast$ implies the so-called strong connectedness of the sets $\left\{\Omega_\e\right\}_{\e>0}$ which means the existence of extension operators $P_\e$ from $H^1_0(\Omega_\e;\partial\Omega)$ to $H^1_0(\Omega)$ such that, for some positive constant $C$ independent of $\e$,
\begin{equation}
\label{5.2aa}
\left\|\nabla\left(P_\e y\right)\right\|_{L^2(\Omega;\mathbb{R}^N)}\le C \left\|\nabla  y \right\|_{L^2(\Omega_\e;\mathbb{R}^N)},\quad\forall\, y\in H^1_0(\Omega_\e;\partial\Omega).
\end{equation}
\end{remark}
\begin{remark}
\label{Rem 5.13}
It is easy to see that in view of the conditions (1)--(ii) of Definition~\ref{Def 5.1} and the Sobolev Trace Theorem \cite{Adams}, for all $\e>0$ small enough, the inequality
\begin{equation}
\label{5.13d}
\|\varphi\|_{L^2(\Gamma_\e)}\le \frac{C}{\sqrt{\mathcal{H}^{N-1}(\Gamma_\e)}}\, \|\varphi\|_{H^1_0(\Omega_\e;\partial\Omega)},\quad \forall\varphi\in C^\infty_0(\Omega)
\end{equation}
holds true with a constant $C=C(\Omega)$ independent of $\e$.
\end{remark}

As a direct consequence of Definition \ref{Def 5.1}, we have the following result.
\begin{proposition}
\label{Prop 5.3}
Assume that $A^\ast\in L^2\big(\Omega;\mathbb{S}^N_{skew}\big)$ is of the $\mathfrak{F}$-type.
Let $\left\{\Omega_\e\right\}_{\e>0}$ be a sequence of perforated domains of $\Omega$ given by \eqref{5.0b}, and let
$\left\{\chi_{\Omega_\e}\right\}_{\e>0}$ be the corresponding sequence of characteristic functions. Then
\begin{equation}
\label{5.3a}
\chi_{\Omega_\e}\rightarrow \chi_{\Omega}\quad\text{strongly in }\ L^2(\Omega)\ \text{ and weakly-$\ast$ in }\ L^\infty(\Omega).
\end{equation}
\end{proposition}
\begin{proof}
As immediately follows from Definition \ref{Def 5.1}, the sequence $\left\{\chi_{\Omega_\e}\right\}_{\e>0}$ is monotonically increasing, i.e., $\chi_{\Omega_{\e_k}}\le \chi_{\Omega_{\e_{k+1}}}$ almost everywhere in $\Omega$ provided $\e_k>\e_{k+1}$. Taking into account the following representation for the cut-off operators
\[
\|{T}_\e (A^\ast(x))\|_{\mathbb{S}^N_{skew}}= \chi_{\Omega_\e}(x)\|A^\ast(x)\|_{\mathbb{S}^N_{skew}}+(1-\chi_{\Omega_\e}(x))\e^{-1},\quad\forall\,\e>0.
\]
and the condition \eqref{5.0}$_2$, we may suppose, within a subsequence, that
\begin{align*}
&\Big(\chi_{\Omega_\e}(x)\|A^\ast(x)\|_{\mathbb{S}^N_{skew}}+(1-\chi_{\Omega_\e}(x))\e^{-1}\Big)\\
=\ &\Big(\chi_{\Omega_\e}(x)\|A^\ast(x)\|_{\mathbb{S}^N_{skew}}+\chi_{\Omega\setminus\Omega_\e}(x)\e^{-1}\Big)\rightarrow \|A^\ast(x)\|_{\mathbb{S}^N_{skew}}\ \text{a.e. in }\ \Omega\ \text{ as }\ \e\to 0,\\
\text{and }\ & \left|\Omega\setminus\Omega_\e\right|:=\mathcal{L}^N\left(\Omega\setminus\Omega_\e\right)\stackrel{\text{by Remark \ref{Rem 5.1.0}}}{=}\, o(\e^2)\rightarrow 0\quad\text{as }\ \e\to 0.
\end{align*}
Hence, in view of the monotonicity property of $\left\{\chi_{\Omega_\e}\right\}_{\e>0}$, we finally obtain (see \cite{{Delfour:01}})
\[
\chi_{\Omega_\e}\rightarrow\chi_\Omega\quad\text{a.e. in }\ \Omega,\ \text{ and, hence, }\
\chi_{\Omega_\e}\rightarrow\chi_\Omega\ \text{strongly in }\ L^1(\Omega).
\]
Since the strong convergence of characteristic functions in $L^1(\Omega)$ implies their strong convergence in $L^2(\Omega)$, this concludes the proof.
\end{proof}

\begin{definition}
\label{Def 5.3.1}
We say that a sequence $\left\{y_\e\in H^1_0(\Omega_\e;\partial\Omega)\right\}_{\e>0}$ is weakly convergent in variable spaces $H^1_0(\Omega_\e;\partial\Omega)$ if there exists an element $y\in H^1_0(\Omega)$ such that
\[
\lim_{\e\to 0}\int_{\Omega_\e} \left(\nabla y_\e,\nabla \varphi\right)_{\mathbb{R}^N}\,dx=
\int_{\Omega} \left(\nabla y,\nabla \varphi\right)_{\mathbb{R}^N}\,dx,\quad\forall\,\varphi\in C^\infty_0(\Omega)
\]
\end{definition}
\begin{remark}
\label{Rem 5.3.1}
Let $y^\ast\in H^1_0(\Omega)$ be a weak limit in $H^1_0(\Omega)$  of the extended functions  $\left\{P_\e y_\e\in H^1_0(\Omega)\right\}_{\e>0}$.
Since
\begin{gather*}
\int_{\Omega} \left(\nabla y,\nabla \varphi\right)_{\mathbb{R}^N}\,dx=\lim_{\e\to 0}\int_{\Omega_\e} \left(\nabla y_\e,\nabla \varphi\right)_{\mathbb{R}^N}\,dx =
\lim_{\e\to 0}\int_{\Omega} \left(\nabla \left(P_\e y_\e\right),\nabla \varphi\right)_{\mathbb{R}^N}\chi_{\Omega_\e}\,dx\\\, \stackrel{\text{by \eqref{5.3a} and \eqref{5.2aa}}}{=}\,
\int_{\Omega} \left(\nabla y^\ast,\nabla \varphi\right)_{\mathbb{R}^N}\,dx,\quad\forall\,\varphi\in C^\infty_0(\Omega),
\end{gather*}
it follows that
\[
\lim_{\e\to 0}\int_{\Omega_\e} \left(\nabla y_\e,\nabla \varphi\right)_{\mathbb{R}^N}\,dx =
\lim_{\e\to 0}\int_{\Omega} \left(\nabla \left(P_\e y_\e\right),\nabla \varphi\right)_{\mathbb{R}^N}\,dx
\]
and, hence, the weak limit in the sense of Definition \ref{Def 5.3.1} does not depend on the choice of extension operators $P_\e:H^1_0(\Omega_\e;\partial\Omega)\to H^1_0(\Omega)$ with the properties \eqref{5.2aa}.
\end{remark}

Let us consider the following sequence of regularized OCPs associated with perforated domains $\Omega_\e$
\begin{equation}
\label{5.4} \left\{\ \left<\inf_{(A,v,y)\in\Xi_\e}
I_\e(A,v,y)\right>,\quad \e\to 0 \right\},
\end{equation}
where
\begin{gather}
\label{5.5a}
I_\e(A,v,y):=\left\|y-y_d\right\|^2_{L^2(\Omega_\e)}
 + \int_{\Omega_\e}\left(\nabla y, A^{sym}\nabla y\right)_{\mathbb{R}^N}\,dx + \frac{1}{\e^\sigma}\|v\|^2_{H^{-\frac{1}{2}}(\Gamma_\e)},\\[2ex]
\label{5.5b}
\Xi_\e=\left\{(A,v,y)\ \left|\
\begin{array}{c}
-\div\big(A\nabla y\big) = f_\e\quad\text{in }\ \Omega_\e,\\[1ex]
y=0\text{ on }\partial\Omega,\quad
\partial y/ \partial \nu_{A}=v\ \text{on }\Gamma_\e,\\[1ex]
v\in H^{-\frac{1}{2}}(\Gamma_\e),\ y\in H^1_0(\Omega_\e;\partial\Omega),\\[1ex]
A=A^{sym}+A^{skew},\\[1ex]
A\in \mathfrak{A}^\e_{ad}=\mathfrak{A}_{ad,1}\oplus \mathfrak{A}^\e_{ad,2},\
\mathfrak{A}^\e_{ad,2}=U_{a,2}\cap U_{b,2}^\e,\\[1ex]
U_{b,2}^\e=\big\{ A^{skew}=[a_{i\,j}]\in L^2(\Omega;\mathbb{S}^N_{skew})\ :\\[1ex]
\ A^{skew}(x)\preceq A^\ast(x)\ \text{a.e. in }\ \Omega_\e\big\}.
\end{array}
\right.\right\}.
\end{gather}
Here, $y_d\in L^2(\Omega)$ and $f_\e\in
L^2(\Omega)$ are given functions,  $\nu$ is the outward normal unit vector at $\Gamma_\e$ to $\Omega_\e$,
$v\in H^{-\frac{1}{2}}(\Gamma_\e)$ is considered as a fictitious control, and $\sigma$ is a positive number such that
\begin{equation}
\label{5.5c}
\e^{-\sigma}\mathcal{H}^{N-1}(\Gamma_\e)\rightarrow 0\quad\text{as }\ \e\to 0\quad(\text{see \eqref{5.13c}}).
\end{equation}

Using the fact that $A\in L^\infty(\Omega_\e;\mathbb{M}^N)$ for every $\e>0$ and each $A\in \mathfrak{A}^\e_{ad}$, and proceeding as in the proof of Theorem \ref{Th 4.1}, we arrive at the following obvious result.
\begin{theorem}
\label{Th 5.6}
For every $\e>0$ the problem $\left<\inf_{(A,v,y)\in\Xi_\e}
I_\e(A,v,y)\right>$ admits at least one minimizer  $(A^0_\e,v^0_\e,y^0_\e)\in\Xi_\e$.
\end{theorem}

In order to study the asymptotic behavior of the sequences of admissible solutions
$\left\{(A_\e,v_\e,y_\e)\in  \Xi_\e\subset \mathfrak{A}^\e_{ad}\times H^{-\frac{1}{2}}(\Gamma_\e)\times H^1_0(\Omega_\e;\partial\Omega)\right\}_{\e>0}$
in the scale of variable spaces,  we adopt the following concept.
\begin{definition}
\label{Def 5.10} We
say that a sequence
$\left\{(A_\e,v_\e,y_\e)\in  \Xi_\e\right\}_{\e>0}$ weakly converges to a pair $(A,y)\in \mathfrak{A}_{ad}\times H^1_0(\Omega)$ in the scale of spaces
\begin{equation}
\label{5.10a}
\left\{L^2(\Omega;\mathbb{M}^N)\times H^{-\frac{1}{2}}(\Gamma_\e)\times H^1_0(\Omega_\e;\partial\Omega)\right\}_{\e>0},
\end{equation}
(shortly, $(A_\e,v_\e,y_\e)\,\stackrel{w}{\rightarrow}\, (A,y)$), if
\begin{align}
\label{5.10aa}
A_\e:=A_\e^{sym}+A_\e^{skew}&\rightarrow {A}^{sym}+{A}^{skew}=:{A}\ \text{ in }\ L^2(\Omega;\mathbb{M}^N),\\
\label{4.5.1ab}
A_\e^{sym}&\rightarrow {A}^{sym}\quad\text{in }\ L^p(\Omega;\mathbb{S}^N_{sym}),\ \forall\,p\in [1,+\infty),\\
\label{4.5.1ac}
A_\e^{sym}\,&\stackrel{\ast}{\rightharpoonup}\, {A}^{sym}\quad\text{in }\ L^\infty(\Omega;\mathbb{S}^N_{sym}),\\
\label{4.5.1ad} A_\e^{skew}&\rightarrow {A}^{skew}\quad\text{in }\ L^2(\Omega;\mathbb{S}^N_{skew}),\\
\label{5.10b}
 y_\e&\rightharpoonup y\ \text{ in }\ H^1_0(\Omega_\e;\partial\Omega), \\
\label{5.10c}
\text{and}\quad \sup_{\e>0}&\frac{1}{\mathcal{H}^{N-1}(\Gamma_\e)}\,\|v_\e\|^2_{H^{-\frac{1}{2}}(\Gamma_\e)}<+\infty.
\end{align}
\end{definition}

We are now in a position to state the main result of this section.
\begin{theorem}
\label{Th 5.13}
Assume that the matrix $A^\ast\in L^2\big(\Omega;\mathbb{S}^N_{skew}\big)$ is of the $\mathfrak{F}$-type and such that
\begin{equation}
\label{5.13.1}
\begin{array}{c}
\text{the equality }\quad[y,y]_{A}=0\quad\text{ does not hold in }\  D(A)\\
\text{for all $A\in \mathfrak{A}_{ad}$ with $A^{skew}= A^\ast$ a.e. in $\Omega$.}
\end{array}
\end{equation}
Let
$\left\{\Omega_\e\right\}_{\e>0}$ be a sequence of perforated subdomains of $\Omega$ associated with matrix $A^\ast$.
Let $f\in H^{-1}(\Omega)$ and  $y_d\in  L^2(\Omega)$ be given distributions.
Then the original optimal control problem $\left<\inf_{(A,y)\in\Xi} I(A,y)\right>$, where the sequence $\left\{f_\e\in L^2(\Omega)\right\}_{\e>0}$ is such that $\chi_{\Omega_\e}f_\e\rightarrow f$ strongly in $H^{-1}(\Omega)$, is
variational limit of the sequence \eqref{5.4}--\eqref{5.5b} as the parameter $\e$ tends to zero.
\end{theorem}
\begin{remark}
\label{Rem 5.13.1} As follows from Theorem~\ref{Th 4.18}, if there exists a matrix $A\in \mathfrak{A}_{ad}$ and an element $\widehat{y}\in D(A)$ such that $[\widehat{y},\widehat{y}\,]_A\ne 0$, then OCP \eqref{2.3}--\eqref{2.3a} may admit a non-variational solution. So, \eqref{5.13.1} can be interpreted as a necessary (but not sufficient) condition of the existence of non-variational solutions to OCP \eqref{2.3}--\eqref{2.3a}. On the other hand, condition \eqref{5.13.1} may imply the existence of at least one pair $(A,h^\ast)\in \mathfrak{A}_{ad}\times D(A)$ such that $h^\ast\not\in L^\infty(\Omega)$ and $h^\ast$ is a
solution to homogeneous problem \eqref{4.17a}. It means that in this case the linear form
\[
[h^\ast,\varphi]= \int_\Omega \big(\nabla \varphi,A^{skew}(x)\nabla h^\ast\big)_{\mathbb{R}^N}\,dx,\quad  \forall\,\varphi\in C^\infty_0(\mathbb{R}^N)
\]
has a non-trivial extension onto the entire set $D(A)$ following the rule \eqref{2.8a}.
\end{remark}
\begin{proof}[Proof of Theorem~\ref{Th 5.13}]
Since each of the optimization problems $\left<\inf_{(A,v,y)\in\Xi_\e}
I_\e(A,v,y)\right>$ lives in the corresponding space $\mathfrak{A}^\e_{ad}\times H^{-\frac{1}{2}}(\Gamma_\e)\times H^1_0(\Omega_\e;\partial\Omega)$, we have to show that in this case  all conditions of Definition~\ref{Def 1.4} hold true. To do so, we divide this proof into three steps.

Step~1. We show that the space $\mathfrak{A}_{ad}\times H^1_0(\Omega)$ possesses the weak approximation property with respect to the weak convergence in the scale of spaces \eqref{5.10a}. Indeed, let $\delta=0$ and let $(A,y)\in \mathfrak{A}_{ad}\times H^1_0(\Omega)$  be an arbitrary pair. We define $h$ as an element of $C^\infty_0(\Omega)$ such that
\begin{equation}
\label{5.13.1a}
\mathrm{div}\,\left(A^{sym}\nabla h + A^{skew}\nabla h\right)\in L^2(\Omega).
\end{equation}
Hence, $h\in D(A)$. As a result, we construct the sequence $$\left\{(A_\e,v_\e,y_\e)\in L^2(\Omega;\mathbb{M}^N)\times H^{-\frac{1}{2}}(\Gamma_\e)\times H^1_0(\Omega_\e;\partial\Omega)\right\}_{\e>0}$$
as follows
\[
A_\e=A,\quad v_\e=\frac{\partial h}{\partial\nu_{A}}\ \text{ on }\ {\Gamma_\e},\quad\text{and }\ y_\e=y,\quad\forall\,\e>0.
\]
Here,
$\frac{\partial h}{ \partial \nu_A}=\sum_{i,j=1}^N \Big(a_{ij}(x)\Big)\frac{\partial h}{\partial x_j} \cos(\nu,x_i)$,
$\cos(n,x_i)$ is the $i$-th directing cosine of $\nu$, and $\nu$ is the outward unit normal vector at $\Gamma_\e$ to $\Omega_\e$.

In view of \eqref{5.3a}, we have $\chi_{\Omega_\e}\,\stackrel{\ast}{\rightharpoonup}\, \chi_{\Omega}$ in $L^\infty(\Omega)$. Hence,
\begin{align*}
\lim_{\e\to 0}\int_{\Omega_\e}\left(\nabla\varphi,\nabla y_\e\right)_{\mathbb{R}^N}\,dx &=
\lim_{\e\to 0}\int_{\Omega}\left(\nabla\varphi,\nabla y\right)_{\mathbb{R}^N}\chi_{\Omega_\e}\,dx\\&=
\int_{\Omega}\left(\nabla\varphi,\nabla y\right)_{\mathbb{R}^N}\,dx\quad\forall\,\varphi\in C^\infty_0(\Omega),
\end{align*}
i.e., $y_\e \rightharpoonup y$ in $H^1_0(\Omega_\e;\partial\Omega)$ as $\e\to 0$.

It remains to show that the sequence $\left\{v_\e\in H^{-\frac{1}{2}}(\Gamma_\e)\right\}_{\e>0}$ is bounded in the sense of Definition~\ref{Def 5.10}. Following Green's identity, for an arbitrary $\varphi\in C^\infty_0(\Omega)$, we get
\begin{align*}
\Big| \Big<\frac{\partial h}{\partial\nu_{A}},&\varphi\Big>_{H^{-\frac{1}{2}}(\Gamma_\e);H^{\frac{1}{2}}(\Gamma_\e)}\Big|\le
\left|\int_{Q_\e} \mathrm{div}\,\left(A^{sym}(x)\nabla h + A^{skew}(x)\nabla h\right)\varphi\,dx\right|\\
&+\left|\int_{Q_\e} \left(\nabla\varphi, A^{sym}(x)\nabla h + A^{skew}(x)\nabla h\right)_{\mathbb{R}^N}\,dx\right|\\
\le\ & \left(\int_{Q_\e} |\mathrm{div}\,\left(A^{sym}(x)\nabla h + A^{skew}(x)\nabla h\right)|^2\,dx\right)^{1/2}\|\varphi\|_{L^2(Q_\e)}\\
&+\beta\|\nabla h\|_{L^2(Q_\e;\mathbb{R}^N)}\|\nabla \varphi\|_{L^2(Q_\e;\mathbb{R}^N)}\\ &\stackrel{\text{by \eqref{5.0c}}}{+}\,
c(h)\sqrt{\frac{|\Omega\setminus\Omega_\e|}{\e}} \left(\int_{\Omega\setminus\Omega_\e} |\nabla \varphi|^2_{\mathbb{R}^N}\,dx\right)^{1/2}\\ \le\ & \left( I_1+I_2 +I_3\right) \|\varphi\|_{H^1(\Omega\setminus\Omega_\e)}.
\end{align*}
Since $\left|\Omega\setminus\Omega_\e\right|=o(\e^2)$ by the $\mathfrak{F}$-type properties of $A^\ast$,
it follows that there exists a suitable change of variables and a constant $C>0$ independent of $\e$ such that
\begin{align}
I_2=\beta\|\nabla h\|_{L^2(Q_\e;\mathbb{R}^N)}=\beta&\left(C\frac{|\Omega\setminus\Omega_\e|}{\e}\int_{\Omega\setminus\Omega_1}\|\nabla h(y)\|^2_{\mathbb{R}^N}\,dy\right)^{1/2}\notag\\
\,\stackrel{\text{by \eqref{5.13cc}}}{\le}&\, C_1 \sqrt{\mathcal{H}^{N-1}(\Gamma_\e)}\,
\|h\|_{H^1(\Omega)}.\label{5.12a}
\end{align}
Following the similar arguments, in view of \eqref{5.13.1a}, we get
\[
I_1=\Big\|\mathrm{div}\,\left(\nabla h + A(x)\nabla h\right)\Big\|_{L^2(Q_\e)}\le C_2(h)\sqrt{\mathcal{H}^{N-1}(\Gamma_\e)}.
\]
As a result, summing up the previous inequalities, we come to the following conclusion: there exists a constant $C=C(h)$ independent of $\e$ such that
\[
\frac{1}{\sqrt{\mathcal{H}^{N-1}(\Gamma_\e)}}\left<\frac{\partial h}{\partial\nu_{A}},\varphi\right>_{H^{-\frac{1}{2}}(\Gamma_\e);H^{\frac{1}{2}}(\Gamma_\e)}
\le C(h)\|\varphi\|_{H^1(\Omega\setminus\Omega_\e)}
\,\quad\forall\,\varphi\in C^\infty_0(\Omega).
\]
Hence,
\begin{equation}
\label{5.13}
\sup_{\e> 0}\left( \frac{1}{\sqrt{\mathcal{H}^{N-1}(\Gamma_\e)}}\, \Big\|\frac{\partial h}{\partial\nu_{A}}\Big\|_{H^{-\frac{1}{2}}(\Gamma_\e)}\right)\le C.
\end{equation}
Thus, the weak approximation property is proved.

Step~2. We show on this step that condition (ddd) of Definition~\ref{Def 1.4} holds true with $\delta=0$.  Let $(A,y)\in\Xi$ be an arbitrary admissible pair to the original OCP \eqref{2.3}--\eqref{2.3a}.
We will indicate two cases.
\begin{enumerate}
\item[Case~1.] The set $L(A)$, defined in \eqref{4.17aa}, is a singleton. It means that $h\equiv 0$ is a unique solution of homogeneous problem \eqref{4.17a};
\item[Case~2.] The set $L(A)$ is not a singleton. So, we suppose that the set $L(A)$ is a linear subspace of $H^1_0(\Omega)$ and it contains at least one non-trivial element of $D(A)\subset H^1_0(\Omega)$.
\end{enumerate}

We start with the Case~2. Let $h\in D(A)$ be a element of the set $L(A)$ such that $h$ is a non-trivial solution of homogeneous problem \eqref{4.17a}. In the sequel, the choice of element $h\in L(A)$ will be specified (see \eqref{5.22}). Then we construct a $(\Gamma,0)$-realizing sequence $\left\{(A_\e,v_\e,y_\e)\in\Xi_\e\right\}_{\varepsilon>0}$ in the following way:
\begin{enumerate}
\item[(j)] $A_\e=A$ for all $\e>0$.  In view of definition of the set $\mathfrak{A}_{ad}^\e$, we obviously have that $\left\{A_\e\in \mathfrak{A}_{ad}^\e\subset L^2(\Omega;\mathbb{M}^N)\right\}_{\e>0}$ is a sequence of admissible controls to the problems \eqref{5.4}.
    \begin{remark}
    \label{Rem 5.13a}
    Note that in this case the properties \eqref{5.10aa}--\eqref{4.5.1ad} are obviously true for the sequence $\left\{A_\e\right\}_{\e>0}$.
    \end{remark}
 \item[(jj)] Fictitious controls $\left\{v_\e\in H^{-\frac{1}{2}}(\Gamma_\e)\right\}_{\e>0}$ are defined as follows
\begin{equation}
\label{5.13a}
v_\e:= w_\e+\frac{\partial h}{\partial\nu_{A_\e}}\quad\,\forall\, \e>0,
\end{equation}
where distributions $w_\e$ are such that
\begin{equation}
\label{5.20a.2}
\sup_{\e> 0}\left( \frac{1}{\sqrt{\mathcal{H}^{N-1}(\Gamma_\e)}}\, \left\|w_\e\right\|_{H^{-\frac{1}{2}}(\Gamma_\e)}\right)\le C
\end{equation}
with some constant $C$ independent of $\e$.
\item[(jjj)] $\left\{y_\e\in H^1_0(\Omega_\e;\partial\Omega)\right\}_{\e>0}$ is the sequence of weak solutions to the corresponding boundary value problems
\begin{gather}
\label{5.13.1aa}
    -\div\big(A\nabla y_\e\big) = -\div\big(A^{sym}\nabla y_\e+A^{skew}\nabla y_\e\big) =f_\e\quad\text{in }\ \Omega_\e,\\
\label{5.13.2aa}
y_\e=0\text{ on }\partial\Omega,\quad
\partial y_\e/ \partial \nu_{A}=v_\e\ \text{on }\Gamma_\e.
\end{gather}
\end{enumerate}
Since $A={T}_\e (A)$ whenever $x\in\Omega_\e$ for every $\e>0$, it means that $A\in L^\infty(\Omega_\e;\mathbb{M}^N)$.
Hence, due to the Lax-Milgram lemma and the superposition principle, the sequence $\left\{y_\e\in H^1_0(\Omega_\e;\partial\Omega)\right\}_{\e>0}$ is defined in a unique way and for every $\e>0$ we have the following decomposition
$y_\e=y_{\e,1}+y_{\e,2}$, where $y_{\e,1}$ and $y_{\e,2}$ are elements of $H^1_0(\Omega_\e)$  such that
\begin{alignat}{2}
   \int_{\Omega} \big(\nabla \varphi,A^{sym}\nabla y_{\e,1}&+A^{skew}\nabla y_{\e,1}\big)_{\mathbb{R}^N}\chi_{\Omega_\e}\,dx
=\int_{\Omega}f_\e \chi_{\Omega_\e}\varphi\,dx\notag\\
\label{5.14}
 & +
 \left< w_\e,\varphi\right>_{H^{-\frac{1}{2}}(\Gamma_\e);H^{\frac{1}{2}}(\Gamma_\e)},
 \quad\forall\,\varphi\in C^\infty_0(\Omega;\partial\Omega),\\
\notag
 \int_{\Omega} \big(\nabla \varphi,A^{sym}\nabla y_{\e,2}&+A^{skew}\nabla y_{\e,2}\big)_{\mathbb{R}^N}\chi_{\Omega_\e}\,dx\\
&=\left<\frac{\partial h}{\partial\nu_{A}},\varphi\right>_{H^{-\frac{1}{2}}(\Gamma_\e);H^{\frac{1}{2}}(\Gamma_\e)},\
\forall\,\varphi\in C^\infty_0(\Omega;\partial\Omega).\label{5.15}
\end{alignat}
\begin{remark}
\label{Rem 5.2}
Hereinafter, we suppose that the functions $y_\e$ of $H^1_0(\Omega_\e,\partial\Omega)$ are extended by operators $P_\e$ outside of $\Omega_\e$.
\end{remark}

By the skew-symmetry property of $A^{skew}\in L^\infty(\Omega_\e;\mathbb{S}^N_{skew})$, we have
$$
\int_{\Omega} \big(\nabla y_{\e,i},A^{skew}\nabla y_{\e,i}\big)_{\mathbb{R}^N}\chi_{\Omega_\e}\,dx=0,\quad  i=1,2.
$$
Then \eqref{5.14}--\eqref{5.15} lead us to the energy equalities
\begin{alignat}{2}
   \int_{\Omega} \big(\nabla y_{\e,1}, A^{sym}\nabla y_{\e,1}\big)_{\mathbb{R}^N}\chi_{\Omega_\e}\,dx
&=\int_{\Omega}f_\e \chi_{\Omega_\e}y_{\e,1}\,dx\notag\\
\label{5.16}
 &+
 \left< w_\e,y_{\e,1}\right>_{H^{-\frac{1}{2}}(\Gamma_\e);H^{\frac{1}{2}}(\Gamma_\e)},\\
\label{5.17}
 \int_{\Omega} \big(\nabla y_{\e,2}, A^{sym}\nabla y_{\e,2}\big)_{\mathbb{R}^N}\chi_{\Omega_\e}\,dx
&=\left<\frac{\partial h}{\partial\nu_{A}},y_{\e,2}\right>_{H^{-\frac{1}{2}}(\Gamma_\e);H^{\frac{1}{2}}(\Gamma_\e)}.
\end{alignat}
By the initial assumptions, we have $h\in L(A)$. Then the condition (iii) of Definition~\ref{Def 5.1} implies that
(for the details we refer to \eqref{5.12a})
\begin{multline*}
\left|\left<\frac{\partial h}{\partial\nu_{A}},\varphi\right>_{H^{-\frac{1}{2}}(\Gamma_\e);H^{\frac{1}{2}}(\Gamma_\e)}\right|=
\left|\int_{\Omega\setminus\Omega_\e} \big(\nabla \varphi,A^{sym}\nabla h+A^{skew}\nabla h\big)_{\mathbb{R}^N}\,dx\right|\\
\le \sqrt{\frac{|\Omega\setminus\Omega_\e|}{\e}}\left(C_1(h)+C_2(h)\right)\|\varphi\|_{H^1(\Omega\setminus\Omega_\e)}\\
\stackrel{\text{by \eqref{5.13c}}}{\le}\,
C(h)\sqrt{\mathcal{H}^{N-1}(\Gamma_\e)}\|\varphi\|_{H^1(\Omega\setminus\Omega_\e)},\quad\forall\,\varphi \in H^1_0(\Omega)
\end{multline*}
with some constant $C(h)$ independent of $\e$.
Hence,
\begin{equation}
\label{5.20a.1}
\sup_{\e>0}\left(\mathcal{H}^{N-1}(\Gamma_\e)\right)^{-1}\,\Big\|\frac{\partial h}{\partial\nu_{A}}\Big\|^2_{H^{-\frac{1}{2}}(\Gamma_\e)}<C(h)<+\infty.
\end{equation}
Thus, using the continuity of the embedding $H^{\frac{1}{2}}(\Gamma_\e)\hookrightarrow L^2(\Gamma_\e)$ and Sobolev Trace Theorem, we get
\begin{alignat}{2}
\notag
\Big|\left< w_\e,y_{\e,1}\right>_{H^{-\frac{1}{2}}(\Gamma_\e);H^{\frac{1}{2}}(\Gamma_\e)}\Big|&\stackrel{\text{by \eqref{5.20a.2}}}{\le}\, C\,\|y_{\e,1}\|_{ L^2(\Gamma_\e)}\left(\mathcal{H}^{N-1}(\Gamma_\e)\right)^{\frac{1}{2}}\\
&\stackrel{\text{by \eqref{5.13d}}}{\le}\, C_1\,\|y_{\e,1}\|_{H^1_0(\Omega_\e;\partial\Omega)},\label{5.20a}\\
\notag
\Big|\left<\frac{\partial h}{\partial\nu_{A}},y_{\e,2}\right>_{H^{-\frac{1}{2}}(\Gamma_\e);H^{\frac{1}{2}}(\Gamma_\e)}\Big|&\le C\,\|y_{\e,2}\|_{ L^2(\Gamma_\e)}\left(\mathcal{H}^{N-1}(\Gamma_\e)\right)^{\frac{1}{2}}\\
&\stackrel{\text{by \eqref{5.13d}}}{\le}\, C_1 \|y_{\e,2}\|_{H^1_0(\Omega_\e;\partial\Omega)}\label{5.20b}.
\end{alignat}
As a result, we arrive at the following the a priori estimates
\begin{align}
\label{5.18}
   \left(\int_{\Omega} \big\|\nabla y_{\e,1}\big\|^2_{\mathbb{R}^N}\chi_{\Omega_\e}\,dx\right)^{1/2}  &\le \alpha^{-1}\left(\|f_\e\chi_{\Omega_\e}\|_{H^{-1}(\Omega)}
   +C\right),\\
\label{5.19}
\left(\int_{\Omega} \big\|\nabla y_{\e,2}\big\|^2_{\mathbb{R}^N}\chi_{\Omega_\e}\,dx\right)^{1/2} &\le C\alpha^{-1}.
\end{align}
Hence, the sequences
$$\left\{y_{\e,1}\in H^1_0(\Omega_\e;\partial\Omega)\right\}_{\e>0}\quad\text{ and }\quad\left\{y_{\e,2}\in H^1_0(\Omega_\e;\partial\Omega)\right\}_{\e>0}
$$
are weakly compact with respect to the weak convergence in variable spaces \cite{Zh_98}, i.e., we may assume that there exists a couple of functions $\widehat{y}_1$ and $\widehat{y}_2$ in $H^1_0(\Omega)$ such that
\begin{equation}
\label{5.20}
\lim_{\e\to 0}\int_{\Omega} \big(\nabla \varphi,\nabla y_{\e,i}\big)_{\mathbb{R}^N}\chi_{\Omega_\e}\,dx=
\int_{\Omega} \big(\nabla \varphi,\nabla \widehat{y}_{i}\big)_{\mathbb{R}^N}\\,dx,\
\forall\,\varphi\in C^\infty_0(\Omega)
\end{equation}
for $i=1,2$.

Now we can pass to the limit in the integral identities \eqref{5.14}--\eqref{5.15} as $\e\to 0$. Using  \eqref{5.20a.2}, \eqref{5.20},  \eqref{5.20a.1}, $L^2$-property of $A\in \mathfrak{A}_{ad}$, and the fact that $\chi_{\Omega_\e}f_\e\rightarrow f$ strongly in $H^{-1}(\Omega)$, we finally obtain
\begin{gather}
\label{5.21a}
\int_{\Omega} \big(\nabla \varphi,A^{sym}\nabla \widehat{y}_{1}+A^{skew}\nabla \widehat{y}_{1}\big)_{\mathbb{R}^N}\,dx
=\left<f, \varphi\right>_{H^{-1}(\Omega);H^1_0(\Omega)}\\
\label{5.21b}
\int_{\Omega} \big(\nabla \varphi,A^{sym}\nabla \widehat{y}_{2}+A^{skew}\nabla \widehat{y}_{2}\big)_{\mathbb{R}^N}\,dx=0
\end{gather}
for every $\varphi\in C^\infty_0(\Omega)$. Hence, $\widehat{y}_1$ and $\widehat{y}_2$ are weak solutions to the boundary value problem \eqref{2.1}--\eqref{2.2} and \eqref{4.17a}, respectively. Hence,  $\widehat{y}_2\in L(A)$ and  $\widehat{y}_1\in D(A)$ by Proposition~\ref{Prop 2.9}.  As a result, we arrive at the conclusion: the pair $(A,\widehat{y}_1+h)$ belongs to the set $\Xi$, for every $h\in L(A)$.
Since  by the initial assumptions $(A,y)\in\Xi$, it follows that having set in \eqref{5.13a}
\begin{equation}
\label{5.22}
h=y-\widehat{y}_1,
\end{equation}
we obtain
\begin{equation}
\label{5.23}
h\in L(A)\ \text{ and }\ y_\e=y_{\e,1}+y_{\e,2}\rightharpoonup y\quad\text{ in }\ H^1_0(\Omega_\e;\partial\Omega)\ \text{ as }\ \e\to 0.
\end{equation}
Therefore, in view of \eqref{5.23}, \eqref{5.20a.1}, \eqref{5.20a.2}, and Remark~\ref{Rem 5.13a}, we see that
$$
(A_\e,v_\e,y_\e)\,\stackrel{w}{\rightarrow}\, (A,y)\quad\text{in the sense of Definition~\ref{Def 5.10}}.
$$
Thus,
the properties \eqref{1.7a}--\eqref{1.7b} hold true.
\begin{remark}
It is worth to notice that in the Case~1, we can give the same conclusion, because we originally have $h\equiv 0$. Hence, the solutions to boundary value problems \eqref{5.21a}--\eqref{5.21a} are unique and, therefore, we can claim that $y=\widehat{y}_1$, $\widehat{y}_2=0$, and $h=0$.
\end{remark}

It remains to prove the inequality \eqref{1.7c}. To do so, it is enough to show that
\begin{align}
\notag
I(A,y):= &\left\|y-y_d\right\|^2_{L^2(\Omega)}
 + \int_\Omega\left(\nabla y, A^{sym}\nabla y\right)_{\mathbb{R}^N}\,dx\\
 \notag
 = &\lim_{\e\to 0}I_\e(u_\e,v_\e,y_\e)\\
 = &
 \lim_{\e\to 0}\Big[\left\|y_\e-y_d\right\|^2_{L^2(\Omega_\e)}\\
 & + \int_{\Omega_\e}\left(\nabla y_\e, A^{sym}\nabla y_\e\right)_{\mathbb{R}^N}\,dx + \frac{1}{\e^\sigma}\|v_\e\|^2_{H^{-\frac{1}{2}}(\Gamma_\e)}\Big],\label{5.23a}
\end{align}
where the sequence $\left\{(u_\e,v_\e,y_\e)\in\Xi_\e\right\}_{\varepsilon>0}$ is defined by \eqref{5.13a} and \eqref{5.22}.

In view of this, we make use the following relations
\begin{gather}
\label{5.24}
\left.
\begin{array}{c}
\ds \|v_\e\|^2_{H^{-\frac{1}{2}}(\Gamma_\e)}\le 2\|w_\e\|^2_{H^{-\frac{1}{2}}(\Gamma_\e)}+2\Big\|\frac{\partial h}{\partial\nu_{A}}\Big\|^2_{H^{-\frac{1}{2}}(\Gamma_\e)}<+\infty,\\[2ex]
\ds\lim_{\e\to 0} \frac{1}{\e^\sigma}\|w_\e\|^2_{H^{-\frac{1}{2}}(\Gamma_\e)}\,\stackrel{\text{by \eqref{5.20a.2}}}{\le}\, C
\lim_{\e\to 0} \frac{\mathcal{H}^{N-1}(\Gamma_\e)}{\e^\sigma}=0,\\[2ex]
\ds\lim_{\e\to 0} \frac{1}{\e^\sigma}\left\|\frac{\partial h}{\partial\nu_{A}}\right\|^2_{H^{-\frac{1}{2}}(\Gamma_\e)}\,\stackrel{\text{by \eqref{5.20a.1}}}{\le}\, C
\lim_{\e\to 0} \frac{\mathcal{H}^{N-1}(\Gamma_\e)}{\e^\sigma}=0,\\[1ex]
\ds\lim_{\e\to 0}\left\|y_\e-y_d\right\|^2_{L^2(\Omega_\e)}\,\stackrel{\text{by \eqref{5.3a} and \eqref{5.23}}}{=}\,
\left\|y-y_d\right\|^2_{L^2(\Omega)}.
\end{array}
\right\}
\end{gather}

In order to obtain the convergence
\begin{equation}\label{5.25}
\lim\limits_{\e \to 0}
\int_{\Omega_\e}\left(\nabla y_\e, A^{sym}\nabla y_\e\right)_{\mathbb{R}^N}\,dx =
\int_\Omega\left(\nabla y, A^{sym}\nabla y\right)_{\mathbb{R}^N}\,dx,
\end{equation}
we apply the energy equality
which comes from the condition $(A,y)\in\Xi$
\begin{equation}
\int_\Omega\left(\nabla y, A^{sym}\nabla y\right)_{\mathbb{R}^N}\,dx
=-[y,y]_{A}+\left<f, y\right>_{H^{-1}(\Omega);H^1_0(\Omega)},\label{5.27}
\end{equation}
and make use of the following trick.
It is easy to see that
the integral identity for the weak solutions $y_\e$ to boundary value problems \eqref{5.5b} can be represented in the so-called extended form
\begin{alignat}{2}
   \int_{\Omega} \big(\nabla \varphi,A^{sym}\nabla y_{\e}&+A^{skew}\nabla y_{\e}\big)_{\mathbb{R}^N}\chi_{\Omega_\e}\,dx
=\int_{\Omega}f_\e \chi_{\Omega_\e}\varphi\,dx\notag\\
 & +
 \left< w_\e,\varphi\right>_{H^{-\frac{1}{2}}(\Gamma_\e);H^{\frac{1}{2}}(\Gamma_\e)} + \left<\frac{\partial h}{\partial\nu_{A}},\varphi\right>_{H^{-\frac{1}{2}}(\Gamma_\e);H^{\frac{1}{2}}(\Gamma_\e)}\notag\\
 \label{5.26.1}
 &- \int_{\Omega} \big(\nabla \psi,A^{sym}\nabla h^\ast\big)_{\mathbb{R}^N}\,dx-[h^\ast,\psi]_A,\quad
\forall\,\varphi, \psi\in C^\infty_0(\Omega),
\end{alignat}
where $h^\ast$ is an arbitrary element of $L$. Indeed, because of the equality
\[
\int_{\Omega} \big(\nabla \psi,A^{sym}\nabla h^\ast\big)_{\mathbb{R}^N}\,dx+[h^\ast,\psi]_A\,\stackrel{\text{by \eqref{4.17aa}}}{=}\,0,\quad\forall\,\psi\in C^\infty_0(\Omega),
\]
we have an equivalent identity to the classical definition of the weak solutions of boundary value problem \eqref{5.5b}.

As follows from \eqref{5.20a.1}, \eqref{5.23}, and the Sobolev Trace Theorem, the numerical sequences
$$
\left\{\left< w_\e,y_\e\right>_{H^{-\frac{1}{2}}(\Gamma_\e);H^{\frac{1}{2}}(\Gamma_\e)}\right\}_{\e>0}\quad\text{and}\quad
\left\{ \left<\frac{\partial h}{\partial\nu_{A}},y_\e\right>_{H^{-\frac{1}{2}}(\Gamma_\e);H^{\frac{1}{2}}(\Gamma_\e)}\right\}_{\e>0}
$$
are bounded.
Therefore, we can assume, passing to a subsequence if necessary, that there exists a value $\xi_1\in \mathbb{R}$ such that
\begin{equation}
\label{5.24b}
\left< w_\e,y_\e\right>_{H^{-\frac{1}{2}}(\Gamma_\e);H^{\frac{1}{2}}(\Gamma_\e)}+
\left<\frac{\partial h}{\partial\nu_{A}},y_\e\right>_{H^{-\frac{1}{2}}(\Gamma_\e);H^{\frac{1}{2}}(\Gamma_\e)}
\longrightarrow \xi_1\quad\text{as }\ \e\to 0.
\end{equation}
Since $y_\e\rightharpoonup y$ weakly in $H^1_0(\Omega_\e;\partial\Omega)$ and $y\in D(A)$, it follows that there exists a sequence of smooth functions $\left\{\psi_\e\in C^\infty_0(\Omega)\right\}_{\e>0}$ such that $\psi_\e\rightarrow y$ strongly in $H^1_0(\Omega)$. Therefore, following the extension rule \eqref{2.8a}, we have
\begin{align}
\label{5.26.2.1}
\lim_{\e\to 0}\int_{\Omega} \big(\nabla \psi_\e,A^{sym}\nabla h^\ast\big)_{\mathbb{R}^N}\,dx&=\int_\Omega \big(\nabla y,A^{sym}\nabla h^\ast\big)_{\mathbb{R}^N}\,dx,\\
\label{5.26.2.2}
\lim_{\e\to 0} [h^\ast,\psi_\e]_A&=[h^\ast,y]_A.
\end{align}

Because of the initial supposition \eqref{5.13.1} (see Remark \ref{Rem 5.13.1}), we can assume that the element $h^\ast\in L(A)$ is such that
$$[h^\ast,y]_A+\int_\Omega \big(\nabla y,A^{sym}\nabla h^\ast\big)_{\mathbb{R}^N}\,dx\ne 0.$$
Otherwise, we come into conflict with \eqref{5.13.1}.  So, due to this observation, we specify the choice of element $h^\ast\in L(A)$ as follows
\[
\widehat{h}^\ast=\frac{\xi_1+[y,y]_A}{\xi_2+\xi_3}\, h^\ast,\quad\text{where }\ \xi_3:=\int_\Omega \big(\nabla y,A^{sym}\nabla h^\ast\big)_{\mathbb{R}^N}\,dx,\ \xi_2:=[h^\ast,y]_A,
\]
or, in other words, we aim to ensure the condition
$\xi_1-\xi_2-\xi_3 +[y,y]_A=0$.
As a result, we have:  $\widehat{h}^\ast$ is an element of $L(A)$ such that
\begin{equation}
\label{5.26.3}
\lim_{\e\to 0}\int_{\Omega} \big(\nabla \psi_\e,\nabla \widehat{h}^\ast\big)_{\mathbb{R}^N}\,dx=\xi_2\frac{\xi_1+[y,y]_A}{\xi_2+\xi_3},\quad
\lim_{\e\to 0} [\widehat{h}^\ast,\psi_\e]=\xi_3\frac{\xi_1+[y,y]_A}{\xi_2+\xi_3}.
\end{equation}

Having put $\varphi=y_\e$ and $h^\ast=\widehat{h}^\ast$ in \eqref{5.26.1} and using the fact that
$$
\int_{\Omega} \big(\nabla y_\e,A^{skew} \nabla y_{\e}\big)_{\mathbb{R}^N}\chi_{\Omega_\e}\,dx=0,
$$
we arrive at the following energy equality for the boundary value problem \eqref{5.5b}
\begin{gather}
   \int_{\Omega} \big(\nabla y_\e,A^{sym}\nabla y_{\e}\big)_{\mathbb{R}^N}\chi_{\Omega_\e}\,dx
=\int_{\Omega}f_\e \chi_{\Omega_\e}y_\e\,dx
  +
 \left< w_\e,y_\e\right>_{H^{-\frac{1}{2}}(\Gamma_\e);H^{\frac{1}{2}}(\Gamma_\e)}\notag\\
 + \left<\frac{\partial h}{\partial\nu_{A}},y_\e\right>_{H^{-\frac{1}{2}}(\Gamma_\e);H^{\frac{1}{2}}(\Gamma_\e)}
 \label{5.26.4}
 - \int_{\Omega} \big(\nabla \psi_\e,A^{sym}\nabla \widehat{h}^\ast\big)_{\mathbb{R}^N}\,dx-[\widehat{h}^\ast,\psi_\e]_A.
\end{gather}

As a result, taking into account the properties \eqref{5.3a}, \eqref{5.23}, \eqref{5.26.3}, we can pass to the limit as $\e\to 0$ in \eqref{5.26.4}. This yields
\begin{align}
\notag
\lim_{\e\to 0}\int_{\Omega}& \big(\nabla y_\e,A^{sym}\nabla y_{\e}\big)_{\mathbb{R}^N}\chi_{\Omega_\e}\,dx=\lim_{\e\to 0}\int_{\Omega}f_\e \chi_{\Omega_\e}y_{\e}\,dx\\\notag&+
\lim_{\e\to 0}\left< w_\e,y_{\e}\right>_{H^{-\frac{1}{2}}(\Gamma_\e);H^{\frac{1}{2}}(\Gamma_\e)}+ \lim_{\e\to 0}\left< \frac{\partial h}{\partial\nu_{A}},y_{\e}\right>_{H^{-\frac{1}{2}}(\Gamma_\e);H^{\frac{1}{2}}(\Gamma_\e)}\\\notag&
-\lim_{\e\to 0}\int_{\Omega} \big(\nabla \psi_\e,\nabla \widehat{h}^\ast\big)_{\mathbb{R}^N}\,dx -
\lim_{\e\to 0} [\widehat{h}^\ast,\psi_\e]_A\\
\,\stackrel{\text{by \eqref{5.26.3}}}{=}&\left<f,y\right>_{H^{-1}(\Omega);H^1_0(\Omega)}-[y,y]_A\,\stackrel{\text{by \eqref{5.27}}}{=}\, \int_\Omega\left(\nabla y, A^{sym}\nabla y\right)_{\mathbb{R}^N}\,dx.\label{5.28}
\end{align}
Hence, turning back to \eqref{5.23a}, we see that this relation is a direct consequence of \eqref{5.24} and \eqref{5.28}. Thus,
the sequence $\left\{(u_\e,v_\e,y_\e)\in\Xi_\e\right\}_{\varepsilon>0}$, which is defined by \eqref{5.13a} and \eqref{5.22}, is $(\Gamma,0)$-realizing. The property (ddd)  is established.

Step~3. We prove the property (dd) of Definition~\ref{Def 1.4}. Let
$\left\{(A_k,v_k,y_k)\right\}_{k\in \mathbb{N}}$ be a sequence such that  $(A_k,v_k,y_k)\in
\Xi_{\varepsilon_k}$ for some $\e_k\to 0$ as $k\to\infty$,
\begin{equation}
\label{5.28.28}
\left.
\begin{split}
A_k:=A_k^{sym}+A_k^{skew}&\rightarrow {A}^{sym}+{A}^{skew}=:{A}\ \text{ in }\ L^2(\Omega;\mathbb{M}^N),\\
A_k^{sym}&\rightarrow {A}^{sym}\quad\text{in }\ L^p(\Omega;\mathbb{S}^N_{sym}),\ \forall\,p\in [1,+\infty),\\
A_k^{sym}\,&\stackrel{\ast}{\rightharpoonup}\, {A}^{sym}\quad\text{in }\ L^\infty(\Omega;\mathbb{S}^N_{sym}),\\
A_\e^{skew}&\rightarrow {A}^{skew}\quad\text{in }\ L^2(\Omega;\mathbb{S}^N_{skew}),\\
y_k&\rightharpoonup y\ \text{ in }\ H^1_0(\Omega_{\e_k};\partial\Omega),
\end{split}
\right\}
\end{equation}
and the sequence of fictitious controls $\left\{v_k\in H^{-\frac{1}{2}}(\Gamma_{\e_k})\right\}_{k\in \mathbb{N}}$ satisfies inequality \eqref{5.10c}. In view of Definition~\ref{Def 5.10} it means that
\[
(A_k,v_k,y_k)\,\stackrel{w}{\rightarrow}\, (A,y)\ \text{ as }\ k\to\infty
\]

Our aim is to show that
\begin{equation}
\label{5.29} (A,y)\in \Xi \quad\text{and}\quad I(A,y)\le
\liminf_{k\to\infty}I_{\e_k}(A_k,v_k,y_k).
\end{equation}
Following the arguments of the proof of Theorem~\ref{Th 4.1}, it is easy to show that the limit matrix $A$ is an admissible control to OCP \eqref{2.3}--\eqref{2.3a}, i.e. $A\in \mathfrak{A}_{ad}$.

Since the integral identity
\begin{multline}
\label{5.30}
\int_{\Omega} \big(\nabla \varphi,A_k^{sym}\nabla y_{k}+A_k^{skew}\nabla y_{k}\big)_{\mathbb{R}^N}\chi_{\Omega_{\e_k}}\,dx
=\int_{\Omega}f_{\e_k} \chi_{\Omega_{\e_k}}\varphi\,dx\\+
 \left< v_k,\varphi\right>_{H^{-\frac{1}{2}}(\Gamma_{\e_k});H^{\frac{1}{2}}(\Gamma_{\e_k})},\quad
 \forall\,\varphi\in C^\infty_0(\Omega)
\end{multline}
holds true for every $k\in \mathbb{N}$, we can pass to the limit in \eqref{5.30} as $k\to\infty$ using Definition \ref{Def 5.10}  and the estimate
\[
\Big|\left< v_k,\varphi\right>_{H^{-\frac{1}{2}}(\Gamma_{\e_k});H^{\frac{1}{2}}(\Gamma_{\e_k})}\Big|\le C(\Omega)\,\|\varphi\|_{H^1_0(\Omega)}\left(\mathcal{H}^{N-1}(\Gamma_{\e_k})\right)^{\frac{1}{2}},\quad\forall\,\varphi\in C^\infty_0(\Omega)
\]
coming from inequality \eqref{5.10c}. Then proceeding as on the Step~2, it can easily be shown that the limit pair $(A,y)$ is admissible to OCP \eqref{2.3}--\eqref{2.3a}. Hence, the condition \eqref{5.29}$_1$ is valid.

As for the inequality \eqref{5.29}$_2$, we see that
\begin{equation}
\label{5.30.1}
\lim_{k\to\infty} \left\|y_k-y_d\right\|^2_{L^2(\Omega_{\e_k})}=\lim_{k\to\infty}\left\|(y_k-y_d)\chi_{\Omega_{\e_k}}\right\|^2_{L^2(\Omega)}
=\left\|y-y_d\right\|^2_{L^2(\Omega)}
\end{equation}
by \eqref{5.3a} and compactness of the embedding $H^1_0(\Omega)\hookrightarrow L^2(\Omega)$. In view of the properties \eqref{5.28.28} and \eqref{1.1}, the sequence $\left\{\left(A_k^{sym}\right)^{1/2}\right\}_{k\in \mathbb{N}}$ is obviously bounded in $L^2(\Omega; \mathbb{S}^N_{sym})$. Moreover, taking into account the norm convergence property
\begin{gather*}
\lim_{k\to\infty}\|\left(A_k^{sym}\right)^{1/2}\xi\|^2_{L^2(\Omega;\mathbb{R}^N)}=
\lim_{k\to\infty}\int_{\Omega}\left(\xi, A_k^{sym} \xi\right)_{\mathbb{R}^N}\,dx\\=
\int_\Omega\left(\xi, A^{sym} \xi\right)_{\mathbb{R}^N}\,dx=
\|\left(A^{sym}\right)^{1/2}\xi\|^2_{L^2(\Omega;\mathbb{R}^N)},\quad\forall\,\xi\in \mathbb{R}^N,
\end{gather*}
we can conclude that the sequence $\left\{\left(A_k^{sym}\right)^{1/2}\right\}_{k\in \mathbb{N}}$ strongly converges to $\left(A^{sym}\right)^{1/2}$ in $L^2(\Omega;\mathbb{S}^N_{sym})$. Hence, combining this fact with \eqref{5.28.28}$_5$ and \eqref{5.3a}, we finally obtain
\[
\chi_{\Omega_{\e_k}}\left(A_k^{sym}\right)^{1/2}\nabla y_k\rightharpoonup \chi_{\Omega}\left(A^{sym}\right)^{1/2}\nabla y\quad\text{in }\ L^2(\Omega;\mathbb{R}^N).
\]
As a result, the lower semicontinuity of $L^2$-norm with respect to the weak convergence, immediately leads us to the inequality
\begin{align}
\notag
\liminf_{k\to\infty} &\int_{\Omega_{\e_k}}\left(\nabla y_k, A_k^{sym} \nabla y_k\right)_{\mathbb{R}^N}\,dx=
\liminf_{k\to\infty} \|\chi_{\Omega_{\e_k}}\left(A_k^{sym}\right)^{1/2}\nabla y_k\|^2_{L^2(\Omega;\mathbb{R}^N)}\\
&\ge \|\chi_{\Omega}\left(A^{sym}\right)^{1/2}\nabla y\|^2_{L^2(\Omega;\mathbb{R}^N)}=
\int_{\Omega}\left(\nabla y, A^{sym} \nabla y\right)_{\mathbb{R}^N}\,dx.
\label{5.30.2}
\end{align}

Thus, in order to prove the inequality \eqref{5.29}$_2$, it remains to combine relations \eqref{5.30.1}, \eqref{5.30.2}, and take into account the following estimate
\begin{equation}
\label{5.30.3}
\frac{1}{{(\e_k)}^\sigma}\,\|v_k\|^2_{H^{-\frac{1}{2}}(\Gamma_{\e_k})}
\le C
 \frac{\mathcal{H}^{N-1}(\Gamma_{\e_k})}{(\e_k)^\sigma}\rightarrow 0\quad\text{as }\ k\to\infty.
\end{equation}
The proof is complete.
\end{proof}

In conclusion of this section, we consider the variational properties of OCPs \eqref{5.4}--\eqref{5.5b}. To this end, we apply
Theorem~\ref{Th 1.8}.
\begin{theorem}
\label{Th 5.31}
Let $A^\ast\in L^2\big(\Omega;\mathbb{S}^N_{skew}\big)$ be a matrix of the $\mathfrak{F}$-type such that
\begin{equation}
\label{5.31a}
\text{the equality }\quad[y,y]_A=0\quad\text{ does not hold in }\  D(A).
\end{equation}
Let $y_d\in L^2(\Omega)$ and $f\in
H^{-1}(\Omega)$ be given distributions.
Let $\left\{(A^0_\varepsilon,v^0_\e,y^0_\e)\in \Xi_\varepsilon\right\}_{\e>0}$ be a sequence of optimal solutions to regularized problems \eqref{5.4}--\eqref{5.5b}, where $\chi_{\Omega_\e}f_\e\rightarrow f$ strongly in $H^{-1}(\Omega)$.
Then there exists an optimal pair $(A^0,y^0)\in \mathfrak{A}_{ad}$ to the original OCP \eqref{2.3}--\eqref{2.3a}, which is attainable in the following sense
\begin{gather}
\label{5.32}
(A^0_\varepsilon,v^0_\e,y^0_\e)\,\stackrel{w}{\rightarrow}\, (A^0,y^0)\ \text{ as }\ \e\to 0\\\notag \text{ in variable space }\ L^2(\Omega;\mathbb{M}^N)\times H^{-\frac{1}{2}}(\Gamma_\e)\times H^1_0(\Omega_\e;\partial\Omega),\\
\inf_{(A,y)\in\,\Xi}I(A,y)= I\left(A^0,y^0\right) =\lim_{\e\to
0} I_{\e}(A^0_\e,v^0_\e,y^0_\e)=\lim_{\e\to
0} \inf_{(A,v,y)\in\Xi_\e}
I_\e(A,v,y).\label{5.33}
\end{gather}
\end{theorem}
\begin{proof}
In order to show that this result is a direct consequence of Theorem \ref{Th 1.8}, it is enough to establish the compactness property for the sequence of optimal solutions $\left\{(A^0_\e,v^0_\e,y^0_\e)\in \Xi_\varepsilon\right\}_{\e>0}$ in  the sense of Definition \ref{Def 5.10}.

Let $h\in C^\infty_0(\Omega)$ be a non-zero function such that
$\mathrm{div}\,\left(A^{sym}\nabla h + A^\ast\nabla h\right)\in L^2(\Omega)$, where we assume that $A=A^{sym}+A^\ast$ is an admissible control, $A\in \mathfrak{A}_{ad}$. We set $v_\e=\left.\frac{\partial h}{\partial\nu_A}\right|_{\Gamma_\e}\in H^{-\frac{1}{2}}(\Gamma_\e)$. In view of the initial assumptions and estimate \eqref{5.13}, there is a constant $C>0$ independent of $\e$ such that
\[
 \Big\|\frac{\partial h}{\partial\nu_{A}}\Big\|^2_{H^{-\frac{1}{2}}(\Gamma_\e)}\le C {\mathcal{H}^{N-1}(\Gamma_\e)},
\]

Let $y_\e=y_\e(A_\e,v_\e,f)\in H^1_0(\Omega_\e;\partial\Omega)$ be a corresponding solution to boundary value problem \eqref{5.5b}.
Then following \eqref{5.18}, we come to the estimate
$$\|y_\e\|^2_{H^1_0(\Omega_\e;\partial\Omega)}\le \widetilde{C},$$
where the constant $\widetilde{C}$ is also independent of $\e$. As a result, we get
\begin{align*}
I_{\e}(A^0_\e,v^0_\e,y^0_\e)=&\left\|y^0_\e-y_d\right\|^2_{L^2(\Omega_\e)}
 + \int_{\Omega_{\e}}\left(\nabla y_\e^0, (A_\e^0)^{sym} \nabla y_\e^0\right)_{\mathbb{R}^N}\,dx\\& + \frac{1}{{\e}^\sigma}\|v^0_\e\|^2_{H^{-\frac{1}{2}}(\Gamma_\e)}
\le I_\e(A_\e,v_\e,y_\e)\\\le &(2C_1+\beta)\widetilde{C} +2\|y_d\|^2_{L^2(\Omega)} + C\frac{\mathcal{H}^{N-1}(\Gamma_{\e})}{{\e}^\sigma}.
\end{align*}
Since ${\e}^{-\sigma}\mathcal{H}^{N-1}(\Gamma_{\e})\rightarrow 0$ as $\e\to 0$, it follows that
the minimal values of the cost functional \eqref{5.5a} bounded above
uniformly with respect to $\e$.
Thus, the sequence of optimal solutions $\left\{(A^0_\e, v^0_\e,y^0_\e)\right\}_{\e>0}$ to the problems \eqref{5.4}--\eqref{5.5b} uniformly bounded in $L^2(\Omega;\mathbb{M}^N)\times H^{-\frac{1}{2}}(\Gamma_\e)\times H^1_0(\Omega_\e)$ and, hence, in view of Proposition~\ref{Prop 2.3f} , it is relatively compact with respect to the weak convergence in the sense of Definition \ref{Def 5.10}. For the rest of proof, it remains to apply Theorem~\ref{Th 1.8}.
\end{proof}

\begin{remark}
We note that variational properties of optimal solutions, given by Theorem~\ref{Th 5.31}, do not allows to say that the convergence of optimal states $P_\e(y_\e^0)$ to $y^0$ is strong in $H^1_0(\Omega)$. Indeed, the convergence
\begin{equation}
\label{5.35}
\int_{\Omega_\e} \big(\nabla y^0_\e,\left(A_\e^0\right)^{sym}\nabla y^0_{\e}\big)_{\mathbb{R}^N}\,dx \,\stackrel{\e\to 0}{\longrightarrow}\,
\int_{\Omega_\e} \big(\nabla y^0,\left(A^0\right)^{sym}\nabla y^0\big)_{\mathbb{R}^N}\,dx,
\end{equation}
which comes from \eqref{5.32}--\eqref{5.33}, does not implies the norm convergence in $H^1_0(\Omega)$. At the same time, combining relation \eqref{5.35} with energy identities
\[
 \int_{\Omega_\e} \big(\nabla y^0_\e,\left(A_\e^0\right)^{sym}\nabla y^0_{\e}\big)_{\mathbb{R}^N}\,dx
=\int_{\Omega_\e}f_\e y^0_\e\,dx
  +
 \left< v^0_\e,y^0_\e\right>_{H^{-\frac{1}{2}}(\Gamma_\e);H^{\frac{1}{2}}(\Gamma_\e)}
\]
and
\[
\int_\Omega\left(\nabla y^0, \left(A^0\right)^{sym}\nabla y^0\right)_{\mathbb{R}^N}\,dx
=-[y^0,y^0]_{A^0}+\left<f, y^0\right>_{H^{-1}(\Omega);H^1_0(\Omega)}
\]
rewritten for optimal solutions of the problems \eqref{5.13.1aa}--\eqref{5.13.2aa} and \eqref{2.1}--\eqref{2.2}, respectively, we get
\begin{equation}
\label{5.34}
\lim_{\e\to 0} \left<v^0_\e,y^0_{\e}\right>_{H^{-\frac{1}{2}}(\Gamma_\e);H^{\frac{1}{2}}(\Gamma_\e)}=-[y^0,y^0]_{A^0}.
\end{equation}
It gives us another example of the product of two weakly convergent sequences that can be recovered in the limit in an explicit form. Moreover, this limit does not coincide with the product of their weak limits.
\end{remark}

Our next remark deals with a motivation to put forward another concept of the weak solutions to the approximated boundary value problems \eqref{5.5b} and \eqref{4.0a} which can be viewed as a refinement of the integral identities \eqref{5.14} and \eqref{4.4}, respectively.
\begin{definition}
\label{Def 5.30.0a}
Let
$\left\{\Omega_\e\right\}_{\e>0}$ be a sequence of perforated subdomains of $\Omega$ associated with matrix $A$ by the rule \eqref{5.0a}--\eqref{5.0b}.
We say that a function $y_\e=y_\e(A,f,v)\in H^1_0(\Omega_\e)$ is  a weak
solution to the boundary value problem \eqref{5.5b} for
given $A\in \mathfrak{A}_{ad}$,  $f_\e\in
L^2(\Omega)$, and $v\in H^{-\frac{1}{2}}(\Gamma_\e)$, if
 the relation
\begin{multline}
\label{5.30.1a}
\int_\Omega \big(\nabla \varphi,A\nabla y_\e\big)_{\mathbb{R}^N}\chi_{\Omega_\e}\,dx + \int_{\Omega} \big(\nabla \psi,A\nabla h \big)_{\mathbb{R}^N}\,dx\\
-\int_\Omega f_\e\varphi \chi_{\Omega_\e}\,dx -\left< v,\varphi\right>_{H^{-\frac{1}{2}}(\Gamma_\e);H^{\frac{1}{2}}(\Gamma_\e)}=0.
\end{multline}
holds true
for all $h\in L(A)$, $\varphi\in C^\infty_0(\Omega)$, and $\psi\in C^\infty_0(\Omega)$.
\end{definition}

\begin{definition}
\label{Def 5.30.0}
Let $A_\e:=T_\e(A)\in L^\infty(\Omega;\mathbb{M}^N)$ be a truncation of a given matrix $A\in \mathfrak{A}_{ad}$.
We say that a function $y_\e=y_\e(A_\e,f)\in H^1_0(\Omega)$ is  a weak
solution to the boundary value problem \eqref{4.4} for
given  $f\in H^{-1}(\Omega)$, if
the relation
\begin{equation}
\label{5.30.1b}
\int_\Omega \big(\nabla \varphi,A_\e\nabla y_\e\big)_{\mathbb{R}^N}\,dx + \int_{\Omega} \big(\nabla \psi,A\nabla h \big)_{\mathbb{R}^N}\,dx
-\left<f,\varphi\right>_{H^{-1}(\Omega);H^1_0(\Omega)}=0.
\end{equation}
holds true for all $h\in L(A)$, $\varphi\in C^\infty_0(\Omega)$, and $\psi\in C^\infty_0(\Omega)$.
\end{definition}

Since for every $A\in \mathfrak{A}_{ad}$ and
$h\in D(A)$ the bilinear form $[h,\varphi]_A$ can be extended by continuity (see \eqref{2.8a}) onto the entire space $H^1_0(\Omega)$, it follows that the integral identities \eqref{5.30.1a}--\eqref{5.30.1b} can be rewritten as follows
\begin{align}
\notag
\int_\Omega \big(\nabla \varphi,&A^{sym}\nabla y_\e+A^{skew}\nabla y_\e\big)_{\mathbb{R}^N}\chi_{\Omega_\e}\,dx\\ \notag &+ \int_{\Omega} \big(\nabla \psi,A^{sym}\nabla h\big)_{\mathbb{R}^N}\,dx +[h,\psi]_A
-\int_\Omega f_\e\varphi \chi_{\Omega_\e}\,dx\\ &-\left< v,\varphi\right>_{H^{-\frac{1}{2}}(\Gamma_\e);H^{\frac{1}{2}}(\Gamma_\e)}=0\quad \forall\, \varphi,\psi\in H^1_0(\Omega),\ \forall\,h\in L(A) \label{5.30.3a}
\end{align}
and
\begin{align}
\notag
\int_\Omega \big(\nabla \varphi,&A^{sym}\nabla y_\e+A^{skew}_\e(x)\nabla y_\e\big)_{\mathbb{R}^N}\,dx\\ \notag &+ \int_{\Omega} \big(\nabla \psi,A^{sym}\nabla h \big)_{\mathbb{R}^N}\,dx + [ h,\psi]_A \\
&-\left<f,\varphi\right>_{H^{-1}(\Omega);H^1_0(\Omega)}=0\quad
\forall\, \varphi,\psi\in H^1_0(\Omega), \ \forall\,h\in L(A),
\label{5.30.3b}
\end{align}
respectively.

Hence, using the skew-symmetry property of the  matrix $A^{skew}\in L^2\big(\Omega;\mathbb{S}^N_{skew}\big)$ and the fact that the set $L(A)$ is closed with respect to the strong topology of $H^1_0(\Omega)$, we conclude: for every $\e>0$ there exist elements $h_\e^1, h^2_\e$ in $L(A)$ such that the relations \eqref{5.30.3a}--\eqref{5.30.3b} can be reduced to the following energy equalities
\begin{align}
\int_\Omega \left(\nabla y_\e, A^{sym}y_\e\right)_{\mathbb{R}^N}\chi_{\Omega_\e}\,dx &+
\int_{\Omega} \big(\nabla y_\e,A^{sym}\nabla h^1_\e\big)_{\mathbb{R}^N}\,dx +[h^1_\e,y_\e]_A\\
&=\int_\Omega f_\e y_\e \chi_{\Omega_\e}\,dx +\left< v,y_\e\right>_{H^{-\frac{1}{2}}(\Gamma_\e);H^{\frac{1}{2}}(\Gamma_\e)},
\label{5.30.4a}
\end{align}
\begin{align}
\int_\Omega \left(\nabla y_\e, A^{sym}y_\e\right)_{\mathbb{R}^N}\,dx &+ \int_{\Omega} \big(\nabla y_\e,A^{sym}\nabla h^2_\e \big)_{\mathbb{R}^N}\,dx + [ h^2_\e,y_\e]_A \\
&=\left<f,y_\e\right>_{H^{-1}(\Omega);H^1_0(\Omega)}
\label{5.30.4b}
\end{align}
for the problems \eqref{5.5b} and \eqref{4.0a}, respectively.

Thus, in contrast to the "typical"\ energy equalities to the boundary value problems \eqref{5.5b} and \eqref{4.0a}, relations \eqref{5.30.4a}--\eqref{5.30.4b} include some extra terms which coming from the singular energy of the boundary value problem \eqref{2.1}--\eqref{2.2} that was originally hidden in approximated problems \eqref{5.5b} and \eqref{4.0a}.
However, in contrast to the similar functional effect for Hardy inequalities in bounded domains (see \cite{Vaz}),  the terms $\int_{\Omega} \big(\nabla y_\e,A^{sym}\nabla h^i_\e \big)_{\mathbb{R}^N}\,dx + [ h^i_\e,y_\e]_A$ are additive to the total energy, and, hence, their influence may correspond to the increasing or decreasing of the total energy and may even constitute the main part of it.

\section{Optimality System for Regularized OCPs Associated with Perforated Domains $\Omega_\e$ and its Asymptotic Analysis}
\label{Sec 6}

As follows from Theorem~\ref{Th 5.6}, for each $\e>0$ small enough, the optimal control problem $\left<\inf_{(A,v,y)\in\Xi_\e}
I_\e(A,v,y)\right>$, where the cost functional $I_\e:\Xi_\e\rightarrow \mathbb{R}$ and its domain $\Xi_\e\subset \mathfrak{A}^\e_{ad}\times H^{-\frac{1}{2}}(\Gamma_\e)\times H^1_0(\Omega_\e;\partial\Omega)$ are defined by \eqref{5.5a}--\eqref{5.5b}, is a well-posed controllable system. Hence, to deduce an optimality system for this problem, we make use of the following well-know result.
\begin{theorem}[Ioffe and Tikhomirov \cite{IoTi,Fursik}]
\label{Th 6.1}
Let $Y$, $U$, and $V$ be Banach spaces, let $J:Y\times U\to \overline{\mathbb{R}}$ be a cost functional, let $F:Y\times U\to V$ be a mapping, and let $U_\partial$ be a convex subset of the space $U$ containing more than one point.
Let $(\widehat{u},\widehat{y})\in U\times Y$ be a solution to the problem
\begin{gather*}
J(u,y)\rightarrow\inf,\\
F(u,y)=0,\quad u\in U_\partial.
\end{gather*}
For each $u\in U_\partial$, let the mapping $y\mapsto J(u,y)$ and $y\mapsto F(u,y)$ be continuously differentiable for
$y\in \mathcal{O}(\widehat{y})$, where $\mathcal{O}(\widehat{y})$ is some neighbourhood of the point $\widehat{y}$, and let
$\mathrm{Im}\,F^\prime_y(\widehat{u},\widehat{y})$ be closed and it has a finite codimension in $V$. In addition, for $y\in \mathcal{O}(\widehat{y})$, let the function $u\mapsto J(u,y)$ be convex, the functional $J$ is G\^{a}teaus-differentiable with respect to $u$ at the point $(\widehat{u},\widehat{y})$, and the mapping $u\mapsto F(u,y)$ is continuous from $U$ to $Y$ and affine, i.e.,
\[
F(\gamma u_1+(1-\gamma) u_2,y)=\gamma F(u_1,y) + (1-\gamma) F(u_2,y),\quad\forall\,u_1,u_2 \in U, \gamma\in \mathbb{R}.
\]
Then there exists a pair $(\lambda,p)\in \left(R_{+}\times V^\ast\right)\setminus \{0\}$ such that
\begin{gather}
\label{6.1} \left<\mathcal{L}^\prime_y(\widehat{u},\widehat{y},\lambda,p),h\right>_{Y^\ast;Y}=0,\quad\forall\,h\in Y,\\
\label{6.2}
 \left<\mathcal{L}^\prime_u(\widehat{u},\widehat{y},\lambda,p),u\right>_{U^\ast;U}\ge 0,\quad\forall\,u\in U_\partial-\widehat{u},
\end{gather}
where the Lagrange functional $\mathcal{L}$ is defined by equality
\begin{equation}
\label{6.3}
\mathcal{L}(u,y,\lambda,p)=\lambda J(u,y)+\left<p,F(u,y)\right>_{V^\ast;V}.
\end{equation}
If $\mathrm{Im}\,F^\prime_y(\widehat{u},\widehat{y})=V$, then it can be assumed that $\lambda=1$ in \eqref{6.1}--\eqref{6.2}.
\end{theorem}

For our further analysis, we set
\begin{align}
\label{6.3a}
Y&=H^1_0(\Omega_\e;\partial\Omega),\quad V=L^2(\Omega_\e)\times H^{-\frac{1}{2}}(\Gamma_\e),\\
\label{6.3b}
U&=\big(L^2(\Omega;\mathbb{S}^N_{sym})\oplus
L^2(\Omega;\mathbb{S}^N_{skew})\big)\times H^{-\frac{1}{2}}(\Gamma_\e), \\
\label{6.3c}
U_\partial&=\mathfrak{A}_{ad}\times H^{-\frac{1}{2}}(\Gamma_\e):=\left(\mathfrak{A}_{ad,1}\oplus \mathfrak{A}_{ad,2}\right)\times H^{-\frac{1}{2}}(\Gamma_\e),
\end{align}
\begin{align}
\notag
J= I_\e(A,v,y)&:=\left\|y-y_d\right\|^2_{L^2(\Omega_\e)}
 + \int_{\Omega_\e}\left(\nabla y, A^{sym}\nabla y\right)_{\mathbb{R}^N}\,dx\\ &\quad + \frac{1}{\e^\sigma}\|v\|^2_{H^{-\frac{1}{2}}(\Gamma_\e)},\\
\label{6.3d}
F(A,v,y)&=\left(-\div\big(A\nabla y\big) - f_\e, \frac{\partial y}{ \partial \nu_{A}}-v\right).
\end{align}
Since for each $(g,w)\in L^2(\Omega_\e)\times H^{-\frac{1}{2}}(\Gamma_\e)$ the boundary value problem
\begin{gather}
\label{6.4.1}
-\div\big(A\nabla y\big) = g\quad\text{in }\ \Omega_\e,\\[1ex]
\label{6.4.2}
y=0\text{ on }\partial\Omega,\quad
\partial y/ \partial \nu_{A}=w\ \text{on }\Gamma_\e
\end{gather}
has a unique solution $y\in H^1_0(\Omega_\e;\partial\Omega)$ \cite{Lions_Mag:72}, we have $\mathrm{Im}\,F^\prime_y(\widehat{u},\widehat{y})=V$. Thus, the assumptions of Theorem~\ref{Th 6.1} are
obviously satisfied. It means that the Lagrange functional $\mathcal{L}_\e$  to the optimal control problem $\left<\inf_{(A,v,y)\in\Xi_\e}
I_\e(A,v,y)\right>$ can be defined by formula (with $\lambda=1$ in \eqref{6.1}--\eqref{6.2})
\begin{align}
\notag
\mathcal{L}_\e(A,v,y,p,p_1)=\ &\left\|y-y_d\right\|^2_{L^2(\Omega_\e)}
 + \int_{\Omega_\e}\left(\nabla y, A^{sym}\nabla y\right)_{\mathbb{R}^N}\,dx\\
\notag  &+ \frac{1}{\e^\sigma}\|v\|^2_{H^{-\frac{1}{2}}(\Gamma_\e)}
+ \left(-\div\big(A\nabla y\big) - f_\e, p_1\right)_{L^2(\Omega_\e)}\\
& +
\left< \frac{\partial y}{ \partial \nu_{A}}-v,p_{\,2}\,\right>_{H^{-\frac{1}{2}}(\Gamma_{\e});H^{\frac{1}{2}}(\Gamma_{\e})},
\label{6.4}
\end{align}
where $p=(p_1,p_{\,2})\in V^\ast:= L^2(\Omega_\e)\times H^{\frac{1}{2}}(\Gamma_{\e})$.

Let $\gamma^0_{\Gamma_\e}:H^1_0(\Omega_\e;\partial\Omega)\rightarrow H^{\frac{1}{2}}(\Gamma_{\e})$ be the trace operator, i.e.
$\gamma^0_{\Gamma_\e}$ is the extension by continuity of the restriction operator $\gamma^0_{\Gamma_\e}( u)=u\big|_{\Gamma_\e}$ given for all $u\in C_0^\infty(\mathbb{R}^N)$.

We are now in a position to prove the following result.
\begin{theorem}
\label{Th 6.5}
For a given $\e>0$, let
$$
(A_\e^0, v_\e^0, y_\e^0)\in \big(L^2(\Omega;\mathbb{S}^N_{sym})\oplus
L^2(\Omega;\mathbb{S}^N_{skew})\big)\times H^{-\frac{1}{2}}(\Gamma_\e)\times H^1_0(\Omega_\e;\partial\Omega)
$$
be an optimal solution to the regularized problems \eqref{5.4}--\eqref{5.5b}. Assume that the following condition holds true
\begin{equation}
\label{6.5.0}
\div\,\left(\left(A_\e^0\right)^{skew}\nabla y_\e^0\right) \in L^2(\Omega_\e).
\end{equation}
Then there exists an element $p_\e\in H^1_0(\Omega_\e;\partial\Omega)$ such that the tuple
$$(A_\e^0, v_\e^0, y_\e^0,p_\e,\gamma^0_{\Gamma_\e}(p_\e))$$ satisfies the following system of relations
\begin{align}
\label{6.5}
-\div\big(A_\e^0\nabla y_\e^0\big) =\ & f_\e\quad\text{in }\ \Omega_\e,\\
\label{6.5a}
y_\e^0=\ & 0\quad \text{ on }\partial\Omega,\\
\label{6.5b}
\partial y_\e^0/ \partial \nu_{A_\e^0}=\ & v_\e^0\quad \text{on }\Gamma_\e,\\
 \div\left(\left(A_\e^0\right)^t\nabla p_\e\right)=\ &- 2\,\div\left(\left(A_\e^0\right)^{sym}\nabla y^0_\e\right)+ 2\left(y^0_\e-y_d\right),
\   \text{a.e. in }\ \Omega_\e,\label{6.6}\\
\label{6.6a} p_\e=\ & 0\quad \text{ on }\partial\Omega,\\
\label{6.6b}
\partial p_\e^0/ \partial \nu_{(A_\e^0)^t}=\ & 0\quad \text{on }\Gamma_\e,\\
\label{6.7a}
v^0_\e =\ & \frac{\e^\sigma}{2}\Lambda_{H^{\frac{1}{2}}(\Gamma_{\e})}\gamma^0_{\Gamma_\e}(p_\e),
\end{align}
\begin{multline}
\label{6.7b}
\int_{\Omega_\e}\big(\nabla y^0_\e+\nabla p_\e, \left(A^{sym}-(A^0_\e)^{sym}\right)\nabla y^0_\e\big)_{\mathbb{R}^N}\,dx\\
+ \int_{\Omega_\e}\left(\nabla p_\e, \left(A^{skew}-(A^0_\e)^{skew}\right)\nabla y^0_\e\right)_{\mathbb{R}^N}\,dx
\ge 0,\quad\forall\, A\in \mathfrak{A}_{ad},
\end{multline}
where $\Lambda_{H^{\frac{1}{2}}(\Gamma_{\e})}$ is the canonical isomorphism of $H^{\frac{1}{2}}(\Gamma_{\e})$ onto $H^{-\frac{1}{2}}(\Gamma_{\e})$.
\end{theorem}
\begin{remark}
It is worth to notice that, in contrast to \eqref{6.5}, relation \eqref{6.6} should be interpreted as an equality of $L^2$-functions. It means that the description of boundary value problem \eqref{6.6}--\eqref{6.6b} in the sense of distributions takes other form, namely,
\begin{align*}
 \div\left(\left(A_\e^0\right)^t\nabla p_1\right)=\ & 2\left(f_\e+\,\div\left(\left(A_\e^0\right)^{skew}\nabla y^0_\e\right)+\left(y^0_\e-y_d\right)\right),\ \text{ in }\ \Omega_\e,\\
p_\e=\ & 0\quad \text{ on }\partial\Omega,\\
\partial p_\e^0/ \partial \nu_{(A_\e^0)^t}=\ & \partial y_\e^0/ \partial \nu_{(A_\e^0)^{skew}}\quad \text{on }\Gamma_\e,
\end{align*}
where the component $\partial y_\e^0/ \partial \nu_{(A_\e^0)^{skew}}$ is unknown a priori.
Here, we have used the fact that
\begin{equation}
\label{6.7d}
-\div\left(\left(A_\e^0\right)^{sym}\nabla y^0_\e\right)=f_\e+\div\left(\left(A_\e^0\right)^{skew}\nabla y^0_\e\right)\quad\text{in }\ \Omega_\e
\end{equation}
by equation \eqref{6.5}.
\end{remark}

\begin{proof}
By Theorem~\ref{Th 6.1}, there exists a pair $p=(p_1,p_{\,2})\in V^\ast:= L^2(\Omega_\e)\times H^{\frac{1}{2}}(\Gamma_{\e})$ such that the Lagrange functional $\mathcal{L}$ satisfies relations \eqref{6.1}--\eqref{6.2}. The direct computations show that, in view of \eqref{6.4}, the condition \eqref{6.1} takes the form
\begin{multline}
\label{6.8}
\left\langle\mathcal{D}_y\,\widehat{L}_\e(A_\e^0, v_\e^0, y_\e^0,p_1,p_2),h\right\rangle_{Y^\ast;Y}=
2\int_{\Omega_\e}\left(\nabla h, \left(A_\e^0\right)^{sym}\nabla y^0_\e\right)_{\mathbb{R}^N}\,dx\\
+2
\int_{\Omega_\e}\left(y^0_\e-y_d\right) h\,dx
 +
\left< \frac{\partial h}{ \partial \nu_{A_\e^0}},p_{\,2}\,\right>_{H^{-\frac{1}{2}}(\Gamma_{\e});H^{\frac{1}{2}}(\Gamma_{\e})}\\
 -\int_{\Omega_\e}
\div\big(A_\e^0\nabla h\big)p_1\,dx=0,\quad \forall\,h\in H^2(\Omega_\e)\cap H^1_0(\Omega_\e;\partial\Omega)
\end{multline}
(here we have used the fact that $\mathrm{Im}\,F^\prime_y(\widehat{u},\widehat{y})=V$). As follows from \eqref{6.8} and \eqref{6.5.0}, for $h\in C^\infty_0(\Omega_\e)$, we have
\begin{multline}
\label{6.9}
2\int_{\Omega_\e}\left(\nabla h, \left(A_\e^0\right)^{sym}\nabla y^0_\e\right)_{\mathbb{R}^N}\,dx
+2
\int_{\Omega_\e}\left(y^0_\e-y_d\right) h\,dx\\- \int_{\Omega_\e}\div\left(\left(A_\e^0\right)^t\nabla p_1\,\right) h\,dx=
-2\int_{\Omega_\e} \div\left(\left(A_\e^0\right)^{sym}\nabla y^0_\e\right) h\,dx\\
+2
\int_{\Omega_\e}\left(y^0_\e-y_d\right) h\,dx
 - \int_{\Omega_\e}\div\left(\left(A_\e^0\right)^t\nabla p_1\,\right) h\,dx=0.
\end{multline}
Due to equality \eqref{6.7d} and the initial assumptions \eqref{6.5.0}, relation \eqref{6.9} implies that $\div\left(\left(A_\e^0\right)^t\nabla p_1\right)\in L^2(\Omega_\e)$. Hence,  $\left(A_\e^0\right)^t\nabla p_1\in H(\Omega_\e;\div)$, where
\[
H(\Omega_\e;\div)=\left\{\xi\ |\ \xi\in L^2(\Omega_\e;\mathbb{R}^N),\  \div\xi\in L^2(\Omega_\e)\right\}.
\]
Thanks to Lipschitz properties of $\partial\Omega_\e$, we can conclude that (see, for instance, \cite{Lions_Mag:72,Cioran})
$\partial p_1/ \partial \nu_{(A_\e^0)^t}\in H^{-\frac{1}{2}}(\partial\Omega_{\e})$ and the map
\[
\left(A_\e^0\right)^t\nabla p_1\in H(\Omega_\e;\div)\ \mapsto \frac{\partial p_1}{ \partial \nu_{(A_\e^0)^t}}\in H^{-\frac{1}{2}}(\partial\Omega_{\e})
\]
is linear and continuous. Moreover, if $\left(A_\e^0\right)^t\nabla p_1\in H(\Omega_\e;\div)$ and $h\in H^2(\Omega_\e)\cap H^1_0(\Omega_\e;\partial\Omega)$, then  the Green formula
\begin{align}
\notag
-\int_{\Omega_\e}\div & \big(A_\e^0\nabla h\big) p_1\,dx=-\int_{\Omega_\e}\div \big(\left(A_\e^0\right)^t\nabla p_1\,\big) h\,dx\\
&-\left< \frac{\partial h}{ \partial \nu_{A_\e^0}},\gamma^0_{\partial\Omega_\e}\big(p_1\big)\,\right>_{H^{-\frac{1}{2}}(\Omega_{\e});H^{\frac{1}{2}}(\Omega_{\e})}
+ \left< \frac{\partial p_1}{ \partial \nu_{(A_\e^0)^t}},h\right>_{H^{-\frac{1}{2}}(\Gamma_{\e});H^{\frac{1}{2}}(\Gamma_{\e})}
\label{6.10}
\end{align}
is valid.  Then, combining this relation with \eqref{6.8}--\eqref{6.9}, we arrive at the following identity
\begin{multline}
\label{6.11}
\left\langle\mathcal{D}_y\,\widehat{L}_\e(A_\e^0, v_\e^0, y_\e^0,p_1,p_2),h\right\rangle_{Y^\ast;Y}=
\left< \frac{\partial p_1}{ \partial \nu_{(A_\e^0)^t}},h\,\right>_{H^{-\frac{1}{2}}(\Gamma_{\e});H^{\frac{1}{2}}(\Gamma_{\e})}\\ -
\left< \frac{\partial h}{ \partial \nu_{A_\e^0}},\gamma^0_{\partial\Omega_\e}\big(p_1\big)\,\right>_{H^{-\frac{1}{2}}(\Omega_{\e});H^{\frac{1}{2}}(\Omega_{\e})}+
\left< \frac{\partial h}{ \partial \nu_{A_\e^0}},p_{\,2}\,\right>_{H^{-\frac{1}{2}}(\Gamma_{\e});H^{\frac{1}{2}}(\Gamma_{\e})}=0,
\end{multline}
which is valid for all $h\in H^2(\Omega_\e)\cap H^1_0(\Omega_\e;\partial\Omega)$ and all $p=(p_1,p_2)$ such that
\begin{equation}
\label{6.12}
\begin{array}{c}
p_1\ \text{satisfies \eqref{6.9}},\\[1ex]
(p_1,p_{\,2})\in L^2(\Omega_\e)\times H^{\frac{1}{2}}(\Gamma_{\e})\ \text{ and }\
\left(A_\e^0\right)^t\nabla p_1\in H(\Omega_\e,\mathrm{div}).
\end{array}
\end{equation}

As follows from \eqref{6.11}, for each
$$
h\in C^\infty_0(\mathbb{R}^N;\Gamma_\e)\cap C_0(\mathbb{R}^N;\partial\Omega)\subset H^2(\Omega_\e)\cap H^1_0(\Omega_\e;\partial\Omega),
$$
we have
\[
\left< \frac{\partial h}{ \partial \nu_{A_\e^0}},\gamma^0_{\partial\Omega}\big(p_1\big)\,\right>_{H^{-\frac{1}{2}}(\partial \Omega);H^{\frac{1}{2}}(\partial \Omega)}=0.
\]
Since $C^\infty_0(\mathbb{R}^N;\Gamma_\e)\cap C_0(\mathbb{R}^N;\partial\Omega)$ is dense in $H^{-\frac{1}{2}}(\partial \Omega)$ and
the matrix $\left(A_\e^0\right)^{sym}$ is positive defined, it follows that
\begin{equation}
\label{6.13}
\gamma^0_{\partial\Omega}\big(p_1\big)=0.
\end{equation}
Hence, equality \eqref{6.11}, for all $h\in C^\infty_0(\mathbb{R}^N;\Gamma_\e)$, gives
\begin{equation}
\label{6.14}
\left< \frac{\partial h}{ \partial \nu_{A_\e^0}},p_{\,2}\,\right>_{H^{-\frac{1}{2}}(\Gamma_{\e});H^{\frac{1}{2}}(\Gamma_{\e})}-
\left< \frac{\partial h}{ \partial \nu_{A_\e^0}},\gamma^0_{\Gamma_\e}\big(p_1\big)\,\right>_{H^{-\frac{1}{2}}(\Gamma_{\e});H^{\frac{1}{2}}(\Gamma_{\e})}=0.
\end{equation}
Taking into account the fact that the mapping
$$\partial/ \partial \nu_{A_\e^0}: H^2(\Omega_\e)\cap H^1_0(\Omega_\e;\partial\Omega)\rightarrow H^{\frac{1}{2}}(\Gamma_{\e})$$
is an epimorphism (see Theorem~1.1.4 in \cite{Fursik}), from \eqref{6.14} it follows that
\begin{equation}
\label{6.15}
\gamma^0_{\Gamma_\e}\big(p_1\big)=p_{\,2}.
\end{equation}

Thus, in view of \eqref{6.13} and \eqref{6.15}, relation \eqref{6.11} takes the form
\[
\left\langle\mathcal{D}_y\,\widehat{L}_\e(A_\e^0, v_\e^0, y_\e^0,p_1,\gamma^0_{\Gamma_\e}\big(p_{1}\big)),h\right\rangle_{Y^\ast;Y}=
\left< \frac{\partial p_1}{ \partial \nu_{(A_\e^0)^t}},h\,\right>_{H^{-\frac{1}{2}}(\Gamma_{\e});H^{\frac{1}{2}}(\Gamma_{\e})}=0
\]
for all $h\in H^2(\Omega_\e)\cap H^1_0(\Omega_\e;\partial\Omega)$.
Applying the same arguments as before, we finally conclude that
\begin{equation}
\label{6.16}
\frac{\partial p_1}{ \partial \nu_{(A_\e^0)^t}}=0\quad\text{on }\ \Gamma_\e\ \text{ (in the sense of distribution)}.
\end{equation}

As a result, having gathered relations \eqref{6.9}, \eqref{6.13}, and \eqref{6.16}, we arrive at the boundary value problem
\eqref{6.6}--\eqref{6.6b}. Moreover, by the regularity of solutions to the problem \eqref{6.6}--\eqref{6.6b}, we have
$p_\e\in H^2(\Omega_\e)\cap H^1_0(\Omega_\e;\partial\Omega)$ \cite{Gilbarg}.

In order to end of the proof of this theorem, it remains to show the validity of the relations \eqref{6.7a}--\eqref{6.7b}.
With that in mind, we note that, in view of the structure \eqref{6.3a}--\eqref{6.3c}, condition \eqref{6.2} takes the form
\begin{multline}
\left(\mathcal{D}_A \mathcal{L}(A_\e^0, v_\e^0, y_\e^0,p_\e,\gamma^0_{\Gamma_\e}(p_\e)),A-A_\e^0\right)_{L^2(\Omega;\mathbb{M}^N)}\ge 0,\quad\forall\, A\in \mathfrak{A}^\e_{ad}\quad\Longrightarrow\\
\int_{\Omega_\e}\left(\nabla y^0_\e+\nabla p_\e, \left(A^{sym}-(A^0_\e)^{sym}\right)\nabla y^0_\e\right)_{\mathbb{R}^N}\,dx\\
+ \int_{\Omega_\e}\left(\nabla p_\e, \left(A^{skew}-(A^0_\e)^{skew}\right)\nabla y^0_\e\right)_{\mathbb{R}^N}\,dx
\ge 0,\quad\forall\, A\in \mathfrak{A}^\e_{ad},
\label{6.17}
\end{multline}
\begin{equation}
\label{6.18}
\mathcal{D}_v \mathcal{L}(A_\e^0, v_\e^0, y_\e^0,p_\e,\gamma^0_{\Gamma_\e}(p_\e))=0\quad\Longrightarrow
\frac{2}{\e^\sigma} v^0_\e -\Lambda_{H^{\frac{1}{2}}(\Gamma_{\e})}\gamma^0_{\Gamma_\e}(p_\e)=0,
\end{equation}
 Here, we have used the fact that $H^{\frac{1}{2}}(\Gamma_{\e})$ can be reduced to a Hilbert space with respect to an appropriate equivalent norm, and, hence, $H^{-\frac{1}{2}}(\Gamma_{\e})$ is a dual Hilbert space as well (for the details we refer to Lions and Magenes \cite[p.35]{Lions_Mag:72}).
\end{proof}
\begin{remark}
\label{Rem 6.4}
In view of the assumption \eqref{6.5.0}, we make use of the following observation.
Let $\left\{(A_\e,v_\e,y_\e)\in  \Xi_\e\right\}_{\e>0}$ be a weakly convergent sequence in the sense of Definition~\ref{Def 5.10}. Since in this case $\left\{y_\e\in H^1_0(\Omega_\e;\partial\Omega)\right\}_{\e>0}$ are the solutions to the boundary value problem \eqref{6.4.1}--\eqref{6.4.2} with $A=A_\e$, and $g=f_\e\in L^2(\Omega)$, and $w=v_\e\in H^{-\frac{1}{2}}(\Gamma_\e)$, it follows that
the sequence $\left\{\div\big(A_\e\nabla y_\e\big)\chi_{\Omega_\e}\right\}_{\e>0}$ is obviously bounded in $L^2(\Omega)$. However, because of the non-symmetry of $L^2$-matrices $\left\{A_\e\right\}_{\e>0}$, it does not imply the same property for the sequence $\left\{\div\big(A^{skew}_\e\nabla y_\e\big)\chi_{\Omega_\e}\right\}_{\e>0}$. In order to guarantee this property, we make use of the notion of divergence $\div A$ of a skew-symmetric matrix $A\in L^2\big(\Omega;\mathbb{S}^N_{skew}\big)$. We define it as a vector-valued distribution
$d\in H^{-1}(\Omega;\mathbb{R}^N)$ following the rule
\begin{equation}
\label{6.0a}
\left<d_i,\varphi\right>_{H^{-1}(\Omega);H^1_0(\Omega)}= -\int_\Omega
(a_i,\nabla\varphi)_{\mathbb{R}^N}\,dx,\  \forall\,\varphi\in
C^\infty_0(\Omega),\ \forall\,i\in\left\{1,\dots,N\right\},
\end{equation}
where $a_i$ stands for the $i$-th column of the matrix $A$. As a result, we can give the following conclusion:
if $\div A_\e^{skew}\in L^\infty(\Omega;\mathbb{R}^N)$ for all $\e>0$ and the sequence $\left\{\div A_\e^{skew}\right\}_{\e>0}$ is uniformly bounded in $L^\infty(\Omega;\mathbb{R}^N)$, then there exists a constant $C>0$ independent of $\e$ such that
\begin{equation}
\label{6.0b}
\sup_{\e>0} \left\|\chi_{\Omega_\e} \div\big(A^{skew}_\e\nabla y_\e\big)\right\|_{L^2(\Omega)}\le C.
\end{equation}
Indeed, since
\begin{align*}
-\big<\mathrm{div}\,\left(A_\e^{skew}\nabla \psi_\e\right),&\chi_{\Omega_\e}\varphi\big>_{H^{-1}(\Omega);H^1_0(\Omega)}\\
=\ &
-\left<\mathrm{div}\,\left(A_\e^{skew}\nabla \psi_\e\right),\varphi\right>_{H^{-1}(\Omega_\e);H^1_0(\Omega_\e)}\\
=\ &\Big< \mathrm{div}\,\left[
\begin{array}{c}
a^t_{1,\e}\nabla \psi_\e\\ \cdots \\a^t_{N,\e}\nabla \psi_\e
\end{array}\right],\varphi\Big>_{H^{-1}(\Omega_\e);H^1_0(\Omega_\e)}\\
=\ &
\sum_{i=1}^N\left<\mathrm{div}\,a_{i,\e}, \varphi\frac{\partial\psi_\e}{\partial x_i}\right>_{H^{-1}(\Omega_\e);H^1_0(\Omega_\e)}\\
&+ \underbrace{\int_{\Omega_\e}
\sum_{i=1}^N \sum_{j=1}^N \left(a_{ij,\e} \frac{\partial^2 \psi_\e}{\partial x_i\partial x_j}\right) \varphi\,dx}_{=0\ \atop {\text{ since }\ A_\e^{skew}\in L^2(\Omega;\mathbb{S}^N_{skew})}}\\
(\text{due to the fact that }&\text{$\div A_\e^{skew}\in L^\infty(\Omega;\mathbb{R}^N)$ for all $\e>0$})
\\
=\ &\int_{\Omega_\e} \left(\,\div A_\e^{skew},\nabla \psi_\e\right)_{\mathbb{R}^N} \varphi\,dx,
\end{align*}
for any $\psi_\e,\varphi\in C^\infty_0(\Omega_\e)$, it follows that this relation can be extended by continuity to the following one
\[
-\left<\mathrm{div}\,\left(A_\e^{skew}\nabla y_\e\right),\chi_{\Omega_\e}\varphi\right>_{H^{-1}(\Omega);H^1_0(\Omega)}
=\int_{\Omega_\e} \left(\,\div A_\e^{skew},\nabla y_\e\right)_{\mathbb{R}^N} \varphi\,dx.
\]
Hence
\begin{align*}
\left\|\chi_{\Omega_\e} \div\big(A^{skew}_\e\nabla y_\e\big)\right\|_{L^2(\Omega)}&\le
(\mathcal{L}^N(\Omega))^{1/2}\|\div A_\e^{skew}\|_{L^\infty(\Omega;\mathbb{R}^N)}\\
&\times \|\nabla y_\e\|_{L^2(\Omega_\e;\mathbb{R}^N)}<+\infty.
\end{align*}
To deduce the estimate \eqref{6.0b}, it remains to refer to the boundedness of $y_\e$ in variable $H^1(\Omega_\e;\partial\Omega)$ (see Definition~\ref{Def 5.10}).
\end{remark}

Our next intention is to provide an asymptotic analysis of the optimality system \eqref{6.5}--\eqref{6.7b} as $\e$ tends to zero. With that in mind, we assume the fulfilment of the following Hypotheses:
\begin{enumerate}
\item[(AH1)] For each admissible control $A\in \mathfrak{A}_{ad}$  the corresponding bilinear form $[y,\varphi]_A$ is continuous in the following sense:
\begin{equation}
\label{6.19.a}
\lim_{\e\to 0} [y_\e,p_\e]_A = [y,p\,]_A
\end{equation}
provided
$\left\{p_\e\right\}_{\e>0}\subset H^1_0(\Omega)$, $\left\{y_\e\right\}_{\e>0}\subset H^1_0(\Omega)$, $y_\e\rightharpoonup y$ in $H^1_0(\Omega)$, $p_\e\rightarrow p$ in $H^1_0(\Omega)$, and  $y, y_\e\in D(A)$ for $\e>0$ small enough.
\item[(AH2)] Let $\left\{(A_\e^0, v_\e^0, y_\e^0, p_\e\,)\right\}_{e>0} $ be a sequence of tuples such that, for each $\e>0$ the corresponding cortege $(A_\e^0, v_\e^0, y_\e^0, p_\e\,)$ satisfies the optimality system \eqref{6.5}--\eqref{6.7b}. Then there exists a sequence of extension operators
    $$\left\{P_\e\in \mathcal{L}\left(H^1_0(\Omega_\e;\partial\Omega),H^1_0(\Omega)\right)\right\}_{\e>0}$$
    and element $\overline{\psi}\in H^1_0(\Omega)$ such that
\[
P_\e(p_\e) \rightarrow \overline{\psi}\quad\text{strongly in }\ H^1_0(\Omega)\quad\text{and}\quad \overline{\psi}\in D(A^\ast).
\]
\end{enumerate}

\begin{theorem}
\label{Th 6.19} Let $y_d\in L^2(\Omega)$ and $f\in
H^{-1}(\Omega)$ be given distributions. Let $A^\ast\in L^2\big(\Omega;\mathbb{S}^N_{skew}\big)$ be a matrix of the $\mathfrak{F}$-type. Let $\left\{(A^0_\varepsilon,v^0_\e,y^0_\e)\in \Xi_\varepsilon\right\}_{\e>0}$ be a sequence of optimal solutions to regularized problems \eqref{5.4}--\eqref{5.5b}, and let $(A^0,y^0)\in D(A^\ast)\times H^1_0(\Omega)$ be its $w$-limit. Let $\left\{p^0_\e\in H^1_0(\Omega_\e;\partial\Omega)\right\}_{\e>0}$ be a sequence of corresponding adjoint states. Then, the fulfilment of the Hypotheses (H1)--(H2) and (AH1)--(AH2) implies that $(A^0,y^0)\in \mathfrak{A}_{ad}\times H^1_0(\Omega)$ is an optimal pair to the original OCP \eqref{2.3}--\eqref{2.3a} and there exists an element $\overline{\psi}\in H^1_0(\Omega)$ such that
\begin{gather}
\label{6.20}
(A^0_\varepsilon,v^0_\e,y^0_\e)\,\stackrel{w}{\rightarrow}\, (A^0,y^0)\ \text{ as }\ \e\to 0,\\
\label{6.21}
P_\e(p_\e) \rightarrow \overline{\psi}\quad\text{strongly in }\ H^1_0(\Omega),\\
\label{6.22}
\begin{split}
-\div\big(A^0\nabla y^0\big) &=\  f\quad\text{in }\ \Omega,\\
y&=0\quad\text{on }\ \partial\Omega,
\end{split}
\\
\begin{split}
\div\left(\left(A^0\right)^t\nabla \overline{\psi}\right)&=\ - 2\,\div\left(\left(A^0\right)^{sym}\nabla y^0\right)+ 2\left(y^0-y_d\right)
\   \text{ in }\ \Omega,\\
\overline{\psi}&=0\quad\text{on }\ \partial\Omega,
\end{split}
\label{6.23}\\
\int_\Omega \big(\nabla y^0,\big(A^{sym}-\left(A^0\right)^{sym}\big)\left(\nabla y^0+\nabla\overline{\psi}\right)\big)_{\mathbb{R}^N}\,dx\notag\\
\ge\   [y^0,\overline{\psi}]_{A^0}-[y^0,\overline{\psi}]_{A},\ \forall A\in \mathfrak{A}_{ad},
\label{6.24}
\end{gather}
\end{theorem}
\begin{proof}
To begin with, we note that due to Theorem~\ref{Th 5.31}, the sequence of optimal solutions $\left\{(A^0_\varepsilon,v^0_\e,y^0_\e)\in \Xi_\varepsilon\right\}_{\e>0}$ to the regularized problems \eqref{5.4}--\eqref{5.5b}
is compact with respect to $w$-convergence and each of its $w$-cluster pairs $(A^0,y^0)$ is an optimal pair to the original problem \eqref{2.3}--\eqref{2.3a}. Hence, $(A^0,y^0)\in \mathfrak{A}_{ad}$, and the limit passage in \eqref{6.5}--\eqref{6.5b}
as $\e\to 0$ leads us to the integral identity \eqref{33.26}. Thus, the relation \eqref{6.22} holds true in the sense of distributions. In what follows, we divide the proof onto several steps.

Step~1. Since the integral identity
\begin{align}
\notag
\int_{\Omega} \big(\nabla \varphi,&\left(A_\e^0\right)^{sym}\nabla P_\e(p_\e)-\left(A_\e^0\right)^{skew}\nabla P_\e(p_\e)\big)_{\mathbb{R}^N}\chi_{\Omega_{\e}}\,dx\\ \notag
&=-2\int_{\Omega}\left(\nabla \varphi, \left(A_\e^0\right)^{sym}\nabla P_\e(y^0_\e)\right)_{\mathbb{R}^N}\chi_{\Omega_{\e}}\,dx\\
&-2
\int_{\Omega}\left(P_\e(y^0_\e)-y_d\right) \varphi\chi_{\Omega_{\e}}\,dx,\quad
 \forall\,\varphi\in C^\infty_0(\Omega)
\label{6.25}
\end{align}
holds true for every $\e>0$, we can pass to the limit in \eqref{6.25} as $\e\to 0$ due to Hypothesis (H3) and Definition \ref{Def 5.10} (here, we apply the arguments of Remark~\ref{Rem 5.3.1}). Using the strong convergence $\chi_{\Omega_{\e}}\rightarrow \chi_\Omega$ in $L^2(\Omega)$ (see Proposition~\ref{Prop 5.3}), we arrive at the equality
\begin{align}
\notag
\int_{\Omega} \big(\nabla \varphi,&\left(A^0\right)^{t}\nabla \overline{\psi}\big)_{\mathbb{R}^N}\,dx
=-2\int_{\Omega}\left(\nabla \varphi, \left(A^0\right)^{sym}\nabla y^0\right)_{\mathbb{R}^N}\,dx\\
&-2
\int_{\Omega}\left(y^0-y_d\right) \varphi\,dx,\quad
 \forall\,\varphi\in C^\infty_0(\Omega).
\label{6.25.a}
\end{align}
Hence, $\overline{\psi}\in D(A^0)\subset H^1_0(\Omega)$ (see Proposition~\ref{Prop 2.9}) and $\overline{\psi}$ satisfies relation \eqref{6.23} in the sense of distributions.

Step~2. On this step we study the limit passage in inequality \eqref{6.7b} as $\e\to 0$. To this end, we rewrite it as follows
\begin{equation}
\label{6.26}
J_1^\e(A)\ge J_2^\e - J_3^\e(A),\quad\forall\, A\in \mathfrak{A}^\e_{ad},\ \forall\,\e>0,
\end{equation}
where
\begin{align}
\label{6.27}
J_1^\e(A)=\ & \int_{\Omega_\e}\big(\nabla y^0_\e, A^{sym}\nabla y^0_\e\big)_{\mathbb{R}^N}\,dx,\\
\label{6.28}
J_2^\e=\ & \int_{\Omega_\e}\big(\nabla y^0_\e, (A^0_\e)^{sym}\nabla y^0_\e\big)_{\mathbb{R}^N}\,dx,\\
\label{6.29}
J_3^\e(A)=\ & \int_{\Omega_\e}\left(\nabla y^0_\e, \left(A^{t}-(A^0_\e)^{t}\right)\nabla p_\e\right)_{\mathbb{R}^N}\,dx.
\end{align}
By Theorem~\ref{Th 5.31} (see \eqref{5.33}), we have
\begin{align}
\notag
I\left(A^0,y^0\right):=\ & \left\|y^0-y_d\right\|^2_{L^2(\Omega)} + \int_{\Omega}\left(\nabla y^0, \left(A^0\right)^{sym}\nabla y^0\right)_{\mathbb{R}^N}\,dx\\
\notag =\ &\lim_{\e\to 0} I_{\e}(A^0_\e,v^0_\e,y^0_\e):= \lim_{\e\to 0} \left\|(y_\e^0-y_d)\chi_{\Omega_\e}\right\|^2_{L^2(\Omega)}\\
&+ \lim_{\e\to 0} \int_{\Omega_\e}\left(\nabla y_\e^0, \left(A_\e^0\right)^{sym}\nabla y_\e^0\right)_{\mathbb{R}^N}\,dx
+ \lim_{\e\to 0} \frac{1}{\e^\sigma}\|v_\e^0\|^2_{H^{-\frac{1}{2}}(\Gamma_\e)}.
\label{6.30}
\end{align}
Since
\begin{equation}
\label{6.30.aa}
\lim_{\e\to 0} \left\|(y_\e^0-y_d)\chi_{\Omega_\e}\right\|^2_{L^2(\Omega)}=\left\|y^0-y_d\right\|^2_{L^2(\Omega)}
\end{equation}
by the compactness of the embedding $H^1_0(\Omega)\hookrightarrow L^2(\Omega)$, and $\lim_{\e\to 0} \e^{-\sigma}\|v_\e^0\|^2_{H^{-\frac{1}{2}}(\Gamma_\e)}=0$ by Theorem~\ref{Th 5.31} (see estimate \eqref{5.30.3}), it follows from \eqref{6.30} that
\begin{equation}
\label{6.31}
\lim_{\e\to 0}J_2^\e = \int_{\Omega}\left(\nabla y^0, \left(A^0\right)^{sym}\nabla y^0\right)_{\mathbb{R}^N}\,dx=:J_2.
\end{equation}

Step~3. As for the term $J_3^\e(A)$, we see that
\begin{align}
\notag
\lim_{\e\to 0}J_3^\e(A)=\ &\lim_{\e\to 0} \int_{\Omega_\e}\left(\nabla y^0_\e, (A^0_\e)^{t}\nabla p_\e\right)_{\mathbb{R}^N}\,dx=(\ \text{by \eqref{6.25}}\ )\\
\notag=\ &\lim_{\e\to 0}\Big[-2\int_{\Omega}\left(\nabla P_\e(y^0_\e), \left(A_\e^0\right)^{sym}\nabla P_\e(y^0_\e)\right)_{\mathbb{R}^N}\chi_{\Omega_{\e}}\,dx\\
\notag&-2\int_{\Omega}\left(P_\e(y^0_\e)-y_d\right) P_\e(y^0_\e)\chi_{\Omega_{\e}}\,dx\Big] =
(\ \text{by \eqref{6.31} and \eqref{6.30.aa}}\ )\\
\notag=\ &
-2\int_{\Omega}\left(\nabla y^0, \left(A^0\right)^{sym}\nabla y^0\right)_{\mathbb{R}^N}\,dx
-2  \int_{\Omega}\left(y^0-y_d\right) y^0\,dx\\
\notag=\ &\lim_{\e\to 0}\Big[-2\int_{\Omega}\left(\nabla P_\e(y^0_\e), \left(A^0\right)^{sym}\nabla y^0\right)_{\mathbb{R}^N}\chi_{\Omega_{\e}}\,dx\\
\notag&-2  \int_{\Omega}\left(y^0-y_d\right) P_\e(y^0_\e)\chi_{\Omega_{\e}}\,dx\Big]=(\ \text{by \eqref{6.25.a}}\ )\\
\notag=\ &
\lim_{\e\to 0} \int_{\Omega} \big(\nabla P_\e(y^0_\e),\left(A^0\right)^{t}\nabla \overline{\psi}\big)_{\mathbb{R}^N}\chi_{\Omega_{\e}}\,dx\\
\notag=\ & \int_{\Omega} \big(y^0,\left(A^0\right)^{sym}\nabla \overline{\psi}\big)_{\mathbb{R}^N}\,dx +\lim_{\e\to 0} [P_\e(y^0_\e)\chi_{\Omega_{\e}},\overline{\psi}\,]_{A^0}=(\text{by (AH2)})\\
\label{6.32}=\ &
\int_{\Omega} \big(y^0,\left(A^0\right)^{sym}\nabla \overline{\psi}\big)_{\mathbb{R}^N}\,dx + [y^0,\overline{\psi}\,]_{A^0}
\end{align}
and
\begin{align}
\notag
\lim_{\e\to 0} \int_{\Omega_\e}\left(\nabla y^0_\e, A^{t}\nabla p_\e\right)_{\mathbb{R}^N}\,dx&=
\int_{\Omega}\left(\nabla y^0, A^{sym}\nabla \overline{\psi}\right)_{\mathbb{R}^N}\,dx\\
&+\lim_{\e\to 0} \int_{\Omega}\left(\nabla P_\e(p_\e), A^{skew}\nabla P_\e(y^0_\e)\right)_{\mathbb{R}^N}\chi_{\Omega_\e}\,dx
\label{6.33}
\end{align}
as the limit of product of weakly and strongly convergence sequences in $L^2(\Omega;\mathbb{R}^N)$. Hence, combining relations \eqref{6.32} and \eqref{6.33}, we get
\begin{align}
\notag
\lim_{\e\to 0}J_3^\e(A) =\ &\int_\Omega \big(y^0,\left(A^{sym}-\left(A^0\right)^{sym}\right)\nabla \overline{\psi}\big)_{\mathbb{R}^N}\,dx - [y^0,\overline{\psi}\,]_{A^0}\\
\notag
&+ \lim_{\e\to 0} \int_{\Omega}\left(\nabla P_\e(p_\e), A^{skew}\nabla P_\e(y^0_\e)\right)_{\mathbb{R}^N}\chi_{\Omega_\e}\,dx=(\text{by Hypothesis (AH2)})\\
=\ &\int_\Omega \big(y^0,\left(A^{sym}-\left(A^0\right)^{sym}\right)\nabla \overline{\psi}\big)_{\mathbb{R}^N}\,dx - [y^0,\overline{\psi}\,]_{A^0}+[y^0,\overline{\psi}\,]_{A}
\notag\\=:\ &J_3(A).
\label{6.34}
\end{align}

Step~4. At this step we study the asymptotic behaviour of the term $J_1^\e(A)$ in \eqref{6.27} as $\e\to 0$.
To this end, we note that in view of the property \eqref{1.1},  the lower semicontinuity of $L^2$-norm with respect to the weak convergence, immediately leads us to the inequality
\begin{align}
\notag
\lim_{\e\to 0}J_1^\e(A)=\ &\liminf_{\e\to 0} \int_{\Omega_{\e}}\left(\nabla y_\e^0, A^{sym} \nabla y_\e^0\right)_{\mathbb{R}^N}\,dx\\
\notag
=\ &
\liminf_{\e\to 0} \|\chi_{\Omega_{\e}}\left(A^{sym}\right)^{1/2}\nabla y_\e^0\|^2_{L^2(\Omega;\mathbb{R}^N)}\\
\ge\ & \|\left(A^{sym}\right)^{1/2}\nabla y^0\|^2_{L^2(\Omega;\mathbb{R}^N)}=
\int_{\Omega}\left(\nabla y^0, A^{sym} \nabla y^0\right)_{\mathbb{R}^N}\,dx
\notag\\
=:\ & J_1(A).
\label{6.35}
\end{align}
However, because of inequality in \eqref{6.35}, we cannot assert that the limit values are related as follows
\begin{equation}
\label{6.36}
J_1(A)\ge J_2 - J_3(A),\quad\forall\, A\in \mathfrak{A}_{ad}.
\end{equation}
In order to guarantee this relation, we assume the converse, namely, there exists a matrix $A_\sharp\in \mathfrak{A}_{ad}$ such that
$J_1(A_\sharp)< J_2 - J_3(A_\sharp)$. That is, in view of \eqref{6.31},\eqref{6.34}, and \eqref{6.35}, this leads us to the relation
\begin{multline}
\int_{\Omega}\left(\nabla y^0, \left(A_\sharp^{sym} - \left(A^0\right)^{sym}\right)\nabla y^0\right)_{\mathbb{R}^N}\,dx \\
+\int_\Omega \big(y^0,\left(A_\sharp^{sym}-\left(A^0\right)^{sym}\right)\nabla \overline{\psi}\big)_{\mathbb{R}^N}\,dx< [y^0,\overline{\psi}\,]_{A^0}-[y^0,\overline{\psi}\,]_{A_\sharp}.
\label{6.37}
\end{multline}
The direct computations show that, in this case, we arrive at the inequality
\[
\widehat{L}(A_\sharp,y^0,1,\overline{\psi})< \widehat{L}(A^0,y^0,1,\overline{\psi})=I(A_0,y_0)=\inf_{(A,y)\in\Xi}I(A,y),
\]
where $\widehat{L}(A,y,\lambda,p)$ is the Lagrange function given by \eqref{33.1}. However, this contradicts with the Lagrange principle, and therefore, the inequality \eqref{6.36} remains valid. Thus, following \eqref{6.36}, we finally get
\begin{equation*}
\int_{\Omega}\left(\nabla y^0, \left(A^{sym} - \left(A^0\right)^{sym}\right)(\nabla y^0+\nabla \overline{\psi})\right)_{\mathbb{R}^N}\,dx
\ge [y^0,\overline{\psi}\,]_{A^0}-[y^0,\overline{\psi}\,]_{A}
\end{equation*}
for all $A\in \mathfrak{A}_{ad}$. This concludes the proof.
\end{proof}

\begin{remark}
As Theorem~\ref{Th 6.19} indicates, the limit passage in optimality system \eqref{6.5}--\eqref{6.7b} for the regularized problems \eqref{5.4}--\eqref{5.5b} as $\e\to 0$ leads to the relations which coincide with the optimality system for the original OCP \eqref{2.3}--\eqref{2.3a}. However, a strict substantiation of this passage requires rather strong assumptions in the form of Hypotheses (H1)--H2) and (AH1)--(AH2). At the same time, the verification of these Hypotheses becomes trivial provided
\begin{gather}
\label{6.38}
A^\ast\in L^\infty(\Omega;\mathbb{S}^N_{skew})\quad\text{in  \eqref{2.3c}},\\
\label{6.39}
\text{and }\ \exists\,C>0\ :\ \|\div A^{skew}\|_{L^\infty(\Omega;\mathbb{R}^N)}\le C,\quad\forall\,A\in \mathfrak{A}_{ad}.
\end{gather}
Indeed, the validity of Hypotheses (H1)--(H2) evidently follows from \eqref{6.38}. Moreover, in this case the relation \eqref{6.19.a} takes the form
\[
\lim_{\e\to 0} \int_{\Omega}\left(\nabla p_\e, A^{skew}\nabla y_\e\right)_{\mathbb{R}^N}\,dx= \int_{\Omega}\left(\nabla p\,, A^{skew}\nabla y\right)_{\mathbb{R}^N}\,dx
\]
and it holds obviously true provided
$y_\e\rightharpoonup y$ in $H^1_0(\Omega)$, $p_\e\rightarrow p$ in $H^1_0(\Omega)$, and
$A^{skew}\preceq A^\ast\in L^\infty(\Omega;\mathbb{S}^N_{skew})$. Hence, Hypothesis~(AH1) is valid as well. As for Hypothesis~(AH2), we see that  admissible controls $A\in \mathfrak{A}_{ad}$ with extra property \eqref{6.39} form a close set with respect to the strong convergence in $L^2(\Omega;\mathbb{S}^N_{skew})$. Moreover, in this case we have that the sequence
$\left\{\chi_{\Omega_\e}\div\,\left(\left(A_\e^0\right)^{skew}\nabla y_\e^0\right)\right\}_{\e>0}$ is uniformly bounded in $L^2(\Omega)$ (see Remark~\ref{Rem 6.4}). Hence, the sequence of adjoint states $\left\{p_\e\right\}_{\e>0}$, given by \eqref{6.6}--\eqref{6.6b}, is bounded in $H^2(\Omega_\e)$ by the regularity of solutions to the problem \eqref{6.6}--\eqref{6.6b}.
Hence, within a subsequence, we can suppose that the sequence $\left\{P_\e(p_\e)\right\}_{\e>0}$ is weakly convergent in $H^2(\Omega)$. This proves Hypothesis (AH2).
\end{remark}

\section*{Acknowledgments} The authors gratefully acknowledge the support of le Conservatoire National des Arts et M\'{e}tiers (Paris, France)  and the support of the French ANR Project CISIFS.


\medskip
\medskip

\end{document}